\documentclass[a4paper, 11pt, headings = small, abstract]{scrartcl}
\usepackage[english]{babel}

\usepackage{amscd,amsmath,amsthm,array,bbm,hhline,mathrsfs, enumerate,dsfont, booktabs,fancybox,calc,textcomp,xcolor,graphicx,bbm,xspace,nicefrac,stmaryrd,url,arcs,listings,nameref}
\usepackage[theoremfont]{newpxtext}
\usepackage[vvarbb, upint, bigdelims]{newpxmath}
\usepackage[scr=boondoxupr, scrscaled = 1.05, cal = euler, calscaled =1.05]{mathalpha}
\usepackage[scaled=0.95]{inconsolata}
\DeclareMathAlphabet{\mathsf}{OT1}{\sfdefault}{m}{n}
\SetMathAlphabet{\mathsf}{bold}{OT1}{\sfdefault}{b}{n}
\usepackage[utf8]{inputenc}

\usepackage[shortlabels]{enumitem}
\usepackage{etoolbox}
\usepackage[normalem]{ulem}
\usepackage{mathtools}
\usepackage{geometry}
\usepackage{float}

\usepackage{bm}
\usepackage{tikz}
\usepackage[outdir =./]{epstopdf}
\usepackage[citestyle = numeric-comp,bibstyle = numeric, firstinits=true, natbib = true, backend = bibtex,maxbibnames=99,url=false,doi=false]{biblatex}
\usepackage{authblk}
\numberwithin{equation}{section}

\usepackage{chngcntr}
\counterwithin{figure}{section}

\usepackage{xcolor}
\usepackage[colorlinks=true, allcolors=myteal]{hyperref}
\definecolor{WIMgreen}{RGB}{60 134 132}
\definecolor{red_pers}{RGB}{204 37 41}
\definecolor{UMblue}{RGB}{4 47 86}
\definecolor{myteal}{RGB}{0 123 137}
\definecolor{nd}{RGB}{0 0 0}
\definecolor{dartmouthgreen}{rgb}{0.05, 0.5, 0.06}\definecolor{cobalt}{rgb}{0.0, 0.28, 0.67}\definecolor{coolblack}{rgb}{0.0, 0.18, 0.39}
\definecolor{glaucous}{rgb}{0.38, 0.51, 0.71}\definecolor{hooker\'sgreen}{rgb}{0.0, 0.44, 0.0}\definecolor{lemonchiffon}{rgb}{1.0, 0.98, 0.8}\definecolor{oucrimsonred}{rgb}{0.6, 0.0, 0.0}\definecolor{radicalred}{rgb}{1.0, 0.21, 0.37}\definecolor{raspberry}{rgb}{0.89, 0.04, 0.36}\definecolor{royalazure}{rgb}{0.0, 0.22, 0.66}
\definecolor{dex}{RGB}{138 18 34}
\theoremstyle{plain}
\newtheorem{theorem}{Theorem}[section]
\newtheorem{proposition}[theorem]{Proposition}
\newtheorem{lemma}[theorem]{Lemma}

\theoremstyle{definition}
\newtheorem{definition}[theorem]{Definition}
\theoremstyle{assumption}

\theoremstyle{remark}
\newtheorem{remark}[theorem]{Remark}
\newtheorem{example}[theorem]{Example}
\SetLabelAlign{center}{\hss#1\hss}

\def\card{\operatorname{card}}

\def\supp{\operatorname{supp}}

\def\TV{\operatorname{TV}}

\def\H{\mathbb{H}}

\def\E{\mathbb{E}}
\def\G{\mathbb{G}}

\def\I{\mathbb{I}}
\def\N{\mathbb{N}}

\def\N{\mathbb{N}}

\def\R{\mathbb{R}}
\def\S{\mathcal{S}}\definecolor{darkred}{rgb}{0,0.6,0}
\def\X{\mathbf{X}}
\def\Y{\mathbf{Y}}
\def\Z{\mathbf{Z}}

\def\cN{\mathcal{N}}

\def\cC{\mathcal{C}}

\def\cS{\mathcal{S}}

\def\Var{\mathrm{Var}}
\def\Cov{\mathrm{Cov}}

\newcommand{\cA}{\mathscr{A}}
\newcommand{\cB}{\mathcal{B}}

\newcommand{\cF}{\mathcal{F}}
\newcommand{\cG}{\mathcal{G}}

\newcommand{\cL}{\mathcal{L}}

\newcommand{\ep}{\varepsilon}
\newcommand{\cO}{\mathcal{O}}

\newcommand{\GG}{\mathscr{G}}

\newcommand{\PP}{\mathbb{P}}

\newcommand{\Pro}{\mathbb{P}}
\renewcommand{\c}{^{\operatorname{c}}}

\renewcommand{\subseteq}{\subset}

\renewcommand{\hat}{\widehat}
\newcommand{\e}{\mathrm{e}}
\renewcommand{\tilde}{\widetilde}%
\renewcommand{\d}{\mathop{}\!\mathrm{d} }
\newcommand{\lebesgue}{\boldsymbol{\lambda}}

\newcommand{\qv}[1]{\langle#1\rangle}
\newcommand{\1}{\mathbf{1}}
\newcommand\co{
	\mathchoice
	{{\scriptstyle\mathcal{O}}}% \displaystyle
	{{\scriptstyle\mathcal{O}}}% \textstyle
	{{\scriptscriptstyle\mathcal{O}}}% \scriptstyle
	{\scalebox{.7}{$\scriptscriptstyle\mathcal{O}$}}%\scriptscriptstyle
}
\restylefloat{table}

\newcommand*\diff{\mathop{}\!\mathrm{d} }
\newcommand{\cH}{\mathcal{H}}

\newcommand{\vertiii}[1]{{\left\vert\kern-0.25ex\left\vert\kern-0.25ex\left\vert #1
		\right\vert\kern-0.25ex\right\vert\kern-0.25ex\right\vert}}

\def\Beta{\boldsymbol{\eta}}

\def\card{\mathrm{card}}
\def\supp{\mathrm{supp}}

\def\h{\boldsymbol{h}}
\def\ka{\operatorname{k}}

\let\originalleft\left
\let\originalright\right
\renewcommand{\left}{\mathopen{}\mathclose\bgroup\originalleft}
\renewcommand{\right}{\aftergroup\egroup\originalright}

\definecolor{myteal}{RGB}{0 123 137}
\definecolor{cs}{rgb}{0 0 0}
\definecolor{radicalred}{rgb}{1.0, 0.21, 0.37}
\definecolor{dex}{rgb}{0 0 0}
\allowdisplaybreaks
\bibliography{kinetic}
\title{Estimating the characteristics of stochastic damping Hamiltonian systems from continuous observations}
\author{Niklas Dexheimer\thanks{Aarhus University, Department of Mathematics, Ny Munkegade 118, 8000 Aarhus C, Denmark.\newline Email: \href{mailto:dexheimer@math.au.dk}{dexheimer@math.au.dk}} \qquad Claudia Strauch\thanks{Aarhus University, Department of Mathematics, Ny Munkegade 118, 8000 Aarhus C, Denmark.\newline Email: \href{mailto:strauch@math.au.dk}{strauch@math.au.dk}\newline CS gratefully acknowledges financial support of Sapere Aude: DFF-Starting Grant 0165-00061B
		“Learning diffusion dynamics and strategies for optimal control”.} 
}
\begin{document}

	%\thanks{Aarhus University, Department of Mathematics, Ny Munkegade 118, 8000 Aarhus C, Denmark.\newline Email: \href{mailto:dexheimer@math.au.dk}{dexheimer@math.au.dk}} \qquad
%	 Claudia Strauch\thanks{Aarhus University, Department of Mathematics, Ny Munkegade 118, 8000 Aarhus C, Denmark.\newline Email: \href{mailto:strauch@math.au.dk}{strauch@math.au.dk}} 
%}

\maketitle

\begin{abstract}
We consider nonparametric invariant density and drift estimation for a class of multidimensional degenerate resp.\ hypoelliptic diffusion processes, so-called stochastic damping Hamiltonian systems or kinetic diffusions, under anisotropic smoothness assumptions on the unknown functions. 
The analysis is based on continuous observations of the process, and the estimators' performance is measured in terms of the $\sup$-norm loss. 
Regarding invariant density estimation, we obtain highly nonclassical results for the rate of convergence, which reflect the inhomogeneous variance structure of the process. 
Concerning estimation of the drift vector, we suggest both non-adaptive and fully data-driven procedures. 
%In both cases we arrive at the classical nonparametric rate of convergence for the $\sup$-norm loss. 
All of the aforementioned results strongly rely on tight uniform moment bounds for empirical processes associated to deterministic and stochastic integrals of the investigated process, which are also proven in this paper.
% We consider the question of estimating the drift and the invariant density for a large class of scalar ergodic diffusion processes, based on continuous observations, in $\sup$-norm loss. 
% The unknown drift $b$ is supposed to belong to a nonparametric class of smooth functions of unknown order. 
% We suggest an adaptive approach which allows to construct drift estimators attaining minimax optimal $\sup$-norm rates of convergence.
% In addition, we prove a Donsker theorem for the classical kernel estimator of the invariant density and establish its semiparametric efficiency.
% Finally, we combine both results and propose a fully data-driven bandwidth selection procedure which simultaneously yields both a rate-optimal drift estimator and an asymptotically efficient estimator of the invariant density of the diffusion. 
% Crucial tool for our investigation are uniform exponential inequalities for empirical processes of diffusions.
\end{abstract}

\section{Introduction}

Diffusion processes have been in the center of attention of the statistical analysis of stochastic processes for a long time due to their various fields of application, e.g.\ meteorology, genetics, financial mathematics and neuroscience. 
% However, in many related works the diffusion coefficient is assumed to be nondegenerate. 
{\color{cs}
	A standard assumption often imposed for the in-depth investigation is the strict ellipticity of the diffusion operator. 
	In particular, this regularity condition (together with other assumptions) allows to verify helpful analytical tools such as the existence of a spectral gap.
	However, the nondegeneracy assumption excludes many processes which are of great importance for applications, {\color{nd} for an overview see e.g.\ \cite{ditlevsen2019}}. 
	A prominent example are so-called \emph{stochastic damping Hamiltonian systems}, which are often interpreted as a model for the coupled velocity $\Y$ and position $\X$ of some object.}
More specifically, one considers a multivariate diffusion process $\Z=(Z_t)_{t\ge0}=(X_t,Y_t)_{t\ge0}=(\X,\Y)$, which is governed by the stochastic differential equation (SDE) 
\begin{align}
	\begin{split}\label{eq: sde}
		\d X_t&=Y_t\d t,\\
		\d Y_t&= -(c(X_t,Y_t)Y_t+\nabla V(X_t))\d t+\sigma(X_t,Y_t)\d W_t.
	\end{split}
\end{align}
Here, $c\colon\R^{d}\times\R^{d}\to\R^{d\times d}, V\colon\R^d\to\R, \sigma\colon\R^{d}\times\R^{d}\to \R^{d\times d}$, and $(W_t)_{t\geq0}$ is a $d$-dimensional Brownian motion independent of the random variables $X_0,Y_0$. 
In particular, the velocity process $\Y$ is given as a nondegenerate diffusion process, whereas the definition of the position process $\X$ implies $\Z$ to be degenerate, respectively hypoelliptic.

In the following, we focus on the nonparametric estimation of the invariant density and the drift of a multidimensional diffusion {\color{cs}$\Z$ described by \eqref{eq: sde}, assuming a continuous record of observations is available. 
	% More specifically, $\Z$ is governed by the stochastic differential equation (SDE)
	% \begin{align}
		% 	\begin{split}\label{eq: sde}
			% 		\d X_t&=Y_t\d t,\\
			% 		\d Y_t&= -(c(X_t,Y_t)Y_t+\nabla V(X_t))\d t+\sigma(X_t,Y_t)\d W_t,
			% 	\end{split}
		% \end{align}
	% where $c\colon  \R^{d}\times\R^{d}\to\R^{d\times d}, V\colon \R^d\to \R,  \sigma\colon  \R^{d}\times\R^{d}\to \R^{d\times d}$ and $(W_t)_{t\geq0}$ is a $d$-dimensional Brownian motion independent of the random variables $X_0,Y_0$. 
	% Such processes are sometimes referred to as \emph{stochastic damping Hamiltonian systems} and are often thought of as model for the coupled velocity $\Y$ and position $\X$ of some object. 
	% Thus, it is no surprise that these processes are also applied for various modelling purposes (see e.g.\ the references in \cite{ditlevsen2019}). 
	Invariant density estimation for stochastic damping Hamiltonian systems has been investigated before by \cite{cattiaux2014} and \cite{Comte2017}, focusing however on estimation based on discrete (and partial) observations. 
	The framework of continuous observations considered here is different in some respects.}
{\color{nd}
	On the one hand, it has been shown in the current reference \cite{delattre2020} in the scalar framework that the assumption of continuous observations allows for a refined variance analysis and, as a consequence, for the derivation of nonclassical convergence rates for nonparametric estimation of the invariant density.
	On the other hand, the question of estimation under partial observations, i.e., based only on observations of the position $\X$, does not need to be considered in our setting: The assumption of continuous observations trivialises this problem, as the path of $\Y$ is easily computable by derivation of $\X$ in this case.
	Similar to \cite{delattre2020}'s approach, we employ a kernel estimator for estimating the invariant density, considering now however a general $d$-dimensional framework and the $\sup$-norm as a risk criterion.}
% For the estimation of the invariant density of $\Z$, we employ a kernel density estimator and obtain highly nonclassical rates of convergence for the $\sup$-norm risk in all dimensions. 
{\color{cs}To find a good compromise between appropriate generality and technical complexity, we work with an anisotropic mixture of isotropic smoothness assumptions, i.e., we assume that the components $\X$ and $\Y$ are associated with H\"older-smoothness coefficient $\beta_1$ and $\beta_2$, respectively, where $\beta_1=\beta_2$ does not necessarily hold (see Definition \ref{def: holder} for the precise description).}
The upper bounds on the $\sup$-norm rates of convergence which we derive are nonclassical in multiple ways. 
Firstly, they depend largely on the set, respectively point, where the invariant density is estimated. 
Secondly, the ratio of the smoothness parameters of the invariant density's assumed anisotropic H\"older regularity also plays a vital role in the rate of convergence. 
{\color{cs}The function $\Upsilon$ introduced in \eqref{def:Ups} reflects these specifities.}
Lastly, the proven rates of convergence are faster than in the classical case of kernel density estimation based on a discrete set of i.i.d.~observations.
% These results can be seen as an extension of \cite{delattre2020}, where similar observations were made for the pointwise risk in the scalar case. 
All of the aforementioned nonclassicalities stem from the variance bounds for the kernel density estimator in Section \ref{subsec: variance}, which are a result of the degeneracy of $\Z$. 
For a more thorough explanation of this, see Section \ref{subsec: rate inv}. 
The last observation of the rate of convergence being faster than in the classical case of nonparametric density estimation has also been made for nondegenerate diffusion processes, see, e.g., \cite{dalalyan07} or \cite{strauch18}, or more generally for suitable Markov processes \cite{dexheimer2020mixing}.

{\color{cs}
	Our (auxiliary) results for estimating the invariant density are not only of theoretical interest, but are also directly applied to the second question addressed in this paper, namely the estimation of the drift function
	\[
	b\colon \R^{d}\times\R^{d}\to \R^{d},\quad  b(x,y)\mapsto -(c(x,y)y+\nabla V(x)).
	\]
	To the best of our knowledge, nonparametric drift estimation for stochastic damping Hamiltonian systems has so far only been considered in \cite{cattiaux2014drift} who prove asymptotic normality of their estimator.
	By way of contrast, we derive an upper bound on the convergence rate with respect to the the $\sup$-norm risk and, in addition, suggest a $\sup$-norm adaptive estimation scheme. 
	The analysis in particular relies on an extension of the uniform moment bounds for additive functionals of exponentially $\beta$-mixing Markov processes obtained in \cite{dexheimer2020mixing} to uniform moment bounds of stochastic integrals with respect to $\Y$, notably exploiting the nondegeneracy of $\Y$.   
	% In particular these results enable us to provide an \textit{adaptive} and fully data-driven estimation scheme for the drift, for which we also arrive at the classical nonparametric rate of convergence.
	Our tools turn out to be sufficiently tight to permit the formulation of a $\sup$-norm adaptive drift estimation procedure as well. Remarkably, although the analysis is essentially based on the fact that the process $\Z$ is exponentially $\beta$-mixing, the associated mixing constants are not relevant for the adaptive procedure. 
	Note that this is in contrast to adaptive estimation procedures for the invariant density of corresponding processes.}

The structure of the paper is as follows. 
In Section \ref{sec: assumptions}, the necessary assumptions on $\Z$ are denoted, which are of a quite technical nature, followed by a collection of explicit conditions on the coefficients of \eqref{eq: sde}, which imply the necessary assumptions. 
Section \ref{sec: inv dens} consists of the results regarding invariant density estimation, with Section \ref{subsec: variance} containing the aforementioned variance bounds for the kernel density estimator and Section \ref{subsec: rate inv} introducing the concrete rates of convergence. 
Lastly, Section \ref{sec: drift} accommodates the results on drift estimation, with Section \ref{subsec: unif} presenting the extended uniform moment bounds, Section \ref{subsec: drift rate} containing the results on the rate of convergence and Section \ref{subsection: adap} describing our results on adaptive drift estimation. 
For the sake of readability, all proofs have been deferred to the appendix.

\section{Preliminaries}\label{sec: assumptions}
Throughout the paper, we assume that the SDE \eqref{eq: sde} admits a unique, non-explosive weak solution whose semigroup is strongly Feller. 
Furthermore, we impose the following assumptions on the associated solution $\Z$.

\begin{enumerate}[$(\cA_1)$]
	\label{framework}
	\item\label{ass: transition density}
	The marginal laws of $\Z$ are absolutely continuous, i.e., for any $t > 0$ and $z \in \R^{2d}$, there exists a measurable function $p_t\colon \R^{2d} \times \R^{2d} \to \R_+$ such that
	\[
	P_t(z_1,B) = \int_B p_t(z_1,z_2) \diff{z_2}, \quad B \in \mathcal{B}(\R^{2d}),
	\]
	and, moreover, $\Z$ admits a unique absolutely continuous invariant probability measure $\mu$, i.e., there exists a density $\rho\colon \R^{2d} \to \R_+$ such that $\diff \mu=\rho\d\lebesgue$ and, for any Borel set $B$, 
	\[
	\PP^\mu(Z_t\in B) \coloneqq \int_{\R^{2d}} P_t(z_1,B) \, \mu(\diff z_1) = \int_{\R^{2d}} \int_B p_t(z_1,z_2) \rho(z_1) \d{z_2} \d{z_1} = \int_B \rho(z_1) \d{z_1} = \mu(B).
	\]
	\item\label{ass: heat kernel bound} 
	For any bounded set $D\subset\R^{2d}$, there exist constants $c_U,c_G>0$ depending on $D$ such that, for all $z_1=(x_1,y_1)$, $z_2=(x_2,y_2)\in D$, $t\in (0,1]$,
	\[ 
	p_t(z_1,z_2)\leq p_t^G(z_1,z_2)+p_t^U(z_1,z_2),
	\]
	where 
	\[
	p_t^G(x_1,y_1;x_2,y_2)=c_G t^{-2d} \exp\left(-c_G^{-1}\left( \frac{\Vert y_1-y_2\Vert^2}{4t}+ \frac{ 3\Vert x_2-x_1- \frac{t(y_1+y_2)}{2}\Vert^2}{t^3}\right) \right),
	\] 
	and $p^U_t$ is a non-negative, measurable function such that, for any $z_1\in D$, $t\in(0,1]$,
	\[
	\int_{\R^{2d}} p^U_t(z_1,z_2)\d z_2\leq c_U \exp \left(-\frac{1}{c_U t} \right).
	\]
	\item\label{ass: mixing} 
	The process $\Z$ started in the invariant measure $\mu$ is exponentially $\beta$-mixing, i.e., there exist constants $c_\kappa,\kappa>0$ such that
	\[
	\int \Vert P_t(z,\cdot)-\mu(\cdot)\Vert_{\TV}\,\mu(\diff z)\ \le\ c_\kappa\e^{-\kappa t}, \quad t \geq 0,
	\]
	where $\Vert\cdot\Vert_{\TV}$ denotes the total variation norm.	
\end{enumerate}
We will refer to this framework as \hyperref[framework]{$\cA$}. {\color{nd} Contrary to the rather general and standard assumptions \ref{ass: transition density} and \ref{ass: mixing}, the heat kernel bound \ref{ass: heat kernel bound} is specifically tailored for the study of kinetic diffusions. This will be more evident from the new set of assumptions $\tilde\cA$ introduced below, which contains \emph{explicit} conditions on the coefficients of the SDE \eqref{eq: sde}. As will be shown $\tilde\cA$ implies \hyperref[framework]{$\cA$}, and thus our results also hold true under these more practical and verifiable assumptions.}

\begin{description}\label{framework tilde}
	\item[$(\tilde{\cA}_V)$] 
	$V$ is twice continuously differentiable and lower bounded.
	\item[$(\tilde{\cA}_c)$] 
	$c$ is continuously differentiable and uniformly bounded. 
	Furthermore, there exist $c_1,l>0$ such that $c^s(x,y)\geq c_1\mathbb{I}_{d\times d}$ for all $\vert x\vert>l,y\in\R^d$.
	\item[$(\tilde{\cA}_\sigma)$] 
	$\sigma$ is uniformly elliptic, symmetric and infinitely differentiable. 
	Additionally, there exists $c_2>0$ such that $\sigma(x,y)\leq c_2 \mathbb{I}_{d\times d}$.
	\item[$(\tilde{\cA}_{\mathrm{Erg}})$] 
	$\vert x\vert^{-1}\qv{\nabla V(x),x }\to \infty$ as $\vert x\vert \to \infty$.
\end{description}

Here, $c^s(x,y)$ denotes the symmetrization of the matrix $c(x,y)$, i.e., $c^s(x,y)\coloneqq \tfrac 1 2(c_{ij}(x,y)+c_{ji}(x,y))_{1\leq i,j\leq d}$, $\mathbb{I}_{d\times d}$ is the identity matrix in $\R^{d\times d}$, and the order relation on symmetric matrices is the usual one defined by definite non-negativeness. 
It is easy to see that $\tilde \cA$ is a multidimensional generalization of the assumptions \textbf{HReg} and \textbf{HErg} in \cite{delattre2020}. 
Hence, our results (such as upper bounds on the $\sup$-norm risk for invariant density and drift estimators) also hold for the class of processes investigated in their recent paper.

The next auxiliary result confirms that the explicit assumptions $\tilde\cA$ indeed provide the technical framework required for our analysis.

\begin{lemma}
	$\tilde \cA$ implies $\cA$.
\end{lemma}
\begin{proof}
	The assertion follows by results of \cite{wu2001}, \cite{konakov2010} and \cite{cattiaux2014}. 
	To be more precise, the fact that $\tilde{\cA}$ implies that \eqref{eq: sde} has a unique, non-explosive weak solution with strong Feller semigroup follows by Lemma 1.1 and Proposition 1.2 in \cite{wu2001}, and the fact that a unique invariant probability measure exists follows by Theorem 3.1, once we note that the conditions (3.1) and (3.2) therein are fulfilled when $\tilde{\cA}$ holds (see Remark 3.2 in \cite{wu2001}). 
	The mixing property follows by the proof of Theorem 3.1 in \cite{wu2001}, which implies the existence of a Lyapunov function larger than $1$. 
	Thus, the process is exponentially ergodic due to Theorem 2.4 in \cite{wu2001}, which is based on results of \cite{down1995}, and since Theorem 2.4 also states that any Lyapunov function is integrable with respect to the invariant measure, the mixing property follows. 
	The heat kernel bound in $\cA$ was shown in Theorem 2.1 of \cite{konakov2010} for differentiable, globally Lipschitz continuous and uniformly bounded coefficients and extended in Corollary 2.12 of \cite{cattiaux2014} to hold under local conditions. 
	Noting that the results of \cite{konakov2010} only require differentiability of the coefficients, the proof of Corollary 2.12 in \cite{cattiaux2014} still holds true under $\tilde{\cA}$. 
	Additionally, these results imply that the invariant distribution admits a density with respect to the Lebesgue measure. 
	Hence, $\cA$ holds if $\tilde{\cA}$ holds.
\end{proof}

\paragraph{Additional assumptions and notation}
In the following, we will always assume $Z_0\sim \mu,$ i.e., the process $\Z$ is stationary, %\textcolor{purple}{with respect to the invariant probability measure $\mu$}
and we denote $\Pro^\mu = \Pro$, $\E^\mu=\E$.
Additionally, we define $\mu(g) \coloneqq \int g \diff{\mu}$ for $g \in L^1(\mu)$, and we introduce 
\[
a\colon \R^{d}\times\R^{d}\to \R^{d\times d},\quad  a(x,y)\mapsto \sigma(x,y)\sigma^\top(x,y).
\]
Given a class of functions $\GG$ and a function $b$, we set $\GG b\coloneqq \{gb: g\in\GG \},$ {\color{nd} and for $x\in\R^d, \ep>0$, we denote the open ball with radius $\ep$ around $x$ by $B(x,\ep)$.} 
Throughout all proofs, $c$ will denote a positive constant, whose value may change from line to line, whereas specific constants are denoted by a $c$ with additional subscript. 
Furthermore, we denote the restriction of the $\sup$-norm to a domain $D\subset \R^{2d}$ by $\lVert \cdot \rVert_{L^\infty(D)}$. 
Lastly, the $\sup$-norm risk of an estimator $\hat{f}$ of a function $f$ on a domain $D$ is denoted by
\[
\mathcal{R}^{(p)}_{\infty}\big(\hat{f}, f; D\big)\coloneqq \E\left[\Vert \hat{f}-f\Vert_{L^\infty(D)}^p\right]^{\frac 1 p}, \quad p\geq 1.
\]

\section{Invariant density estimation}\label{sec: inv dens}
We start by investigating the issue of estimating the invariant density of $\Z$ on some bounded domain $D\subset \R^{2d}$.
As mentioned earlier, our particular interest lies in identifying the effect of a continuous observation scheme on the convergence rate. 
Reflecting the two-component structure of the process $\Z=(\X,\Y)$, our estimator has the form
\begin{equation}\label{def: inv dens est}
	\hat{\rho}_{h_1,h_2,T}(x,y)=\frac 1 T \int_0^T K_{h_1,h_2}(x-X_s,y-Y_s)\d s, \quad x,y\in\R^d,
\end{equation}
where 
\[
K_{h_1,h_2}\colon \R^d\times \R^d\to \R,\quad K_{h_1,h_2}(x,y)\mapsto (h_1h_2)^{-d}K_1\left(\frac x h_1\right)K_2\left( \frac y h_2 \right),
\]
with $K_1,K_2\colon \R^d\to\R$ bounded functions with $\supp(K_i)\subset [-1/2,1/2]^d$, $i=1,2$, and bandwidths $h_1,h_2>0$. 
When the context is clear, we will often abbreviate this (by a slight abuse of notation) with $K_{\h}(z)$, where $\h=(h_1,h_2)$, $z=(x,y)$, $x,y\in \R^{d}$. 
Usually, $h_1,h_2$ are functions depending on $T$, however (by another abuse of notation), we often suppress this dependence.

\subsection{Variance bounds}\label{subsec: variance}
The proof of tight upper bounds on the speed of convergence of estimators requires sufficiently neat variance bounds.
To achieve these, we will extend the results of Propositions 2 and 3 of \cite{delattre2020}, where remarkable bounds are shown for fixed values of $(x,y)$ and $d=1$. 
It is particularly interesting that the bounds are different for $y=0$ and $y\neq 0$. 

In the sequel, we will also consider the multidimensional case $d>1$, and we will provide a \emph{uniform} generalization for values of $(x,y)$ in some bounded set $D\subset\R^{2d}$.
The bandwidths will be assumed to belong to the set
\begin{equation}\begin{split}\label{def:H}
		\mathcal{H}&=\mathcal{H}(Q_1,Q_2)\coloneqq \Big\{h\colon [0,\infty)\to(0,\infty) :\exists Q_1,Q_2>0 \textrm{ such that }\forall T>0:\\
		&\hspace*{10em} h^{-1}(T)\leq Q_1(1+T)^{Q_1}\textrm{ and } h(T)\leq Q_2T^{-Q_2}\wedge 1\Big\},
\end{split}\end{equation}
thus confining their speed of convergence to $0$ {\color{nd} to be approximately polynomial}.
We start with investigating the more general case.

\begin{proposition} \label{prop: variance not bounded}
	Assume \hyperref[framework]{$\cA$}, $\Vert \rho\Vert_\infty<\infty$, and let $f\colon \R^d\times\R^d\times [0,\infty)\to \R$ be a bounded function such that there exist 
	$n\in\N$, $\mathsf{x}_1,\mathsf{y}_1,\ldots,\mathsf{x}_n,\mathsf{y}_n\in\R^d$ and functions $\mathsf{s}_1,\mathsf{s}_2\colon[0,\infty)\to(0,\infty)\in\mathcal{H} $ fulfilling 
	$$\operatorname{supp}(f(\cdot,\cdot,T))\subseteq\bigcup_{i=1}^n B(\mathsf{x}_i,\mathsf{s}_1(T))\times B(\mathsf{y}_i,\mathsf{s}_2(T)).$$ 
	Then, there exists a constant $c>0$ such that, for large enough $T$,
	\begin{equation}
		\Var\left(T^{-1}\int_0^T f(X_s,Y_s,T)\d s \right)\leq cT^{-1}\Vert f\Vert^2_\infty \times \begin{cases}
			\mathsf{s}_1^{2}(T)\log (T)&d=1,\\
			\mathsf{s}_1^{d+1}(T)\mathsf{s}_2^{d-1}(T),&d\geq2,
		\end{cases} \label{eq: var th 2}
	\end{equation} 
	and 
	\begin{equation}
		\Var\left(T^{-1}\int_0^T f(X_s,Y_s,T)\d s \right)\leq cT^{-1}\Vert f\Vert^2_\infty \times\begin{cases}
			\mathsf{s}_1^{4/3}(T)\mathsf{s}_2^{2}(T),&d=1,\\
			\mathsf{s}_1^{2}(T)\mathsf{s}_2^{4}(T)\log(T),&d=2,\\
			\mathsf{s}_1^{d}(T)\mathsf{s}_2^{d+2}(T), &d\geq3.
		\end{cases} \label{eq: var th 1}
	\end{equation}
\end{proposition}
% \begin{proposition} \label{prop: variance not bounded}
	% Assume \hyperref[framework]{$\cA$}, $\Vert \rho\Vert_\infty<\infty$, and let $D\subset \R^{2d}$ be a bounded set. 
	% Then, for bandwidths $h_1,h_2$ fulfilling $(\cH)$, there exists a constant $c>0$ such that, for large enough $T$, it holds for all $(x,y)\in D$, if $h_1\leq h_2$ for $d=1$ and $h_1(\log T)^{3/2}\leq h_2$ for $d\geq 2$,
	% 	\begin{equation} \label{eq: var th 2}
		% 		\Var\left(\hat{\rho}_{h_1,h_2,T}(x,y) \right)\le c\ \begin{cases}
			% 			T^{-1}h_2^{-2}\log T&d=1,\\
			% 		T^{-1}h_1h_2^{-1}(h_1h_2)^{-d},&d\geq2,
			% 		\end{cases}
		% 	\end{equation} 
	% 	and, if $h_2^{2}\leq h_1^{2/3}$ in the case $d=2$,
	% 	\begin{equation} \label{eq: var th 1}
		% 		\Var\left(\hat{\rho}_{h_1,h_2,T}(x,y) \right)\leq c\ \begin{cases}
			% 			T^{-1}h_1^{-2/3},&d=1,\\
			% 		T^{-1}h_1^{-2}\log T,&d=2,\\
			% 		T^{-1}h_2^2(h_1h_2)^{-d}, &d\geq3.
			% 		\end{cases}
		% 	\end{equation}
	% \end{proposition}

The assumption on the support of $f$ in Proposition \ref{prop: variance not bounded} is tailored to functions of the form 
\[f(x,y,t)=\sum_{i=1}^{n}K_1((\mathsf{x}_i-\cdot)/\mathsf{s}_1(t))K_2(\mathsf{y}_i-\cdot)/\mathsf{s}_2(t)),\] 
where $K_1,K_2\colon \R^d\to \R$ are bounded functions with compact support, {\color{nd} and $\mathsf{s}_1,\mathsf{s}_2\in\mathcal{H}$. In particular the function $K_{h_1,h_2}$ corresponding to the estimator introduced in \eqref{def: inv dens est} is of such a form as soon as the two bandwidths are elements of $\mathcal{H}$.
	The results in \cite{delattre2020} now suggest that the variance bounds \eqref{eq: var th 2} and \eqref{eq: var th 1} can be improved if $\min_{i=1,\ldots,n}\Vert \mathsf{y}_i\Vert>0$. 
	However, as in the referred work, the proof of these refined results largely relies on the upper bound $\Vert x-x'\Vert \lesssim \mathsf{s}_1$, being valid for all $x,x'$ in the support of $f$ (see \eqref{eq: proof variance bounded} and the arguments thereafter). Now note that if $n>1,$ there exist $(x,y),(x',y')\in\supp(f),$ such that $x\in B(x_1,\mathsf{s}_1(T)),x'\in B(x_2,\mathsf{s}_1(T)),$ which implies $\Vert x-x'\Vert\geq \Vert \mathsf{x}_1-\mathsf{x}_2\Vert-2\mathsf{s}_1(T)$ by the reverse triangle inequality. This contradicts the needed upper bound for small enough values of $\mathsf{s}_1$ and thus the following variance bound only concerns the case $n=1$.}
%As explained before, these bounds can be refined if {\color{radicalred}$n=1$ and $\Vert \mathsf{y}\Vert$} is bounded away from $0$.

\begin{proposition}\label{prop: variance bounded}
	Let everything be given as in Proposition \ref{prop: variance not bounded} with $n=1$ and assume, additionally, that $\Vert \mathsf{y}_1\Vert >0 $.
	Then, there exists a constant $c>0$ such that, for large enough $T$,
	\begin{equation}\label{eq: var th 3}
		\Var\left(T^{-1}\int_0^T f(X_s,Y_s,T)\d s \right)
		\leq cT^{-1}\Vert f\Vert_\infty^2 \mathsf{s}_1^{d+1}(T)\mathsf{s}_2^d(T),
	\end{equation}
	and, additionally, for $d=1$,
	\begin{equation}\label{eq: var th 4}
		\Var\left(T^{-1}\int_0^T f(X_s,Y_s,T)\d s \right)\leq 
		cT^{-1}\Vert f\Vert_\infty^2 \mathsf{s}_1^{3/2}(T)\mathsf{s}_2^2(T).		
	\end{equation}
\end{proposition}

An interpretation of the highly nonclassical results stated in Propositions \ref{prop: variance not bounded} and \ref{prop: variance bounded} will be given after our main results on the rate of convergence in the next section (see Theorem \ref{cor: inv dens 1}).

\subsection{Rate of convergence}\label{subsec: rate inv}
With the introduced variance bounds, we are able to bound the convergence rate of the estimator under specific assumptions on the invariant density $\rho$, resulting in new upper bounds. 
In order to use the different results of Propositions \ref{prop: variance not bounded} and \ref{prop: variance bounded}, we also introduce the functions $\psi_d(x,y,t)\coloneqq \psi_{1,d}(x,y,t)\land\psi_{2,d}(x,y,t)$, where
\[\psi_{1,d}(x,y,t)\coloneqq \begin{cases}
	y^{-1}\sqrt{\log t} , & d=1,\\
	x^{(1-d)/2} y^{-(1+d)/2}, & d\geq 2,
\end{cases}
\quad\text{ and }\quad
\psi_{2,d}(x,y,t)\coloneqq \begin{cases}
	x^{-1/3} , & d=1,\\
	x^{-1}\sqrt{\log t}, & d= 2,\\
	x^{-d/2} y^{1-d/2}, & d\geq 3,
\end{cases}
\]
and $\psi^\circ_d(x,y,t)\coloneqq \psi^\circ_{1,d}(x,y,t)\land \psi^\circ_{2,d}(x,y,t)$, with
\[ \psi^\circ_{1,d}(x,y,t)\coloneqq x^{(1-d)/2} y^{-d/2}\quad\text{ and }\quad 
\psi^\circ_{2,d}(x,y,t)\coloneqq\begin{cases}
	x^{-1/4}, &d=1,\\
	\psi_{2,d}(x,y,t), &d\geq 2.
\end{cases}
\]
Note that $\psi_d^2$ and $(\psi^\circ_d)^2$ represent the variance bounds (up to the term $T^{-1}$) for the estimator $\hat\rho_{h_1,h_2,t}$ implied by Propositions \ref{prop: variance not bounded} and \ref{prop: variance bounded}.
\textcolor{dex}{
	One remarkable fact in this context is that
	\[\psi_d^2(h_1,h_2)< (h_1h_2)^{-d}, \] i.e., our obtained variance bounds are tighter compared to the classical one obtained for the kernel density estimator $\hat\rho_{h_1,h_2,t}$. 
	This is one of the reasons for the faster rates of convergence we will see later, compared to the classical nonparametric rate of convergence.} 

\begin{proposition}\label{th: inv density estimation general}
	\textcolor{dex}{Let $1\leq p\leq\gamma \log T$, for $\gamma>0$,} and $D\subset \R^{2d}$ be a bounded, open set, assume \hyperref[framework]{$\cA$}, and choose $h_1=h_1(T)$, $h_2=h_2(T)\in \cH$. 
	Then, there exists a constant $c>0$ independent of $p$ such that, for large enough $T$,
	\begin{equation}\label{eq: inv density estimation general}\mathcal{R}^{(p)}_{\infty}\big(\hat{\rho}_{h_1,h_2,T} ,\rho; D\big)
		\leq c\left(\mathcal B_\rho(h_1,h_2)+\frac{p\log T}{T(h_1h_2)^d}\left(\log T+p \right)+\frac{\psi_d(h_1,h_2,T)}{\sqrt{T}}\left(\sqrt{\log T}+\sqrt{p}\right)
		\right),
	\end{equation}
	where the bias term is given as $\mathcal B_\rho(\h)=\mathcal B_\rho(h_1,h_2)\coloneqq \sup_{z\in\R^{2d}}\left|(\rho\ast K_{h_1,h_2}-\rho)(z)\right|$.
\end{proposition}

The refined variance bounds stated in Proposition \ref{prop: variance bounded} imply the following result:

\begin{proposition}\label{th: inv density estimation bounded}
	Let $D\subset \R^{2d}$ be a bounded, open set such that $\inf_{(x,y)\in D}\Vert y\Vert>0$, assume \hyperref[framework]{$\cA$}, and choose $h_1=h_1(T)$, $h_2=h_2(T)\in \cH$ such that
	\begin{equation}\label{ass: bandwidth speed}
		\frac{(\log T)^{3/2} }{\sqrt{T}(h_1h_2)^d\psi^\circ_d(h_1,h_2,T)}\longrightarrow0, \quad \textrm{as } T\to\infty.
	\end{equation}
	Then, there exists a constant $c>0$ such that, for large enough $T$,
	\begin{equation}\label{eq: inv density estimation bounded}
		\mathcal{R}^{(1)}_{\infty}\big(\hat{\rho}_{h_1,h_2,T} ,\rho; D\big)
		\leq c\left(\mathcal B_\rho(h_1,h_2)+\psi^\circ_d(h_1,h_2,T)\sqrt{\frac{\log T}{T}}\right).
	\end{equation}
\end{proposition}
{\color{nd} Note that Assumption \eqref{ass: bandwidth speed} reflects the upper bound of Proposition \ref{th: inv density estimation general}, since it implies the following for large enough values of $T$ 
	\[\frac{(\log T)^2}{T(h_1h_2)^d}\le\sqrt{\frac{\log T}{T}} \psi^\circ_d(h_1,h_2,T). \]
	In fact this condition is only a marginal restriction as will be made clear in Theorem \ref{cor: inv dens 1}. 
	Furthermore, the results of Propositions \ref{th: inv density estimation general} and \ref{th: inv density estimation bounded} reflect the classical bias-variance decomposition, with the term $\mathcal{B}_\rho$ denoting the bias term and $\psi_d(\cdot,T)\sqrt{T^{-1}\log T}$ representing the stochastic error.}
For translating the above results into concrete upper bounds on the convergence rate, we will work under classical H\"older smoothness assumptions for the invariant density, with a small adjustment due to our concrete problem:
In order to reflect the specific form of the process $\Z=(\X,\Y)$, we will use a mixture of isotropic and anisotropic H\"older conditions as described in the following definition.

\begin{definition}\label{def: holder}
	Let $\beta_1,\beta_2,\cL_1,\cL_2>0$ and $D\subset\R^{2d}$ be an open set. 
	A function $g\colon \R^{2d}\to \R$ is said to belong to the anisotropic H\"older class $\cH_{D}(\beta_1,\beta_2,\cL_1,\cL_2)$ if, for all $i=1,\ldots,d$,
	\begin{align*}
		\Vert D^k_i g\Vert_{L^\infty(D)} &\leq \cL_1,\quad \forall k= 0,\ldots, \llfloor\beta_1\rrfloor,\\
		\Vert D^{\llfloor\beta_1\rrfloor}_i g(\cdot +t\mathsf{e}_i)-D^{\llfloor\beta_1\rrfloor}_i g(\cdot )\Vert_{L^\infty(D)} &\leq \cL_1\vert t\vert^{\beta_1-\llfloor\beta_1\rrfloor},\quad \forall t\in \R,
	\end{align*}	
	and, for all $i=d+1,\ldots, 2d$,
	\begin{align*}
		\Vert D^k_i g\Vert_{L^\infty(D)} &\leq \cL_2,\quad \forall k= 0,\ldots, \llfloor\beta_2\rrfloor,\\
		\Vert D^{\llfloor\beta_2\rrfloor}_i g(\cdot +t\mathsf{e}_i)-D^{\llfloor\beta_2\rrfloor}_i g(\cdot )\Vert_{L^\infty(D)} &\leq \cL_2\vert t\vert^{\beta_2-\llfloor\beta_2\rrfloor},\quad \forall t\in \R,
	\end{align*}
	where $D^k_i g$ is the $k$-th order partial derivative of $g$ with respect to the $i$-th component, $\llfloor \beta\rrfloor$ denotes the largest integer \textit{strictly} smaller than $\beta$ and $\mathsf{e}_1,\ldots,\mathsf{e}_{2d}$ is the canonical basis in $\R^{2d}$ .
\end{definition}

For estimating the invariant density $\rho$ of the process $\Z$ on a domain $D$, assuming that $\rho \in \mathcal{H}_D(\beta_1,\beta_2,\cL_1,\cL_2)$, we choose $K_1,K_2$ to be smooth Lipschitz continuous kernel functions of order $\llfloor \beta_1 \rrfloor,\llfloor \beta_2 \rrfloor$.
Recall that a kernel $K\colon\R^d\to\R$ is said to be of order $\ell\in\N$ if, for any $\alpha\in\N^d$ with $\lvert\alpha\rvert\le\ell$, $x\mapsto x^\alpha K(x)$ is integrable and, moreover,
\[
\int_{\R^d} K(x) \diff{x} = 1, \quad \int_{\R^d} K(x)x^\alpha\d x=0, \quad \alpha \in \N^d,\ \lvert \alpha \rvert \in \{1,\ldots,\ell\},
\]
where $|\alpha|=\sum_{i=1}^d\alpha_i$ and $x^\alpha=\prod_{i=1}^dx_i^{\alpha_i}$ for all $x\in\R^d$, $\alpha\in\N^d$.
For notational convenience, we denote the harmonic mean of the smoothness parameters $\beta_1$ and $\beta_2$ by
\[
\overline\beta_{1,2}\coloneqq 2(\beta_1^{-1}+\beta_2^{-1})^{-1}.
\]
When the context is clear, we will omit the index in this notation. 
For stating our results on the convergence rate in a compact way, it is also useful to introduce the functions \textcolor{dex}{$\Upsilon\colon\R_0^+\times\R_0^+\times\N\times \R_0^+ \to\R_0^+$, $\Phi\colon\R_0^+\times\R_0^+\times\R_0^+\times\N\times \R_0^+\to\R_0^+$ and $\chi_{\mathcal{B}}\colon\R_0^+\times\R_0^+\times\R_0^+\times\N\times \R_0^+\to\R_0^+$, specified as} 
\begin{align}\label{def:Ups}
	\Upsilon(\beta_1,\beta_2,d,\ep)&\coloneqq \begin{cases}
		\frac 2 3	\frac{3\beta_1+\beta_2}{\beta_1+\beta_2},\quad &3\beta_1\geq \beta_2, d=1, \ep=0,\\
		\frac{4\beta_1}{\beta_1+\beta_2},\quad &3\beta_1\geq \beta_2, d\geq2, \ep=0,\\
		2\frac{\beta_2-\beta_1}{\beta_1+\beta_2},\quad &3\beta_1< \beta_2, \ep=0,\\
		\frac{2\beta_1+\beta_2}{\beta_1+\beta_2},\quad &2\beta_1\geq \beta_2, d=1, \ep>0,\\
		\frac{4\beta_1}{\beta_1+\beta_2},\quad &2\beta_1\geq \beta_2, d\geq2, \ep>0,\\
		\frac{2\beta_2}{\beta_1+\beta_2},\quad &2\beta_1< \beta_2, \ep>0,\\
	\end{cases}\\\label{def:Psi}
	\Psi(T,\beta_1,\beta_2,d,\ep)&\coloneqq \left(\frac{\log T}{T} \right)^{\frac{\overline \beta}{2(\overline\beta+d)-\Upsilon(\beta_1,\beta_2,d,\ep)}},\\\nonumber
	\chi_{\mathcal{B}}(T,\beta_1,\beta_2,d,\ep)&\coloneqq
	\begin{cases}
		1,&(\beta_1,\beta_2,d,\ep)\notin \mathcal B\\
		\sqrt{\log T},&(\beta_1,\beta_2,d,\ep)\in \mathcal B,
	\end{cases}
\end{align}
where the set $\mathcal{B}$ is given as
\begin{equation}\label{def:B}
	\mathcal B\coloneqq \left\{(\beta_1,\beta_2,d,\ep)\in\R^4,\textrm{such that one of the following holds: }\begin{cases}
		&3\beta_1< \beta_2\land d=1\land\ep=0 \\
		&3\beta_1>\beta_2\land d=2\land\ep=0\\
		&2\beta_1>\beta_2\land d=2\land \ep>0
	\end{cases}  \right\}.
\end{equation}

We are now ready to state the first of the bounds on convergence rates announced in the introduction.
In line with the two different variance bounds in Propositions \ref{prop: variance not bounded} and \ref{prop: variance bounded}, we will consider both the general case and the case where $\inf_{(x,y)}\|y\|>0$. 
The proof of the upper bounds on the classical $\sup$-norm risk $\mathcal{R}^{(1)}_\infty$ (stated in part (b) below) relies on a classical combination of Proposition \ref{prop: variance bounded} with a discretization of the domain and the exploitation of concentration results. 
For the general case, we even obtain an upper bound for arbitrary $p$-th moments, $p\geq1$ (see part (a)): Since Proposition \ref{prop: variance not bounded} permits to bound the \textit{difference} of kernels, we are able to bound the entropy integrals in the uniform moment bounds obtained in \cite{dexheimer2020mixing}, which then yields \eqref{upper: gen}.

\begin{theorem}\label{cor: inv dens 1}
	Let $D\subset\R^{2d}$ be a bounded, open set, and assume \hyperref[framework]{$\cA$} and $\rho\in \cH_D(\beta_1,\beta_2,\cL_1,\cL_2)$ for $\beta_1>1$, $\beta_2>2$. 
	\begin{enumerate}
		\item[$\operatorname{(a)}$] If the bandwidth is chosen such that
		\begin{equation}\label{def:h Psi}
			h_i\sim \Psi(T,\beta_1,\beta_2,d,0)^{\frac 1 \beta_i},\quad i=1,2,
		\end{equation}
		then the associated invariant density estimator fulfills
		\begin{equation}\label{upper: gen}
			\mathcal{R}^{(p)}_{\infty}\big(\hat{\rho}_{h_1,h_2,T} ,\rho; D\big)
			%\E\left[\sup_{z \in D}\vert \hat{\rho}_{h_1,h_2,T}(z)-\rho(z)\vert^p \right]^{\frac 1 p}
			\in \cO((\Psi\chi_\cB)(T,\beta_1,\beta_2,d,0)), \quad p\ge1. 
		\end{equation}
		\item[$\operatorname{(b)}$]
		Define $\ep_D\coloneqq\inf_{(x,y)\in D}\Vert y\Vert$.
		Then, specifying
		\begin{equation}\label{def:h Psi 2}
			h_i\sim \Psi(T,\beta_1,\beta_2,d,\ep_D)^{\frac 1 \beta_i},\quad i=1,2,
		\end{equation}
		yields
		\[
		\mathcal{R}^{(1)}_{\infty}\big(\hat{\rho}_{h_1,h_2,T} ,\rho; D\big)
		%\E\left[\sup_{z \in D}\vert \hat{\rho}_{h_1,h_2,T}(z)-\rho(z)\vert^p \right]^{\frac 1 p}
		\in \cO((\Psi\chi_\cB)(T,\beta_1,\beta_2,d,\ep_D)). 
		\]
	\end{enumerate}
\end{theorem}

\begin{remark}
	\begin{enumerate}
		\item[(a)]
		Note that, for the proposed specification of bandwidths, the rate of convergence in certain cases only depends on \emph{one} of the smoothness parameters.
		More precisely, the convergence rate is specified as $(\Psi\chi_\cB)(T,\beta_1,\beta_2,d,\ep_D)=(\log T/T)^\alpha\chi_\cB(T,\beta_1,\beta_2,d,\ep_D)$ with
		\[
		\alpha=\alpha(\beta_1,\beta_2,d,\ep_D)\coloneqq
		\begin{cases}
			\frac{\beta_1}{2\beta_1+(2/3)}, &3\beta_1\geq\beta_2, d=1, \ep_D=0, \\
			\frac{\beta_1}{2\beta_1+2}, &3\beta_1\geq\beta_2, d=2, \ep_D=0, \\
			\frac{\beta_2}{2\beta_2+2}, &3\beta_1<\beta_2, d=1, \ep_D=0, \\
			\frac{\beta_1}{2\beta_1+(1/2)}, &2\beta_1\geq\beta_2, d=1, \ep_D>0, \\
			\frac{\beta_1}{2\beta_1+2}, &2\beta_1\geq\beta_2, d=2, \ep_D>0, \\
			\frac{\beta_2}{2\beta_2+1}, &2\beta_1<\beta_2, d=1, \ep_D>0. 
		\end{cases} 
		\]
		Similar results were also obtained in \cite{delattre2020} for the pointwise risk in the scalar case.
		\item[(b)]
		Although the function $\Upsilon$ introduced in \eqref{def:Ups} may seem like a technical artifact of our procedures, it is regular in the sense of being continuous in the smoothness parameters $\beta_1$ and $\beta_2$ for fixed values of $d$ and $\ep$. 
		The only thing that counteracts this regularity in the derived convergence rate is the appearance of an additional logarithmic term in some cases, described by the set $\cB$ and the function $\chi_\cB$. 
		However, this concerns only some cases in a low-dimensional setting ($d=1$ or $d=2$) and was also observed in \cite{delattre2020}. 
		%As can be seen in the proof of the variance bounds in Propositions \ref{prop: variance not bounded} and \ref{prop: variance bounded}, the main reason for these additional logarithmic terms is the fact that the primitive of $x^{-1}$ is the logarithm and thus monotonically increasing, which is in contrast to the behaviour of the primitive of $x^{-c}$ for values of $c>1$. 
		%This property is also the reason for the particular convergence rate for $d=2$ in \cite{strauch18} and \cite{dexheimer2020mixing} and, up to our knowledge, there are yet no results on corresponding lower bounds. We conjecture that convergence rates of this form are optimal in the minimax sense. 
		{\color{cs}\item[(c)]
			In order to translate the above result into a statement on minimax optimality, two steps are necessary: The upper bounds have to be verified uniformly for the class of all diffusions satisfying Assumption \hyperref[framework]{$\cA$}, and the upper bound has to be complemented by a corresponding lower bound. 
			It is very challenging to obtain the mixing control uniformly over a class of diffusions. 
			Instead of directing our efforts in this direction, we focus on constructive aspects: In the upcoming Section \ref{sec: drift}, we study the issue of nonparametric drift estimation, for which we even propose an adaptive procedure. 
			In particular, since our primary interest is not in optimality issues, we refrain from proving lower bounds. However, it is to be expected that such statements can be derived by (elaborate) adaptations of the procedures of \cite{delattre2020}, who considered the case $d=1$ and the pointwise risk.}
		\item[(d)]
		Let us finally compare one aspect of Theorem \ref{cor: inv dens 1} to the scalar, pointwise risk estimates in Theorems 1 and 2 of \cite{delattre2020}.
		In these theorems, one of the bandwidths (depending on the ratio of the two smoothness parameters) can be chosen rather freely, as long as it fulfills some regularity assumptions. However, in the much more delicate context considered in Theorem \ref{cor: inv dens 1} (we investigate the $\sup$-norm risk in a multidimensional situation), we specify both bandwidths explicitly. 
		The reason for this is the bound obtained in Proposition \ref{th: inv density estimation general}, respectively Assumption \eqref{ass: bandwidth speed}, which are also the reasons for assuming $\beta_1>1$ and $\beta_2>2$. 
		In fact, these assumptions on $\beta_1$ and $\beta_2$ can be relaxed slightly in some cases, but for the sake of brevity and since this does not offer much further insight, we decided to omit this result.
	\end{enumerate}
\end{remark}

We continue by providing interpretations of the functions $\Psi,\Upsilon$ and the set $\cB$ introduced in \eqref{def:Ups}, \eqref{def:Psi} and \eqref{def:B}, respectively, and explaining our reasons for this particular form of notation. 

\begin{remark}
	In the classical setting of $n$ $d$-dimensional, i.i.d.~observations, the minimax optimal convergence rate for the $\sup$-norm risk, given the estimated density belongs to an isotropic H\"older class with smoothness $\beta$, is given by $(\log n/n )^{\beta/(2\beta+d)},$ where the logarithmic term in the convergence rate stems from investigating the $\sup$-norm risk.
	Thus, for our specific problem, {\color{nd}an analogous rate }would be of the form $(\log T/T )^{\beta/(2\beta+2d)}$ (recall that $\Z$ is $2d$-dimensional), where the smoothness index $\beta$ is replaced by the harmonic mean of the smoothness indices in the anisotropic framework. 
	Note now that the function $\Psi$ with $\Upsilon\equiv 0$ corresponds to this classical nonparametric rate of convergence. 
	However, it has already been observed that this rate can be improved for invariant density estimation of diffusion-type processes when \emph{continuous} observations are available, corresponding to $\Upsilon$ being strictly positive in our notation. 
	Specifically, we refer, e.g., to Corollary 1 in \cite{dalalyan07} for a result on the convergence rate of the pointwise risk in the continuous diffusion context, Theorem 3.4 in \cite{strauch18}, which concerns the rate of convergence of the $\sup$-norm risk for an adaptive estimator of the invariant density of a continuous diffusion under anisotropic H\"older assumptions, or Theorem 4.3 in \cite{dexheimer2020mixing}, which bounds the rate with respect to the $\sup$-norm risk for a more general class of exponentially $\beta$-mixing Markov processes. 
	In all these cases, the rate of convergence is essentially given by $\Psi$ with $\Upsilon\equiv 2.$ 
	In particular, contrary to our result, $\Upsilon$ does not depend on the dimension, the smoothness indices or any other entities, especially not even in the anisotropic framework considered in \cite{strauch18}. 
	For a summary of the mentioned results in this paragraph see the following table, which contains the polynomic rates of convergence:
	\begin{center}
		\begin{tabular}{l|c|c}
			& (invariant) density & drift vector\\\hline
			nondegenerate diffusion & $\frac{\overline{\beta}}{2(\overline\beta+d)-2}$ & $\frac{\overline{\beta}}{2(\overline\beta+d)}$\\
			kinetic diffusion &  $\frac{\overline{\beta}}{2(\overline\beta+d)-\Upsilon(\beta_1,\beta_2,d,\ep)}$ &   $\frac{\overline{\beta}}{2(\overline\beta+d)}$\\
			i.i.d.~case & $\frac{\overline{\beta}}{2(\overline\beta+d)}$ & -
		\end{tabular}
	\end{center}
\end{remark}

As can be seen in the proof of the variance bounds in Propositions \ref{prop: variance not bounded} and \ref{prop: variance bounded}, which are the quintessential reason for our results, the particular form of $\Upsilon$ in our case is caused by the heat kernel bound in \hyperref[framework]{$\cA$}. More specifically, the function $p_t^G$ suggests that the variances of the processes $\X$ and $\Y$ are of a different order. 
To illustrate this further, we refer to the following example taken from \cite{cattiaux2014}. 
\begin{example}[Example 2.9 in \cite{cattiaux2014}]
	Let $d=1$ and $c=V=0$. 
	Then, $Z_t$ is a two-dimensional Gaussian vector with 
	$$ \E[X_t]=x_0+y_0t,\quad \E[Y_t]=y_0,$$ and
	$$\Var(X_t)= \frac{t^3}{3},\quad \Var(Y_t)=t,\quad \Cov(X_t,Y_t)=\frac{t^2}{2}.$$ 
\end{example}

\section{Drift estimation}\label{sec: drift}
We now turn to the question of proposing a nonparametric estimator of the drift function appearing in \eqref{eq: sde}, specified as $b(x,y)= -(c(x,y)y+\nabla V(x))$, $x,y\in\R^d$.
Throughout this entire section, we will assume $b$ to be locally bounded and $\sigma$ to be uniformly bounded. 
Note that these assumptions are satisfied under $\tilde{\cA}$.

Given two bounded kernel functions $K_1,K_2\colon\R^d\to\R$ with compact support, $x,y\in\R^d$ and $j\in\{1,\ldots,d\}$, set
\begin{equation*}
	\overline{b}_{j,h_1,h_2,T}(x,y)\coloneqq \frac1T\int_0^T K_{h_1,h_2}(x-X_u,y-Y_u)\d Y^j_u, \quad\text{ where }\quad K_{h_1,h_2}\coloneqq(h_1h_2)^{-d}K_1\left(\frac x h_1\right)K_2\left( \frac y h_2 \right).
\end{equation*}
For some strictly positive $r_T\in \co(1)$, an estimator of the $j$-th component of the drift vector $b$ is then given by a Nadaraya--Watson-type estimator of the form
\begin{equation}\label{def:drift}
	\widehat{b}_{j,\h,T,r_T}\coloneqq \frac{\overline{b}_{j,h_1,h_2,T}}{\vert \hat{\rho}_{h^{(\rho)}_1,h^{(\rho)}_2,T}\vert+r_T},\quad x,y\in\R^d, \h\coloneqq(h_1,h_2,h^{(\rho)}_1,h^{(\rho)}_2).
\end{equation}
Note that our bounds on the rate of convergence of the estimator $\hat{\rho}$ stated in the previous section continue to hold for $\vert \hat{\rho}\vert$.
Strict positivity of $r_T$ ensures that the drift estimator $\hat b$ introduced in \eqref{def:drift} is well-defined. 
However, since $\overline b$ is defined via some stochastic integral, a crucial point for deriving upper bounds on the rate of convergence of this estimator will be uniform moment bounds of stochastic integrals with respect to $\Y$ over countable classes of bounded functions. 
This will be the main focus of the subsequent section. 
{\color{nd} In principle, all the applied techniques would also be suitable for the estimation of a drift function $b,$ which is not in the specified form. 
	However, as the results of \cite{wu2001}, which in particular imply the exponential $\beta$-mixing property, only consider such drifts, we focus our analysis on this case.}

\subsection{Uniform moment bounds}\label{subsec: unif}
In Section 3 of \cite{dexheimer2020mixing}, uniform moment bounds over countable classes of bounded functions $\cG$ were derived for suprema of functionals of the form
\[
\sup_{g\in\cG}\vert\G_t(g)\vert,\quad \textrm{where} \quad \G_t(g)\coloneqq \frac{1}{\sqrt{t}}\int_0^t g(X_s)\d s,\ g\in L^2_0(\mu),
\]
under the assumption of $\X$ being exponentially $\beta$-mixing. 
\textcolor{dex}{
	For the reader's convenience, we start this section with a reminder of the relevant results.
	As the bounds are derived via an application of the generic chaining device based on \cite{Dirksen2015}, they are stated in terms of covering numbers, so recall that, for any given $\ep>0$, the covering number $\mathcal N(\ep,\mathcal G,d)$ of $\mathcal G$ denotes the smallest number of balls of $d$-radius $\ep$ needed to cover $\mathcal G$. 
	Furthermore, given $f,g\in\mathcal G$, we define the following semi-metrics,
	\begin{equation}\begin{split}\label{def:dG}
			d_\infty(f,g)&\coloneqq \|f-g\|_\infty,\quad d^p_{L^p(\mu)}(f,g)\coloneqq \mu((f-g)^p)\quad p\ge 1,\\
			d_{\G,t}^2(f,g)&\coloneqq\Var\left(\frac{1}{\sqrt t}\int_0^{t}(f-g)(X_s)\d s\right),
	\end{split}\end{equation}
	with $\X$ being the Markov process in Theorem \ref{th: bernstein mix}.
}

\textcolor{dex}{\begin{theorem}[Theorem 3.2 in \cite{dexheimer2020mixing}]\label{th: bernstein mix}
		Suppose that $\X$ is an exponentially $\beta$-mixing Markov process. 
		Let $\mathcal{G}$ be a countable class of bounded real-valued functions with $\mu(g)=0$, and let $m_t\in [0,t/4)$. 
		Then, there exist $\tau\in [m_t,2m_t]$ and constants $\tilde{C}_1,\tilde{C}_2>0$ such that, for any $1\leq p <\infty$,
		\begin{align*}
			\left(\E\left[\sup_{g\in\mathcal{G}} \vert \G_t(g)\vert^p\right]\right)^{1/p}&\leq
			\tilde{C}_1\int_0^\infty \log \mathcal{N}(u,\mathcal{G},\frac{2m_t}{\sqrt{t}}d_\infty)\d u+ \tilde{C}_2\int_0^\infty \sqrt{\log \mathcal{N}(u,\mathcal{G},d_{\G,\tau}) }\d u
			\\
			&\qquad+ 4\sup_{g\in\mathcal{G}}\left(\frac{2m_t}{\sqrt{t}}\Vert g\Vert_\infty \tilde{c}_1p+ \Vert g\Vert_{\G,\tau}\tilde{c}_2\sqrt{p}+\frac{1}{2}\Vert g\Vert_\infty c_\kappa \sqrt{t}\e^{-\frac{\kappa m_t}{p}} \right),
		\end{align*}
		where $\tilde{c}_1,\tilde{c}_2$ are positive constants, defined in equation (B.3) of \cite{dexheimer2020mixing}, and $c_\kappa,\kappa>0$ are specified in Assumption ($\mathcal{A}\beta$) of \cite{dexheimer2020mixing}.
	\end{theorem}
}

One of the main tools in the derivation of this result was the following Bernstein-type concentration inequality.
\begin{lemma}[Lemma 3.1 in \cite{dexheimer2020mixing}]\label{Bernstein1}
	Suppose that $\X$ is an exponentially $\beta$-mixing Markov process, and let $g$ be a bounded, measurable function fulfilling $\mu(g)=0$.
	Then, for any $t,u>0$ and $m_t\in(0,\tfrac{t}{4}]$, there exists $\tau\in [m_t,2m_t]$ such that
	\begin{align*}
		\Pro\left(\frac{1}{\sqrt{t}}\int_0^t g(X_s)\d s> u \right)
		\le\,&2\exp\left(-\frac{u^2}{32\big(\Var\big(\tfrac{1}{\sqrt{\tau}}\int_0^{\tau}g(X_s)\d s\big)+2u\Vert g\Vert_\infty \tfrac{m_t}{\sqrt{t}}\big)}\right)\\&\quad +\frac{t}{m_t} c_\kappa\e^{-\kappa m_t}\mathds{1}_{(0,4\sqrt{t}\Vert g\Vert_\infty)}(u).
	\end{align*}
\end{lemma}
\textcolor{dex}{
	The term $m_t$ in the previous statements arises from the use of the classical Bernstein inequality for independent random variables in the proofs. 
	It can thus be interpreted as a kind of loss compared to results concerning i.i.d.~random variables. 
	However, since $\X$ is exponentially $\beta$-mixing, the decay in $m_t$ is exponentially fast, resulting in this loss being almost negligible. 
	In fact, $m_t$ can be viewed as a tuning parameter in our concentration results, with the typical choice being $m_t=c\log t$, where $c>0$ is suitably large so that the error term decays with an adequately fast polynomial rate.
}

We will extend Theorem \ref{th: bernstein mix} in our specific framework to functionals of the form
\begin{equation}\label{eq:mathbbH}
	\sup_{g\in\GG} \vert \mathbb{H}^j_t(g)\vert, \quad \textrm{where}\quad\mathbb{H}^j_t(g)\coloneqq\frac{1}{\sqrt{t}}\int_0^t g(X_s,Y_s)\d Y^j_s, g\in L^2_0(\mu),j\in\{1,\ldots,d\}.
\end{equation}

Once again, a Bernstein-type concentration inequality will play a vital role in our proofs, namely the Bernstein inequality for continuous martingales (see, e.g., p.~153 in \cite{revuzyor1999}).
Given a continuous local martingale $(M_t)_{t\ge0}$, it states that
\begin{equation}\label{inequ: Bernstein martingal}
	\forall t,x,y>0,\quad	\Pro\left(M_t\geq x, \qv{M}_t\leq y\right)\leq \exp\left(-\frac{x^2}{2y}\right),
\end{equation}
where $(\qv{M}_t)_{t\geq0}$ denotes the quadratic variation process of $(M_t)_{t\geq0}$. 
Combining the concentration inequalities in Lemma \ref{Bernstein1} and equation \eqref{inequ: Bernstein martingal} will enable us to again employ the generic chaining device for the derivation of the required uniform moment bounds, similar to \cite{Dirksen2015}.
	They are stated in terms of entropy integrals with respect to the semi-metrics defined in \eqref{def:dG}, with the semi-metric $d_{\G}$ induced by the variance of the integral functional now being specified as
	\[
	d_{\G,t}^2(f,g)\coloneqq\sigma_t^2(f-g),\text{ where } \sigma_t^2(f)\coloneqq\Var\left(\frac{1}{\sqrt t}\int_0^{t}f(X_s,Y_s)\d s\right).
	\]
The main result of this section is the following theorem.

\begin{theorem}\label{th: Bernstein drift}
	Assume \hyperref[framework]{$\cA$}, let $\mathcal G$ be a countable class of bounded real-valued functions such that $gb$ is bounded for all $g\in\mathcal{G}$, and let $m_t,\tilde m_t \in (0, t\slash 4]$.
	Then, there exist $\tau \in [m_t, 2m_t],\tilde\tau \in [\tilde m_t, 2\tilde m_t]$ and a constant $c>0$ such that, for large enough $t$, any $1\le p<\infty$ and $j\in\{1,\ldots,d\}$,
	
	\begin{equation*} 
		\begin{split}%\label{eq: unimom drift}
			\left(\E\left[\sup_{g\in\mathcal G}|\mathbb{H}^j_t(g)-\sqrt t \mu(gb^j)|^p\right]\right)^{1/p}
			&\leq c \Bigg(\int_0^\infty\log\mathcal N\big(u,\mathcal Gb^j,\tfrac{m_t}{\sqrt{t}}d_\infty\big)\d u+\int_0^\infty\sqrt{\log \mathcal N(u,\mathcal Gb^j,d_{\G,\tau})}\diff u\\
			&\qquad\textcolor{dex}{+\int_0^\infty\log\mathcal N\big(u,\mathcal G,t^{-1/4}d_\infty+t^{-1/8}d_{L^4(\mu)}\big)\d u}\\
			&\qquad\textcolor{dex}{+\int_0^\infty\sqrt{\log \mathcal N(u,\mathcal G,d_{L^2(\mu)}+t^{-1/8}d_{L^4(\mu)})}\diff u}\\
			&\qquad+\sup_{g\in\mathcal G}\Big(\frac{m_t}{\sqrt{t}}\|g\|_\infty p+ \lVert g\rVert_{\mathbb{G},\tau}\tilde c_2\sqrt p +\frac{1}{2}\lVert g \rVert_{\infty} \sqrt{t} \mathrm{e}^{-\frac{\kappa m_t}{p}} \\
			&\qquad+p\sqrt{\frac{\tilde m_t\Vert a_{jj}\Vert_\infty}{t}}\Vert g\Vert_\infty\textcolor{dex}{+p^{3/4}(\tilde{\tau}/t)^{1/4}\Vert g\Vert_{L^4(\mu)}}\\
			&\qquad + \sqrt{p\Vert a_{jj}\Vert_\infty} \Vert g\Vert_\infty \e^{-\tfrac{\kappa \tilde m_t}{2p}}+\sqrt{p\Vert a_{jj}\Vert_\infty}\Vert g\Vert_{L^2(\mu)}\Big)\Bigg).
		\end{split}
	\end{equation*}
\end{theorem}
{\color{nd} The form of the upper bound in Theorem \ref{th: Bernstein drift} reflects the result obtained in Theorem 3.5 of \cite{Dirksen2015}, where the generic chaining device is applied on stochastic processes with a mixed tail behaviour. As we perform the chaining procedure twice using two different concentration inequalities, it is not surprising that the obtained result contains four different entropy integrals, each concerning a different distance.  }
\subsection{Rate of convergence}\label{subsec: drift rate}
Applying the powerful result stated in Theorem \ref{th: Bernstein drift} together with bounds on the involved covering numbers already yields a bound on the rate of convergence of the $\sup$-norm risk of the estimator $\overline{b}_{j,h_1,h_2,T}$ of $b^j\rho$, for adequately chosen bandwidths $h_1,h_2$.
%and kernels $K_1,K_2$.

\begin{proposition}\label{theorem: drift}
	Let $D\subset \R^{2d}$ be an open and bounded set, assume \hyperref[framework]{$\cA$}, \textcolor{dex}{and let $h_1,h_2\in \mathcal{H}$ such that $(h_1h_2)^d\geq T^{-\frac{1}{2}}\log(h_1^{-1}+h_2^{-1})$.
		Then, for any $\gamma>0$, there exists a constant $c_\gamma$ such that, for any $1\leq p\leq \gamma\log T$,} it holds for large enough $T$
	\[
	\mathcal{R}^{(p)}_{\infty}\big(\overline{b}_{j,h_1,h_2,T}, b^j\rho; D\big)\leq \mathcal B_{b^j\rho}(h_1,h_2)+\textcolor{dex}{c_\gamma}(h_1h_2)^{-d/2}T^{-1/2}\sqrt{\log(h_1^{-1}+h_2^{-1})}.
	\]
\end{proposition}

%\begin{lemma}\label{theorem: drift}
%Let $D\subset \R^{2d}$ be an open and bounded set, assume \hyperref[framework]{$\cA$} and $b^j\rho\in \mathcal{H}_D(\beta_3,\beta_4, \cL_3,\cL_4)$. 
%Then, for any $p \geq 1$, choosing
%\begin{equation}\label{band: drift}
%	h_3\simeq \left(\frac{\log T}{T} \right)^{\frac{\overline{\beta}}{2\beta_3(\overline{\beta}+d)}},\quad h_4\simeq\left(\frac{\log T}{T} \right)^{\frac{\overline{\beta}}{2\beta_4(\overline{\beta}+d)}},
%	%\quad \textrm{where }\overline{\beta}\coloneqq 2\left(\tfrac 1 \beta_3+\tfrac 1 \beta_4\right)^{-1},
%\end{equation}
%yields, if $K_3,K_4$ are kernels of order $\llfloor\beta_3\rrfloor,\llfloor \beta_4\rrfloor$,
%\[
%	\mathcal{R}^{(p)}_{\infty}\big(\overline{b}_{j,h_3,h_4,T}, b^j\rho; D\big) \in \mathcal{O}\left(\left(\frac{\log T}{T}\right)^{\frac{\overline{\beta}}{2(\overline{\beta}+d)}} \right).
%\]
%\end{lemma}

Combining Theorem \ref{cor: inv dens 1} with Proposition \ref{theorem: drift} and a specific choice of $r_T$ then yields our next main result, an upper bound on the rate of convergence for a weighted version of the $\sup$-norm risk of the drift estimator $\hat b$ introduced in \eqref{def:drift}.

\begin{theorem}\label{thm: rate drift}
	Let $D\subset \R^{2d}$ be a bounded, open set, fix $j\in\{1,\ldots,d\}$, assume \hyperref[framework]{$\cA$}, $b^j\rho,\rho\in\cH_D(\beta_1,\beta_2,\cL_1,\cL_2)$, and set
	\[
	r_T\coloneqq (\Psi\chi_\cB)(\beta_1,\beta_2,d,0) \exp\left(\sqrt{\log T}\right).
	\]  
	For defining the drift estimator, choose $K_1,K_2,h^{(\rho)}_1,h^{(\rho)}_2$ as in Theorem \ref{cor: inv dens 1}(a), and specify
	\begin{equation}\label{band: drift}
		h_i\sim \left(\frac{\log T}{T}\right)^{\frac{\bar{\beta}}{2\beta_i(\bar{\beta}+d)}},\quad i=1,2. 
	\end{equation}
	Then, if $\overline{\beta}>d$, $\beta_1>1$, $\beta_2>2$, it holds
	\[
	\E\left[\Vert (\hat{b}_{j,\h,T,r_T}(z)-b^j(z))\rho(z)\Vert_{L^\infty(D)}\right]\in \cO\left(\left(\frac{\log T}{T} \right)^{\frac{\overline{\beta}}{2(\overline{\beta}+d)}}\right).
	\]
\end{theorem}
The entire proof can again be found in the appendix. 
One may wonder why we arrive at the classical nonparametric rate of convergence even though highly nonclassical results were obtained in Corollary \ref{cor: inv dens 1}. 
The technical reason for this is the occurrence of the covering number with respect to $d_{L^2(\mu)}$ in Theorem \ref{th: Bernstein drift}. 
Contrary to the approach taken to find the variance bounds in Propositions \ref{prop: variance not bounded} and \ref{prop: variance bounded}, we cannot use the exponential $\beta$-mixing property of $\Z$ to bound this, and since the heat kernel bound in \ref{ass: heat kernel bound} only applies to values of $t$ in $(0,1]$, it cannot be used either. 
The fact that a faster convergence rate for the invariant density estimate is not equivalent to a faster convergence rate for the drift estimate is well-known. 
In particular, it has been shown in some cases that the classical nonparametric convergence rate is optimal in the minimax sense (see, e.g., \cite{str15,str16}).

\subsection{Adaptive estimation scheme} \label{subsection: adap}
We now address the question of finding a data-driven approach to drift estimation. 
Our interest is in bounding the $\sup$-norm risk $\mathcal R_\infty^{(p)}(\hat b_{j,\h,T},b^j;D)$, $1\le p<\infty$, of (the components of) the drift estimator $\hat b_{j,\h,T}$ over an open and bounded set $D\subset\R^{2d}$. 
The bandwidths specified in Theorem \ref{thm: rate drift} (see \eqref{band: drift}) clearly depend on the typically unknown smoothness of $b^j\rho$.

For defining an adaptive drift estimator which relies on bandwidths specified in a data-driven way, consider some symmetric, Lipschitz continuous kernel functions $K_1,K_2\colon \R^d\to\R$ of order $\ell_1,\ell_2$ fulfilling
\[
\int K_i\d\lebesgue=1,\quad \|K_i\|_\infty<\infty\quad\text{ and }\quad \supp(K_i)\subset[-1/2,1/2]^d,\quad i=1,2.
\]
For any bandwidths $(h_1,h_2)^\top, (\eta_1,\eta_2)^\top\in(0,1]^2$ and any points $x,y\in\R^{d}$, denote
\begin{equation*}\begin{split}
		\left(K_{h_1,h_2}\star K_{\eta_1,\eta_2}\right)(x,y)&\coloneqq \left(K_{h_1}\ast K_{\eta_1}\right)(x)\cdot\left(K_{h_2}\ast K_{\eta_2}\right)(y)\\
		&\ =\int_{\R^d}K_{h_1}(u-x)K_{\eta_1}(u)\d u\int_{\R^d}K_{h_2}(u-y)K_{\eta_2}(u)\d u.
\end{split}\end{equation*}
For $x,y\in\R^d$, $j\in\{1,\ldots,d\}$, define the kernel estimators
\begin{align*}
	\overline b_{j,h_1,h_2,t}(x,y)=\overline b_{j,\h}(x,y)&\coloneqq \frac1t\int_0^tK_{h_1,h_2}(x-X_s,y-Y_s)\d Y_s^j,\\
	\overline b_{j,h_1,h_2,\eta_1,\eta_2}(x,y)=\overline b_{j,\h,\Beta}(x,y)&\coloneqq \frac1t\int_0^t\left(K_{h_1,h_2}\star K_{\eta_1,\eta_2}\right)(X_s-x,Y_s-y)\d Y^j_s.
\end{align*}
Specify the set $\mathcal H_t$ of candidate bandwidths as
\[
\mathcal H_t\coloneqq \left\{\h=(h_1,h_2)^\top\in(0,1]^2: h_i=\eta^{-k_i}\text{ with } k_i\in\N_0, \textcolor{dex}{\eta^{d(k_1+k_2)}\leq t^{\frac{1}{2}}\log(\eta^{k_1}+\eta^{k_2})^{-1}}\right\},\quad \eta>1\text{ arbitrary},
\]
choose $q\geq 1$, and let
\[
\hat\Delta_t^j(\h)\coloneqq \sup_{\boldsymbol\eta=(\eta_1,\eta_2)\in\mathcal H_t}\left\{\left[\|\overline b_{j,\h,\Beta}-\overline b_{j,\Beta}\|_\infty -A^{(q)}_t(\eta_1,\eta_2)\right]_+\right\},\]
for
\begin{equation}\label{def:At}
	A^{(q)}_t(\Beta)=A^{(q)}_t(\eta_1,\eta_2)\coloneqq \e\sqrt{32d\Vert a_{jj}\Vert_\infty\Vert \rho\Vert_\infty}\left(\tilde C_1\sqrt{192} \Vert K\Vert_\infty
	+\tilde C_2\sqrt{q }\Vert K\Vert_{L^2(\lebesgue)}\right)\sqrt{\frac{\log(\eta_1^{-1}+\eta_2^{-1})}{t(\eta_1\eta_2)^{d}}},
\end{equation}
{\color{nd} where the constants $\tilde{C}_1,\tilde{C}_2$ are specified in the proof of Theorem \ref{th: Bernstein drift} in Appendix B. }
Finally, define $\hat{\h}^j=(\hat h_1^j,\hat h_2^j)$ by setting
\[
\hat\Delta_t^j(\hat{\h}^j)+A^{(q)}_t(\hat{\h}^j)=\inf_{\h=(h_1,h_2)\in\mathcal H_t}\left\{\hat\Delta_t^j(\h)+A^{(q)}_t(\h)\right\}.\]
{\color{nd} Our approach is based on the work of \cite{lepski2013}, which itself relies on ideas developed in \cite{goldenshluger2011}. 
	Intuitively speaking, we make use of the classical decomposition of the error into a bias term and a stochastic error by approximating it through the bias proxy $\hat\Delta_t$ and $A_t^{(q)},$ which mimics the stochastic error (see Proposition \ref{theorem: drift}).  }

\begin{proposition}\label{thm:adapdrift}
	Let $D\subset \R^{2d}$ be an open and bounded set, assume \hyperref[framework]{$\cA$}, and let $K_1,K_2\colon\R^d\to\R$ be symmetric, Lipschitz continuous kernel functions.
	Then, there exists a constant $c$ such that, for any $t>0$ sufficiently large,
	\begin{equation}\label{upperbound:numerator}
		\mathcal R_\infty^{(p)}(\overline b_{j,\hat{\h}},b^j\rho;D)\le c\left(\mathfrak R_t(b^j\rho)+(\log t)^{2/q+1/2}t^{-1/2}\right),\quad \forall 1\leq p\leq q,
	\end{equation}
	where
	\[
	\mathfrak R_t(b^j\rho)\coloneqq \inf_{\h=(h_1,h_2)\in\mathcal H_t}\left\{\mathcal B_{b^j\rho}(\h)+(h_1h_2)^{-d/2}\sqrt{\frac{\log(h_1^{-1}+h_2^{-1})}{t}}\right\}.
	\]
\end{proposition}

In the last step of our investigation, we transfer the above result to a finding on the original question of drift estimation. 
Given a bounded, open set $D\subset\R^{2d}$, denote by $\rho_\star>0$ an a priori lower bound on the invariant density fulfilling $\inf_{(x,y)\in D}\rho(x,y)\ge \rho_\star$.
Similarly to \eqref{def:drift}, define
\[\hat b_{j,h_1,h_2,t}=\hat b_{j,\h,t}\coloneqq \frac{\overline b_{j,h_1,h_2,t}}{\hat\rho_{h_1,h_2,t}\vee\rho_\star}=
\frac{\overline b_{j,\h,t}}{\hat\rho_{\h,t}\vee\rho_\star},\quad \h=(h_1,h_2)^\top\in \overline{\mathcal H}_t\coloneqq \mathcal H_t\cap \mathcal H(Q_1,Q_2).\]

\begin{theorem}\label{thm:adapdrift2}
	Grant the assumptions of Proposition \ref{thm:adapdrift}, and assume, in addition, that\\ $\rho,b^j\rho\in \mathcal H_D(\beta_1,\beta_2,\mathcal L_1,\mathcal L_2)$ with $\overline\beta>d, \beta_1\leq \ell_1,\beta_2\leq \ell_2.$
	Defining the bandwidth $\hat{\h}=\hat{\h}^j$ via
	\[\hat\Delta_t^j(\hat{\h}^j)+A^{(q)}_t(\hat{\h}^j)=\inf_{\h=(h_1,h_2)\in\overline{\mathcal H}_t}\left\{\hat\Delta_t^j(\h)+A^{(q)}_t(\h)\right\}\]
	then yields
	\begin{equation}\label{upperbounddrivec}
		\mathcal R_\infty^{(p)}(\hat b_{j,\hat{\h}^j,t},b^j;D)\in\mathcal O\left(\left(\log t/t\right)^{\frac{\overline\beta}{2(\overline\beta+d)}}\right),\quad \forall 1\leq p\leq q.
	\end{equation}
\end{theorem}
%\begin{theorem}\label{thm:adapdrift2}
%Grant the assumptions of Theorem \ref{thm:adapdrift}, and assume, in addition, that, for some integer $\mathfrak b\ge 2$, {\color{purple}(Bedingung an Ordnung des Kerns)}.
%Define the bandwidth $\hat{\h}=\hat{\h}^j$ via
%\[\hat\Delta_t^j(\hat{\h}^j)+A^{(q)}_t(\hat{\h}^j)=\inf_{\h=(h_1,h_2)\in\overline{\mathcal H}_t}\left\{\hat\Delta_t^j(\h)+A^{(q)}_t(\h)\right\}.\]
%Then, if $\rho,b^j\rho\in \mathcal H_D(\beta_1,\beta_2,\mathcal L_1,\mathcal L_2)$ with $\overline\beta>d$, 
%\begin{equation}\label{upperbounddrivec}
%\mathcal R_\infty^{(p)}(\hat b_{j,\hat{\h}^j,t},b^j;D)\in\mathcal O\left(\left(\log t/t\right)^{\frac{\overline\beta}{2(\overline\beta+d)}}\right),\quad \forall 1\leq p\leq q.
%\end{equation}
%\end{theorem}

\appendix
\section{Proofs for Section \ref{sec: inv dens}}
We will require the following auxiliary result for the proof of the variance bounds in Propositions \ref{prop: variance not bounded} and \ref{prop: variance bounded}.
\renewcommand{\thesection} {\Alph{section}}
\begin{lemma}\label{lemma: semi group bound}
	Let $D$ be a bounded subset of $\R^{2d}$. 
	Then, there exists a constant $c_D>0$ such that, for all $0<\upsilon<1,\upsilon\leq t$, for all $z\in\R^{2d}$ and any bounded measurable function $f$ with support $D$,
	\[
	\vert P_t(f)(z)\vert\leq c_D\left(\frac{\Vert f\Vert_{L^1(\R^{2d})}}{\upsilon^{2d}}+\Vert f\Vert_\infty \exp\left(-\frac{1}{c_D \upsilon } \right) \right).
	\]
\end{lemma}
\begin{proof}
	We start by proving the assertion for $\upsilon=t$ and $z\in\tilde{D},$ where $\tilde{D}\coloneqq\{z\in\R^{2d}: d(z,D)\leq 1\}$, with $d(z,D)\coloneqq \inf_{x\in D}\Vert z-x\Vert$ denoting the distance of the point $z$ to the set $D$. 
	Since $\upsilon<1$, \ref{ass: heat kernel bound} implies that there exists a constant $c$ depending on $\tilde{D}$ such that 
	\[
	\vert P_\upsilon(f)(z)\vert\leq \int\vert f(z')\vert p_\upsilon (z,z')\d z'\leq c\upsilon^{-2d} \int\vert f(z')\vert \d z'+c \Vert f\Vert_\infty \exp\left(-\frac{1}{c\upsilon } \right),
	\]
	thus completing the proof in this case. 
	For analysing the case $z\notin\tilde{D}$, introduce the first hitting time of $\tilde{D}$, defined as $\tau_{\tilde{D}}\coloneqq\inf\{t\geq0:Z_t\in\tilde{D}\}$. 
	Continuity and the strong Markov property of $\Z$ imply
	\begin{align*}
		P_\upsilon(f)(z)=\E_z\left[f(Z_\upsilon) \1_{[0,\upsilon]}(\tau_{\tilde{D}})\right]=\E_z\left[P_{\upsilon-\tau_{\tilde{D}}}(f)(Z_{\tau_{\tilde{D}}}) \1_{[0,\upsilon]}(\tau_{\tilde{D}})\right],
	\end{align*}
	and $\upsilon-\tau_{\tilde{D}}\in (0,\upsilon)$, $d(Z_{\tau_{\tilde{D}}},D)=1$. 
	Thus, it is enough to find a bound for $\vert P_s(f)(z')\vert$ such that $s\in(0,\upsilon)$, $z' \in \tilde{D}: d(z',D)=1$.
	Assumption \ref{ass: heat kernel bound} now implies for this case
	\[
	\vert P_s(f)(z')\vert\le\int \vert f(w)\vert p_s(z',w)\d w\leq c \left(\mathcal{P}\int \vert f(w)\vert \d w+\Vert f\Vert_\infty \exp\left(-\frac{1}{cD} \right) \right),
	\]
	where
	\[
	\mathcal{P}\coloneqq \sup\left\{s^{-2d}\exp\left(-c^{-1}\left(\frac{\Vert z_2'-w_2\Vert^2}{4s}+\frac{3\Vert w_1-z_1'-\frac{s(z_2'+w_2)}{2}\Vert^2 }{s^3} \right)\right):\begin{array}{l}
		s\in(0,\upsilon), (w_1,w_2)\in D,\\
		(z_1',z_2') \in \tilde{D}: d(z',D)=1 
	\end{array}\right\}
	\]
	and the constant $c$ only depends on $\tilde{D}$. 
	To show that $\mathcal{P}$ is finite, fix $w=(w_1,w_2)\in D$, $s\in(0,1)$, $z'=(z_1',z_2') \in \tilde{D}: d(z',D)=1$. 
	Then, the reverse triangle inequality and the inequality $(A-B)^2\geq A^2\frac{s}{s+1}-B^2s$, valid for any $A,B\in\R$, imply
	\begin{align*}
		\frac{\Vert z_2'-w_2\Vert^2}{4s}+\frac{3\Vert w_1-z_1'-\frac{s(z_2'+w_2)}{2}\Vert^2 }{s^3}&\geq\frac{1}{4}\left(\frac{\Vert z_2'-w_2\Vert^2}{s}+\frac{(\Vert w_1-z_1'\Vert-\Vert\tfrac{s(z_2'+w_2)}{2}\Vert)^2}{s^3}\right)
		\\
		&\geq\frac{1}{4}\left(\frac{\Vert z_2'-w_2\Vert^2}{s}+\frac{\Vert w_1-z_1'\Vert^2}{s^2(s+1)}-\frac{\Vert\tfrac{s(z_2'+w_2)}{2}\Vert^2s}{s^3}\right)
		\\
		&\geq\frac{1}{4}\left(\frac{\Vert z_2'-w_2\Vert^2+\Vert w_1-z_1'\Vert^2}{2s}-c_2\right)=\frac{1}{4}\left(\frac{\Vert w-z'\Vert^2}{2s}-c_2\right)
		\\
		&\geq\frac{1}{8s}-\frac{c_2}{4},
	\end{align*}
	where $c_2$ denotes some uniform bound of $\Vert z_2'+w_2\Vert$ which is finite because $D$ is bounded. 
	Thus, $\mathcal{P}$ is indeed bounded by a finite constant (depending on $D$ and $d$) and hence the assertion also follows in this case because $\upsilon<1$. 
	For the case $\upsilon<t$, we have
	$$\vert P_t(f)(z)\vert\leq \int p_{t-\upsilon}(z,z')\vert P_D(f)(z')\vert\d z'.$$ 
	Thus, the assertion follows by the bound derived above. This completes the proof.
\end{proof}

\begin{proof}[Proofs of Propositions \ref{prop: variance not bounded} and \ref{prop: variance bounded}]
	Throughout the whole proof, we will suppress the dependence of functions on $T$ for notational convenience.
	
	\paragraph{Proof of \eqref{eq: var th 2}}
	We start with the following well-known bound of the variance functional
	\begin{align}
		\Var\left(\frac 1 T \int_0^T f(X_s,Y_s)\d s \right)
		\leq\frac 2 T \int_0^T \underbrace{\vert\Cov(f(X_s,Y_s),f(X_0,Y_0))\vert}_{\eqqcolon \cC(s)}\d s.\label{eq: variance start}
	\end{align}
	We proceed by splitting the integral, using integral bounds $0\leq\delta_0\leq \delta\leq D_1\leq D_2\leq T$. 
	Note that stationarity of $\Z$, boundedness of $\rho$ and the Cauchy--Schwarz inequality imply
	\[\cC(s)\leq\Var(f(X_0,Y_0) )\leq \E\left[f^2(X_0,Y_0)\right]\leq c\Vert f\Vert_\infty^2\lebesgue(\S),\]
	for some suitable constant $c>0,$ where $\cS\coloneqq\supp(f)$.
	Hence,
	\begin{equation}\label{eq: variance 0.5}
		\int_0^{\delta_0} \cC(s)\d s\leq c\delta_0\Vert f\Vert_\infty^2\lebesgue(\S).
	\end{equation}
	\textcolor{dex}{
		Furthermore, the heat kernel bound \ref{ass: heat kernel bound} and boundedness of $\rho$ imply for $\delta_0<s<\delta<1$
		\begin{align*}
			\notag	\cC(s)&\leq \E[f(X_s,Y_s)f(X_0,Y_0)]+\E[f(X_0,Y_0)]^2
			\\&\notag\le c\int_{\R^{4d}}\vert f(x',y')\vert \vert f(x'',y'')\vert\1_{\S}(x',y')\1_{\S}(x'',y'')\\
			&\notag\hspace*{3em}\times s^{-2d} \exp\left(-c^{-1}\left(\frac{\Vert y'-y''\Vert^2}{4s}+\frac{3\Vert\notag x''-x'-\frac{s(y'+y'')}{2}\Vert^2 }{s^3} \right) \right) \d x''\d y''\rho(x',y')\d x'\d y'\\
			&\notag\hspace*{3em}+c\Vert f\Vert_\infty^2\lebesgue(\S) \exp\left(-(cs)^{-1}\right)+c\Vert f\Vert_\infty^2 \lebesgue^2(\S)\\
			&\notag\le c\Vert f\Vert^2_\infty\int_{\R^{2d}}\1_{\bigcup_{i=1}^nB(x_i,\mathsf{s}_1)}(x')\1_{\bigcup_{i=1}^nB(x_i,\mathsf{s}_1)}(x'')\1_{\bigcup_{i=1}^nB(y_i,\mathsf{s}_2)}(y')\1_{\bigcup_{i=1}^nB(y_i,\mathsf{s}_2)}(y'')\\
			&\notag\hspace*{3em}\times\int_{\R^{2d}} s^{-2d} \exp\left(-c^{-1}\left(\frac{\Vert y'-y''\Vert^2}{4s}+\frac{3\Vert\notag x''-x'-\frac{s(y'+y'')}{2}\Vert^2 }{s^3} \right) \right) \d y'\d y''\d x'\d x''\\
			&\notag\hspace*{3em}+c\Vert f\Vert_\infty^2\lebesgue(\S) \exp\left(-(cs)^{-1}\right)+c\Vert f\Vert_\infty^2 \lebesgue^2(\S)
		\end{align*}
		For bounding the inner integral,} note that
	\begin{align*}
		&\int_{\R^{2d}}  \exp\left(-c^{-1}\left(\frac{\Vert y'-y''\Vert^2}{4s}+\frac{3\Vert\notag x''-x'-\frac{s(y'+y'')}{2}\Vert^2 }{s^3} \right) \right)\d y'\d y''\\
		&\hspace*{3em}=s^{d} \int_{\R^{2d}}  \exp\left(-c^{-1}\left(\Vert w\Vert^2+\Vert\notag (x''-x')s^{-3/2}-\frac{w'}{2}\Vert^2  \right) \right)\d w\d w'\\
		&\hspace*{3em}=s^{d}\int_{\R^{2d}}   \exp\left(-c^{-1}\left(\Vert w\Vert^2+\Vert\notag \frac{w'}{2}\Vert^2  \right) \right)\d w\d w',
	\end{align*}
	where we used the transformations $w=(y'-y'')/\sqrt{s}$, $w'=(y'+y'')/\sqrt{s}$ and the invariance of the Lebesgue measure under translation. 
	Using polar coordinates, it is easy to see that the integral in the last line is bounded by some finite constant independent of $x'$ and $x''$.
	Hence, we obtain that the inner integral is bounded by $cs^d$ for some positive finite constant $c$. 
	\textcolor{dex}{Thus,
		\begin{align}
			\notag	\cC(s)
			&\notag\le cs^{-d}\Vert f\Vert^2_\infty\int_{\R^{2d}}\left(\sum_{i=1}^{n}\1_{B(x_i,\mathsf{s}_1)}(x')\right)\left(\sum_{i=1}^{n}\1_{B(x_i,\mathsf{s}_1)}(x'')\right)\d x'\d x''\\
			&\notag\hspace*{3em}+c\Vert f\Vert_\infty^2\lebesgue(\S) \exp\left(-(cs)^{-1}\right)+c\Vert f\Vert_\infty^2 \lebesgue^2(\S)\\
			&\le cn^2\mathsf{s}_1^{2d}s^{-d}\Vert f\Vert^2_\infty\
			+c\Vert f\Vert_\infty^2\lebesgue(\S) \exp\left(-(cs)^{-1}\right)+c\Vert f\Vert_\infty^2 \lebesgue^2(\S),\label{eq: variance 1}
		\end{align}
	}
	which implies
	\begin{equation}\label{eq: variance 16}
		\int_{\delta_0}^\delta \cC(s)\d s
		\leq c\Vert f\Vert_\infty ^2\Big(\mathsf{s}_1^{2d}(\log(\delta/\delta_0)\mathds{1}_{d=1}+\delta_0^{1-d}\mathds{1}_{d\geq2})+\lebesgue(\S) \exp\left(-(c\delta)^{-1}\right)+ \delta\lebesgue^2(\S) \Big).
	\end{equation}
	Subsequently, for $\delta\leq s\leq D_1<1$, \ref{ass: heat kernel bound} implies as in \eqref{eq: variance 1}
	\begin{align*}
		\cC(s)
		&\le c\int_{\R^{2d}\times \R^{2d}} \vert f(x',y')\vert\vert f(x'';y'')\vert\\
		&\hspace*{3em} \times s^{-2d}\exp\left(-c^{-1}\left(\frac{\Vert y'-y''\Vert^2}{4s}+\frac{3\Vert\notag x''-x'-\frac{s(y'+y'')}{2}\Vert^2 }{s^3} \right) \right) \d x''\d y''\rho(x',y')\d x'\d y'\\
		&\hspace*{3em} +c\Vert f\Vert_\infty^2\lebesgue(\S) \exp\left(-(cs)^{-1}\right)+c\Vert f\Vert_\infty^2 \lebesgue^2(\S)\\
		&\le c\left(s^{-2d}\int_{\R^{4d}} \vert f(x',y')\vert\vert f(x'';y'')\vert\notag  \d x''\d y''\d x'\d y'+\Vert f\Vert_\infty^2\lebesgue(\S) \exp\left(-(cs)^{-1}\right)+\Vert f\Vert_\infty^2 \lebesgue^2(\S)\right)\\
		&\le c\Vert f\Vert^2_\infty \Big(s^{-2d}\lebesgue^2(\S)+ \lebesgue(\S) \exp\left(-(cs)^{-1}\right)+\lebesgue^2(\S)\Big),
	\end{align*}
	and thus
	\begin{align}\nonumber
		\int_\delta^{D_1}\cC(s)\d s
		&\le c\Vert f\Vert^2_\infty \int_\delta^{D_1}s^{-2d}\left(\lebesgue^2(\S)+ \lebesgue(\S) \exp\left(-(cs)^{-1}\right)+\lebesgue^2(\S)\right)\d s\\\nonumber
		&\le c\Vert f\Vert^2_\infty \left(\delta^{1-2d}\lebesgue^2(\S)+\lebesgue(\S) \exp\left(-(cD_1)^{-1}\right)+D_1\lebesgue^2(\S)\right)\\ \label{eq: variance 11}
		&\le c\Vert f\Vert^2_\infty \left(\delta^{1-2d}\lebesgue^2(\S)+\lebesgue(\S) \exp\left(-(cD_1)^{-1}\right)\right),
	\end{align}
	where we used $D_1< 1$. 
	We continue by investigating the integral from $D_1$ to $D_2$. 
	Arguing as in the derivation of \eqref{eq: variance 1}, we get 
	\begin{align*}
		\cC(s)
		&\le c\int_{\R^{2d}\times \R^{2d}} \vert f(x',y')\vert\vert f(x'';y'')\vert p_s(x',y',x'',y'')\d x''\d y''\rho(x',y')\d x'\d y'+c\Vert f\Vert^2_\infty \lebesgue^2(\S)\\
		&=c\int_{\R^{2d}} \vert f(x',y')\vert P_s(\vert f\vert)(x',y')\rho(x',y')\d x'\d y'+c\Vert f\Vert^2_\infty \lebesgue^2(\S).
	\end{align*}
	Thus, Lemma \ref{lemma: semi group bound} implies
	\begin{align}\nonumber
		\int_{D_1}^{D_2}\cC(s)\d s
		&\le c\int_{D_1}^{D_2}\int_{\R^{2d}} \vert f(x',y')\vert P_s(\vert f\vert)(x',y')\rho(x',y')\d x'\d y'\d s+cD_2\Vert f\Vert^2_\infty \lebesgue^2(\S)\\\nonumber
		&\le c\int_{D_1}^{D_2}D_1^{-2d}\Vert f\Vert_\infty \lebesgue(\S)\int_{\R^{2d}} \vert f(x',y')\vert \rho(x',y')\d x'\d y'\d s\\\nonumber
		&\hspace*{3em}+c\int_{D_1}^{D_2}\Vert f\Vert_\infty\exp\left(-\frac{1}{c D_1}\right)\int_{\R^{2d}} \vert f(x',y')\vert \rho(x',y')\d x'\d y'\d s+c\Vert f\Vert^2_\infty \lebesgue^2(\S)D_2\\\label{eq: variance 5}
		&\le c\Vert f\Vert_\infty^2\lebesgue(\S)\left(D_2D_1^{-2d}\lebesgue(\S)+D_2\exp\left(-\frac{1}{c D_1}\right)+D_2\lebesgue(\S)\right),
	\end{align}
	where the value of $c$ only depends on $\S$ and $\rho$. 
	For the remaining part of the integral, we make use of the mixing property \ref{ass: mixing}, which implies that there exists $\kappa>0$ such that
	\begin{equation}\label{eq: variance 6}
		\int_{D_2}^{T}\cC(s)\d s
		\leq c\int_{D_2}^T\Vert f\Vert_\infty^2\e^{-\kappa s}\d s
		\leq c\Vert f\Vert_\infty^2\e^{-\kappa D_2}.
	\end{equation}
	The fact that exponential $\beta$-mixing implies a covariance bound of the above form follows from the proof on page 479 in \cite{viennet1997} and is also described in equation (5) of \cite{Comte2017}. 
	Combining \eqref{eq: variance start}, \eqref{eq: variance 0.5}, \eqref{eq: variance 5}, \eqref{eq: variance 6}, \eqref{eq: variance 11} and \eqref{eq: variance 16} then yields that there exist $c,\kappa>0$ such that, for $0\leq \delta_0\leq\delta\leq D_1<1\leq D_2\leq T$,
	\begin{align*}
		\Var\left(\frac 1 T \int_0^T f(X_s,Y_s)\d s \right)
		&\le c T^{-1}\Vert f\Vert^2_\infty \bigg(\delta_0\lebesgue(\S)+\mathsf{s}_1^{2d}(\log(\delta/\delta_0)\mathds{1}_{d=1}+\delta_0^{1-d}\mathds{1}_{d\geq2}) +\lebesgue(\S) \exp\left(-(c\delta)^{-1}\right)\\
		&\hspace*{7em} +\lebesgue^2(\S)\delta+\delta^{1-2d}\lebesgue^2(\S)+\lebesgue(\S) \exp\left(-(cD_1)^{-1}\right)\\
		&\hspace*{7em}+D_2D_1^{-2d}\lebesgue^2(\S)+D_2\lebesgue(\S)\exp\left(-(c D_1)^{-1}\right)+D_2\lebesgue^2(\S)+\e^{-\kappa D_2}\bigg),
	\end{align*}
	and choosing $D_1=(-c\log(n(\mathsf{s}_1\mathsf{s}_2)^d))^{-1}$, $D_2=-2\kappa^{-1}\log(n(\mathsf{s}_1\mathsf{s}_2)^d)$ we get 
	\begin{equation}\begin{split}\label{eq: variance 17}
			\Var\left(\frac 1 T \int_0^T f(X_s,Y_s)\d s \right)
			&\le cT^{-1}\Vert f\Vert^2_\infty \Big(\delta_0n(\mathsf{s}_1\mathsf{s}_2)^d+\mathsf{s}_1^{2d}(\log(\delta/\delta_0)\mathds{1}_{d=1}+\delta_0^{1-d}\mathds{1}_{d\geq2})\\
			&\hspace*{3em}+n(\mathsf{s}_1\mathsf{s}_2)^d \exp\left(-(c\delta)^{-1}\right) +\delta^{1-2d}n^2(\mathsf{s}_1\mathsf{s}_2)^{2d}
			\\&\hspace*{3em}+\log(n^{-1}(\mathsf{s}_1\mathsf{s}_2)^{-d})^{2d+1}n^2(\mathsf{s}_1\mathsf{s}_2)^{2d}\Big),
	\end{split}\end{equation}	
	where we used that $\delta<1$ and $\mathsf{s}_1,\mathsf{s}_2\in\mathcal{H}$. Choosing $\delta_0=\mathsf{s}_1,\delta=\mathsf{s}_2$ if $\mathsf{s}_1<\mathsf{s}_2$ now entails for large enough $T$ in the case $d=1$ 
	\begin{equation}\begin{split}\notag
			\Var\left(\frac 1 T \int_0^T f(X_s,Y_s)\d s \right)
			&\le cT^{-1}\Vert f\Vert^2_\infty\mathsf{s}_1^2\log T.
	\end{split}\end{equation}	
	\textcolor{dex}{
		We will explain at the end of the proof why the assumption $\mathsf{s}_1<\mathsf{s}_2$ is indeed without loss of generality.
	}
	For $d\geq2$, the choice $\delta_0=\mathsf{s}_1\mathsf{s}_2^{-1}$, $\delta=D_1=(-c\log(n(\mathsf{s}_1\mathsf{s}_2)^d))^{-1}$ gives
	\begin{equation}\begin{split}\notag
			\Var\left(\frac 1 T \int_0^T f(X_s,Y_s)\d s \right)
			&\le cT^{-1}\Vert f\Vert^2_\infty\mathsf{s}_1^{d+1} \mathsf{s}_2^{d-1} ,
	\end{split}\end{equation}	
	where we used $\mathsf{s}_1,\mathsf{s}_2\in\mathcal H$.
	\paragraph{Proof of \eqref{eq: var th 1}}
	Again we split up the covariance integral from \eqref{eq: variance start} into five parts.
	The only new bound concerns the integral from $\delta_0$ to $\delta$. 
	Arguing as in \eqref{eq: variance 1}, we obtain for $0<s<\delta<1$
	\begin{align*}
		\cC(s)&\leq c\int_{\R^{4d}} \vert f(x',y')\vert\vert f(x'';y'')\vert p^1_s(x',y',x'',y') \d x''\d y''\rho(x',y')\d x'\d y'\\
		&\hspace*{3em}+c\Vert f\Vert_\infty^2\lebesgue(\S) \exp\left(-(cs)^{-1}\right)+c\Vert f\Vert_\infty^2 \lebesgue^2(\S),
	\end{align*}
	where
	\[p^1_s(x',y',x'',y')\coloneqq s^{-2d}\exp\left(-c^{-1}\left(\frac{\Vert y'-y''\Vert^2}{4s}+\frac{3\Vert\notag x''-x'-\frac{s(y'+y'')}{2}\Vert^2 }{s^3} \right) \right). \]
	We proceed by finding a bound for $p^1_s(x',y',x'',y'')=s^{-d/2}q_s(x''\vert x',y',y''),$ where 
	\[
	q_s(x''\vert x',y',y'')\coloneqq s^{-(3/2)d}\exp\left(-c^{-1}\left(\frac{3\Vert\notag x''-x'-\frac{s(y'+y'')}{2}\Vert^2 }{s^3} \right) \right).
	\] 
	Since $q_s$ resembles the density of a multidimensional normal distribution, we get
	\begin{align*}
		&\sup_{s\in(0,1)}\sup_{x',y',y''\in\R^d}\int q_s(x''\vert x',y',y')\d x''\\
		&\hspace*{3em}\le\sup_{s\in(0,1)}\sup_{x',y',y''\in\R^d}s^{-(3/2)d}\int \exp\left(-c^{-1}\left(\frac{\Vert\notag x''-x'-\frac{s(y'+y'')}{2}\Vert^2 }{s^3} \right) \right)\d x''
		\leq c.
	\end{align*}
	Hence, we can infer
	\begin{align*}
		\cC(s)
		&\le cs^{-d/2}\int_{\R^{4d}} \vert f(x'';y'')\vert q_s(x''\vert x',y',y'') \d x''\d y''\vert f(x',y')\vert\rho(x',y')\d x'\d y'\\
		&\hspace*{3em}+c\Vert f\Vert_\infty^2\lebesgue(\S) \exp\left(-(cs)^{-1}\right)+c\Vert f\Vert_\infty^2\lebesgue^2(\S)\\
		&\le c\left(\Vert f\Vert_\infty \mathsf{s}_2^ds^{-d/2}\int_{ \R^{2d}} \vert f(x',y')\vert\rho(x',y')\d x'\d y'+\Vert f\Vert_\infty^2\lebesgue(\S) \exp\left(-(cs)^{-1}\right)+\Vert f\Vert_\infty^2\lebesgue^2(\S)\right)\\
		&\le c\Vert f\Vert^2_\infty(\mathsf{s}_2^ds^{-d/2}\lebesgue(\S) +\lebesgue(\S) \exp\left(-(cs)^{-1}\right)+\lebesgue^2(\S)).
	\end{align*}
	Consequently, for $\delta_0\leq \delta$,
	\begin{align}\nonumber
		\int_{\delta_0}^\delta \vert \cC(s)\vert \d s
		&\le c\Vert f\Vert^2_\infty\int_{\delta_0}^\delta\left(\mathsf{s}_2^ds^{-d/2}\lebesgue(\S) +\lebesgue(\S) \exp\left(-(cs)^{-1}\right)+\lebesgue^2(\S)\right)\d s\\\nonumber
		&\le c\Vert f\Vert^2_\infty\int_{\delta_0}^\delta\left(\mathsf{s}_2^ds^{-d/2}\lebesgue(\S) +\lebesgue(\S)s^{-2} \exp\left(-(cs)^{-1}\right)+\lebesgue^2(\S)\right)\d s\\\label{eq: variance 8}
		&\le c\Vert f\Vert^2_\infty\left(\mathsf{s}_2^d\lebesgue(\S)(\sqrt{\delta}\1_{d=1}+\log(\delta/\delta_0)\1_{d=2}+\delta_0^{1-d/2}\1_{d\geq3} )+\lebesgue(\S) \exp\left(-(c\delta)^{-1}\right)+\lebesgue^2(\S)\delta\right).
	\end{align}
	Combining \eqref{eq: variance start}, \eqref{eq: variance 0.5}, \eqref{eq: variance 11}, \eqref{eq: variance 5}, \eqref{eq: variance 6} and \eqref{eq: variance 8}, we get by choosing $D_1=(-c\log(n(\mathsf{s}_1\mathsf{s}_2)^d))^{-1}$,\\ $D_2=-2\kappa^{-1}\log(n(\mathsf{s}_1\mathsf{s}_2)^d)$ as in \eqref{eq: variance 17}
	\begin{equation}\begin{split}\notag
			\Var\left(\frac 1 T \int_0^T f(X_s,Y_s)\d s \right)
			&\le cT^{-1}\Vert f\Vert^2_\infty \Big(\delta_0n(\mathsf{s}_1\mathsf{s}_2)^d+n\mathsf{s}_1^d\mathsf{s}_2^{2d}(\sqrt{\delta}\1_{d=1}+\log(\delta/\delta_0)\1_{d=2}+\delta_0^{1-d/2}\1_{d\geq3} )\\
			&\hspace*{1em} +n(\mathsf{s}_1\mathsf{s}_2)^d \exp\left(-(c\delta)^{-1}\right)+\delta^{1-2d}n^2(\mathsf{s}_1\mathsf{s}_2)^{2d}+\log(n^{-1}(\mathsf{s}_1\mathsf{s}_2)^{-d})^{2d+1}n^2(\mathsf{s}_1\mathsf{s}_2)^{2d}\Big).
	\end{split}\end{equation}	
	Choosing $\delta_0=0,\delta=\mathsf{s}_1^{2/3}$ for $d=1$ then yields
	\begin{equation}\begin{split}\notag
			\Var\left(\frac 1 T \int_0^T f(X_s,Y_s)\d s \right)
			&\le cT^{-1}\Vert f\Vert^2_\infty \mathsf{s}_1^{4/3}\mathsf{s}_2^2.
	\end{split}\end{equation}
	For $d=2$, we choose $\delta_0=\mathsf{s}_2^2$, $\delta=\mathsf{s}_1^{2/3}$ if $\mathsf{s}_2\le\mathsf{s}_1^{1/3}$ which entails
	\begin{equation}\begin{split}\notag
			\Var\left(\frac 1 T \int_0^T f(X_s,Y_s)\d s \right)
			&\le cT^{-1}\Vert f\Vert^2_\infty \mathsf{s}_1^{2}\mathsf{s}_2^{4}\log(T),
	\end{split}\end{equation}
	and, for $d\geq3$, we set $\delta_0=\mathsf{s}_2^2,\delta=D_1$, yielding
	\begin{equation}\begin{split}\notag
			\Var\left(\frac 1 T \int_0^T f(X_s,Y_s)\d s \right)
			&\le cT^{-1}\Vert f\Vert^2_\infty \mathsf{s}_1^{d}\mathsf{s}_2^{d+2}.
	\end{split}\end{equation}
	\textcolor{dex}{
		As with the assumption $\mathsf{s}_1<\mathsf{s}_2$ in the verification of \eqref{eq: var th 2}, we will explain at the end of the proof why the assumption $\mathsf s_2\le \mathsf s_1^{1/3}$ is in fact only temporary.
	}
	\paragraph{Proof of \eqref{eq: var th 3}} 
	For proving \eqref{eq: var th 3}, we set $\delta_0=\delta$.
	Then, we only need to find a new bound for the covariance integral from $\delta$ to $D_1$. 
	\textcolor{dex}{Arguing as in the derviation of \eqref{eq: variance 1}, we get for $0<s<1$
		\begin{align*}
			\notag	\cC(s)
			&\notag\le c\Vert f\Vert^2_\infty\int_{\R^{2d}}\1_{\bigcup_{i=1}^nB(x_i,\mathsf{s}_1)}(x')\1_{\bigcup_{i=1}^nB(x_i,\mathsf{s}_1)}(x'')\1_{\bigcup_{i=1}^nB(y_i,\mathsf{s}_2)}(y')\1_{\bigcup_{i=1}^nB(y_i,\mathsf{s}_2)}(y'')\\
			&\notag\hspace*{6em}\times\int_{\R^{2d}} s^{-2d} \exp\left(-c^{-1}\left(\frac{\Vert y'-y''\Vert^2}{4s}+\frac{3\Vert\notag x''-x'-\frac{s(y'+y'')}{2}\Vert^2 }{s^3} \right) \right) \d y'\d y''\d x'\d x''\\
			&\notag\hspace*{3em}+c\Vert f\Vert_\infty^2\lebesgue(\S) \exp\left(-(cs)^{-1}\right)+c\Vert f\Vert_\infty^2 \lebesgue^2(\S).
		\end{align*}
		We continue by bounding the exponent in the inner integral. 
		Under the given assumptions on $f$, the relation $f(x',y')f(x'',y'')\neq 0$ implies
		\[\exists \mathsf{x},\mathsf{y}\in\R^{2d}\colon\quad
		\Vert x'-\mathsf{x}\Vert<\mathsf{s}_1,\quad \Vert x''-\mathsf{x}\Vert<\mathsf{s}_1,\quad \Vert y'-\mathsf{y}\Vert<\mathsf{s}_2,\quad \Vert y''-\mathsf{y}\Vert<\mathsf{s}_2.
		\] 
	}
	The reverse triangle inequality and the well-known inequality $(a+b)^2\leq 2(a^2+b^2)$ additionally yield if $\mathsf{s}_2\leq \Vert \mathsf{y}\Vert/2$, i.e., $T$ is large enough,
	\begin{align}\notag
		\frac{\Vert y'-y''\Vert^2}{4s}+\frac{3\Vert x''-x'-\frac{s(y'+y'')}{2}\Vert^2 }{s^3}&\ge \frac{\Vert y'-y''\Vert^2}{4s}+\frac{3(\Vert x''-x'\Vert-\Vert\frac{s(y'+y'')}{2}\Vert)^2 }{s^3}\\\notag
		&\ge\frac{(\Vert y'-y''\Vert+\Vert y'+y''\Vert)^2}{8s}\\\notag&\hspace*{3em}+\frac{3\Vert x''-x'\Vert}{s^3}(\Vert x''-x'\Vert-s\Vert y'+y''\Vert )\\\notag
		&\ge\frac{\Vert y'\Vert^2}{2s}+\frac{3\Vert x''-x'\Vert}{s^3}(\Vert x''-x'\Vert-s\Vert y'+y''\Vert )\\
		&\ge\frac{\Vert \mathsf{y}\Vert^2}{8s}+\frac{3\Vert x''-x'\Vert}{s^3}(\Vert x''-x'\Vert-s\Vert y'+y''\Vert ).\label{eq: proof variance bounded}
	\end{align}
	Hence, we have for $s\geq 288\mathsf{s}_1/\Vert \mathsf{y}\Vert$
	\begin{align*}
		\frac{\Vert y'-y''\Vert^2}{4s}+\frac{3\Vert x''-x'-\frac{s(y'+y'')}{2}\Vert^2 }{s^3}
		&\ge\frac{\Vert \mathsf{y}\Vert^2}{8s}-\frac{3\Vert x''-x'\Vert\Vert y'+y''\Vert}{s^2}\\
		&\ge\frac{\Vert \mathsf{y}\Vert^2}{8s}-\frac{18\mathsf{s}_1\Vert \mathsf{y}\Vert}{s^2}\\
		&\ge\frac{\Vert \mathsf{y}\Vert^2}{16s}.
	\end{align*}
	Thus, for $288\mathsf{s}_1/\Vert \mathsf{y}\Vert\leq s<1$, it holds
	\textcolor{dex}{
		\begin{align*}
			\cC(s)	&\notag\le c\Vert f\Vert^2_\infty\int_{\R^{4d}}\1_{\bigcup_{i=1}^nB(x_i,\mathsf{s}_1)}(x')\1_{\bigcup_{i=1}^nB(x_i,\mathsf{s}_1)}(x'')\1_{\bigcup_{i=1}^nB(y_i,\mathsf{s}_2)}(y')\1_{\bigcup_{i=1}^nB(y_i,\mathsf{s}_2)}(y'')
			\\&\hspace*{6em}\times s^{-2d} \exp\left(-\frac{1}{cs}\right) \d y'\d y''\d x'\d x''\\
			&\notag\hspace*{3em}+c\Vert f\Vert_\infty^2\lebesgue(\S) \exp\left(-(cs)^{-1}\right)+c\Vert f\Vert_\infty^2 \lebesgue^2(\S)
			\\&\leq c\Vert f\Vert^2_\infty\left(s^{-2d}(\mathsf{s}_1\mathsf{s}_2)^{2d}\exp\left(-(cs)^{-1}\right) +(\mathsf{s}_1\mathsf{s}_2)^{d}\exp\left(-(cs)^{-1}\right)+(\mathsf{s}_1\mathsf{s}_2)^{2d}\right),\notag
		\end{align*}
	}
	which implies for $288\mathsf{s}_1/\Vert \mathsf{y}\Vert\leq \delta<D_1<1$
	\begin{align}\label{eq: variance 4}
		\notag \int_{\delta}^{D_1}\cC(s)\d s&\leq c\Vert f\Vert^2_\infty(\mathsf{s}_1\mathsf{s}_2)^{d} \int_{\delta}^{D_1}\left((\mathsf{s}_1\mathsf{s}_2)^{d}s^{-2d}\exp\left(-(cs)^{-1} \right) + \exp\left(-(cs)^{-1}\right)+(\mathsf{s}_1\mathsf{s}_2)^{d}\right)\d s\\\notag
		&\leq c\Vert f\Vert^2_\infty(\mathsf{s}_1\mathsf{s}_2)^{d}\left( \int_{(cD_1)^{-1}}^{(c\delta)^{-1}}\left((\mathsf{s}_1\mathsf{s}_2)^{d}s^{2(d-1)}\e^{-s}+s^2\e^{-s}\right)\d s+D_1(\mathsf{s}_1\mathsf{s}_2)^{d}\right)\\\notag
		&\leq c\Vert f\Vert^2_\infty(\mathsf{s}_1\mathsf{s}_2)^{d}\left((\mathsf{s}_1\mathsf{s}_2)^{d}\Gamma(2d-1,(cD_1)^{-1})+\Gamma(3,(cD_1)^{-1}) +D_1(\mathsf{s}_1\mathsf{s}_2)^{d}\right)\\\notag
		&\leq c\Vert f\Vert^2_\infty(\mathsf{s}_1\mathsf{s}_2)^{d}\left((\mathsf{s}_1\mathsf{s}_2)^{d}\exp\left(-\frac{1}{cD_1} \right)\sum_{k=0}^{2(d-1)}(cD_1)^{-k}+\exp\left(-\frac{1}{cD_1} \right)\sum_{k=0}^{2}(cD_1)^{-k} +D_1(\mathsf{s}_1\mathsf{s}_2)^{d}\right)\\
		&\leq c\Vert f\Vert^2_\infty(\mathsf{s}_1\mathsf{s}_2)^{d}\left((\mathsf{s}_1\mathsf{s}_2)^{d}\exp\left(-\frac{1}{cD_1} \right)D_1^{-2(d-1)}+\exp\left(-\frac{1}{cD_1} \right)D_1^{-2} +D_1(\mathsf{s}_1\mathsf{s}_2)^{d}\right).
	\end{align}
	Here, $\Gamma(\cdot,\cdot)$ denotes the upper incomplete gamma function, whose explicit values are well-known if the first argument is an integer. 
	Combining now \eqref{eq: variance start}, \eqref{eq: variance 0.5}, \eqref{eq: variance 5}, \eqref{eq: variance 6} and \eqref{eq: variance 4}, we get by choosing\\ $D_1=(-c\log((\mathsf{s}_1\mathsf{s}_2)^d))^{-1}$, $D_2=-2\kappa^{-1}\log((\mathsf{s}_1\mathsf{s}_2)^d)$ as in \eqref{eq: variance 17}
	\begin{align*}
		\Var\left(\frac 1 T \int_0^T f(X_s,Y_s)\d s \right)
		&\le cT^{-1}\Vert f\Vert^2_\infty \Big(\delta (\mathsf{s}_1\mathsf{s}_2)^d+\log((\mathsf{s}_1\mathsf{s}_2)^{-d})^{2d+1}(\mathsf{s}_1\mathsf{s}_2)^{2d}\Big),
	\end{align*}
	and thus choosing $\delta= 288\mathsf{s}_1/\Vert \mathsf{y}\Vert$ entails for large enough $T$
	\begin{align*}
		\Var\left(\frac 1 T \int_0^T f(X_s,Y_s)\d s \right)
		&\le cT^{-1}\Vert f\Vert^2_\infty \mathsf{s}_1 (\mathsf{s}_1\mathsf{s}_2)^d.
	\end{align*}
	
	\paragraph{Proof of \eqref{eq: var th 4}}
	To prove the assertion, it suffices to combine equations \eqref{eq: variance start}, \eqref{eq: variance 0.5}, \eqref{eq: variance 5}, \eqref{eq: variance 6}, \eqref{eq: variance 8} and \eqref{eq: variance 4} with the choices $\delta_0=0,\delta= 288\mathsf{s}_1/\Vert \mathsf{y}\Vert, D_1=(-c\log((\mathsf{s}_1\mathsf{s}_2)^d))^{-1}, D_2=-2\kappa^{-1}\log((\mathsf{s}_1\mathsf{s}_2)^d).$ \\
	{\color{nd} To conclude the proof it only remains to consider why the assumptions $\mathsf{s}_1<\mathsf{s}_2$ in the proof of \eqref{eq: var th 2} and $\mathsf{s}_2\leq \mathsf{s}_1^{1/3}$ for $d=2$ in the proof of \eqref{eq: var th 1} are negligible. For this note that if one of these assumptions fails to hold, the corresponding other assumption is fulfilled and hence the other variance bound holds, which then yields a tighter bound. }
\end{proof}

Define the function class
\[
\mathcal{G}\coloneqq \big\{K((x-\cdot)\slash h_1,(y-\cdot)/h_2): (x,y) \in D \cap \mathbb{Q}^{2d}\big\},\quad h_1,h_2\in(0,1),
\]
where $K(x,y)=K_1(x)K_2(y)$ with $K_1,K_2$ being Lipschitz continuous, bounded functions of compact support with Lipschitz constants $L_1,L_2$ wrt to the $\sup$-norm $\lVert \cdot \rVert_\infty$. 

\begin{lemma} \label{lemma: covering numbers sup}
	Let $D \subset \R^{2d}$ be a bounded set, assume $\mathcal{A}$, and let $f$ be a locally bounded function. 
	Then, for large enough $t$, it holds for any $\varepsilon > 0$
	\begin{align*}
		\mathcal{N}(\varepsilon, \mathcal{G}f, d_\infty) &\leq  \left(\frac{2L_K\sup_{z\in \mathcal K}\vert f(z)\vert (h_1^{-1}+h_2^{-1})\operatorname{diam}(D)}{\ep}\right)^{2d},
		\\
		\cN(\ep, \cG f, d_{L^2(\mu)})
		&\leq \left(\frac{2L_K \operatorname{diam}(D)\sqrt{\Vert \rho\Vert_\infty }\sup_{z\in \mathcal{K}}\vert f(z)\vert(h_1h_2)^{d/2}(h_1^{-1}+h_2^{-1})}{\ep}\right)^{2d},
		\\
		\mathcal{N}(\ep,\cG f,d_{\G,t})
		&\leq \left(\frac{2L_K c_D\sup_{z\in \mathcal{K}}\vert f(z)\vert (h_1h_2)^{d} \psi_d(h_1,h_2,t) (h_1^{-1}+h_2^{-1})\operatorname{diam}(D)}{\ep}\right)^{2d},
	\end{align*}
	
	where $$\mathcal K\coloneqq \bigcup_{(x,y)\in D\cap \mathbb Q^{2d}}\operatorname{supp}(K((x-\cdot)/h_1,(y-\cdot)\slash h_2)),\quad L_K\coloneqq L_1\Vert K_2\Vert_\infty+L_2\Vert K_1\Vert_\infty .$$
\end{lemma}
\begin{proof}
	%	Note that $\supp(\overline{K}((x-\cdot)/h_1,(y-\cdot)\slash h_2))\subset \mathcal K$ for all $(x,y)\in D\cap \mathbb Q^{2d},h_1,h_2\in(0,1),$ which implies $$\Vert b^j\overline{K}((x-\cdot)/h_1,(y-\cdot)\slash h_2)\Vert_\infty \leq 2\sup_{z\in \mathcal K}\vert b^j(z)\vert \Vert K\Vert_\infty^2.$$
	%	Choose $\varepsilon>0$ and let $\cF$ be an $\tfrac{\varepsilon}{\sup_{z\in \mathcal{K}}\vert b^j(z)\vert}$-cover of $\cG$ with respect to $d_\infty$, where we can w.l.o.g. assume that $\sup_{z\in \mathcal K}\vert b^j(z)\vert>0$. Then for any $(x,y)\in D\cap \mathbb Q^{2d}$ exists $f\in \cF$ such that
	%	$$d_\infty(fb^j,\overline{K}((x-\cdot)/h_1,(y-\cdot)\slash h_2)b^j)\leq\sup_{z\in \mathcal K}\vert b^j(z)\vert d_\infty(f,(\overline{K}((x-\cdot)/h_1,(y-\cdot)\slash h_2))\leq \epsilon.$$
	%	Since $K((x-\cdot)/h_1,(y-\cdot)\slash h_2)$ is Lipschitz continuous with Lipschitz constant $L\Vert K\Vert_\infty (h_1^{-1}+h_2^{-1})$ Lemma C.1 in \cite{dexheimer2020mixing} then gives that there exists $\cF$ such that
	%	$$
	%	\vert \cF\vert=\mathcal{N}(\tfrac{\varepsilon}{\sup_{x\in \mathcal{K}}\vert b^j(x)\vert}, \mathcal{G}, \lVert \cdot \rVert_{d_\infty}) \leq \Big(\frac{4\sup_{x\in \mathcal{K}}\vert b^j(x)\vert L\mathrm{diam}(D)}{\varepsilon h}\Big)^{2d}.$$
	Fix $\ep>0$.
	For all $(x,y)\in D\cap\mathbb Q^{2d}$, $K((x-\cdot)/h_1,(y-\cdot)\slash h_2)$ is Lipschitz continuous with Lipschitz constant $ L_K(h_1^{-1}+h_2^{-1})$ wrt to the sup-norm. Hence, for $(x,y)\in \R^{2d}$
	\begin{align}\nonumber
		&B_{d_\infty}(K((x-\cdot)/h_1,(y-\cdot)\slash h_2),\ep)\\\nonumber
		&\hspace*{3em}\supset
		\left\{K((a-\cdot)/h_1,(b-\cdot)\slash h_2): \Vert K((x-\cdot)/h_1,(y-\cdot)\slash h_2)-K((a-\cdot)/h_1,(b-\cdot)\slash h_2)\Vert_\infty<\ep\right\}\\\label{eq: covering number sup norm}
		&\hspace*{3em}\supset\left\{K((a-\cdot)/h_1,(b-\cdot)\slash h_2):   \Vert (x,y)-(a,b)\Vert_\infty<\ep(L_K(h_1^{-1}+h_2^{-1}))^{-1}
		\right\}.
	\end{align}
	Let $Q\supset D$ be a cube of side length $\operatorname{diam}(D)<\infty $. 
	Then, for 
	\[
	\overline{n}\coloneqq \left(\left\lfloor\frac{L_K (h_1^{-1}+h_2^{-1})\operatorname{diam}(D)}{\ep} \right\rfloor\right)^{2d},
	\] 
	there exist points $(x_1,y_1)\ldots,(x_{\overline{n}},y_{\overline{n}})\in Q$ such that $D\subset Q\subset \bigcup_{i=1}^{\overline{n}} B_{d_\infty}((x_i,y_i), \ep(L_K(h_1^{-1}+h_2^{-1}))^{-1})$. 
	It now follows from \eqref{eq: covering number sup norm} that $\{B_{d_\infty}(K((x_i-\cdot)/h_1,(y_i-\cdot)\slash h_2),\ep): i=1,\ldots,\overline{n} \}$ is an external covering of $\cG$.
	Hence, we obtain
	\[
	\cN(\ep,\cG, d_\infty)\leq \cN_{\textrm{ext}}(\ep/2,\cG, d_\infty)\leq \left(\frac{2L_K (h_1^{-1}+h_2^{-1})\operatorname{diam}(D)}{\ep}\right)^{2d}.
	\]
	Now, let $\cF$ be an $\tfrac{\varepsilon}{\sup_{z\in \mathcal{K}}\vert f(z)\vert}$-cover of $\cG$ with respect to $d_\infty$, where we can assume without losing generality that $\sup_{z\in \mathcal K}\vert f(z)\vert>0$. 
	Then, for any $(x,y)\in D\cap \mathbb Q^{2d}$, there exists $g\in \cF$ such that
	$$d_\infty(fg,fK((x-\cdot)/h_1,(y-\cdot)\slash h_2))\leq\sup_{z\in \mathcal K}\vert f(z)\vert d_\infty(g,K((x-\cdot)/h_1,(y-\cdot)\slash h_2))\leq \ep.$$
	Now note that for $(x_1,y_1),(x_2,y_2)\in D\cap \mathbb{Q}^{2d}$
	\begin{align*}
		&d_{L^2(\mu)}(fK((x_1-\cdot)/h_1,(y_1-\cdot)/h_2),fK((x_2-\cdot)/h_1,(y_2-\cdot)/h_2))
		\\
		&\hspace*{3em}\le\sqrt{\Vert \rho\Vert_\infty }\sup_{z\in \mathcal{K}}\vert f(z)\vert(h_1h_2)^{d/2}d_{\infty}(K(x_1/h_1-\cdot,y_1/h_2-\cdot),K(x_2/h_1-\cdot,y_2/h_2-\cdot)),
	\end{align*}
	and hence
	\begin{align*}
		\cN(\ep, \cG f, d_{L^2(\mu)})&\leq \cN\left(\ep\left(\sqrt{\Vert \rho\Vert_\infty }\sup_{z\in \mathcal{K}}\vert f(z)\vert(h_1h_2)^{d/2}\right)^{-1}, \cG, d_{\infty}\right)
		\\
		&\leq \left(\frac{2L_K \operatorname{diam}(D)\sqrt{\Vert \rho\Vert_\infty }\sup_{z\in \mathcal{K}}\vert f(z)\vert(h_1h_2)^{d/2}(h_1^{-1}+h_2^{-1})}{\ep}\right)^{2d}.
	\end{align*}
	Similarly, we obtain from Proposition \ref{prop: variance not bounded} for $(x_1,y_1),(x_2,y_2)\in D\cap \mathbb{Q}^{2d}$ that there exists a constant $c_D>0$, depending on $D$ and $K$, such that for large enough $t$
	\begin{align}
		&d_{\G,t}((fK((x_1-\cdot)/h_1,(y_1-\cdot)/h_2),fK((x_2-\cdot)/h_1,(y_2-\cdot)/h_2))\notag
		\\&\hspace*{3em}\le c_D\sup_{z\in \mathcal{K}}\vert f(z)\vert (h_1h_2)^d \psi_d(h_1,h_2,t) d_{\infty}((K((x_1-\cdot)/h_1,(y_1-\cdot)/h_2)K((x_2-\cdot)/h_1,(y_2-\cdot)/h_2)). \label{eq: variance uniform}
	\end{align}
	Thus, we have for large enough $t$,
	\begin{align*}
		\mathcal{N}(\ep,\cG f,d_{\G,t})&\leq \mathcal{N}\left(\ep\left(c_D\sup_{z\in \mathcal{K}}\vert f(z)\vert (h_1h_2)^d \psi_d(h_1,h_2,t) \right)^{-1},\cG ,d_\infty\right)
		\\
		&\leq \left(\frac{2L_K c_D\sup_{z\in \mathcal{K}}\vert f(z)\vert (h_1h_2)^d \psi_d(h_1,h_2,t) (h_1^{-1}+h_2^{-1})\operatorname{diam}(D)}{\ep}\right)^{2d},
	\end{align*}
	which completes the proof.
\end{proof}
\begin{remark}
	In equation \eqref{eq: variance uniform}, we implicitly used that there exists a \textit{uniform} constant, such that the results of Propositions \ref{prop: variance not bounded} and \ref{prop: variance bounded} hold for all $g\in\GG$. 
	A look at the proof of these assertions shows that this is indeed true, since we can find a bounded set $\tilde D\subset \R^{2d}$ fulfilling
	$$\bigcup_{(x,y)\in D\cap \mathbb Q^{2d}} \supp(K((x-\cdot)/h_1, (y-\cdot)/h_2))\subset \tilde D.$$
	Furthermore, for the case $\inf_{(x,y)\in D} \Vert y\Vert>\ep>0$ and $n=1$, the constants depending on $\Vert \mathsf{y}_1\Vert$ in Proposition \ref{prop: variance bounded} can all be replaced by analogous constants with respect to $\ep$. 
	This observation will also be used in the proof of Proposition \ref{th: inv density estimation bounded}.
\end{remark}

\begin{proof}[Proof of Proposition \ref{th: inv density estimation general}] 
	First note that the decomposition into bias and stochastic error yields
	\[\mathcal{R}^{(p)}_{\infty}\big(\hat{\rho}_{h_1,h_2,T} ,\rho; D\big)
	\leq\E\left[\sup_{z \in D}\vert \hat{\rho}_{h_1,h_2,T}(z)-\mu(\hat\rho_{h_1,h_2,T})\vert^p \right]^{\frac 1 p}
	+\mathcal B_\rho(h_1,h_2).\]
	We continue by bounding the first term. Denseness of $\mathbb{Q}$ in $\R$ gives
	\[\E\left[\sup_{z \in D}\vert \hat{\rho}_{h_1,h_2,T}(z)-\mu(\hat\rho_{h_1,h_2,T})\vert^p \right]^{\frac 1 p}=T^{-1/2}(h_1h_2)^{-d}\E\left[\sup_{g\in\overline\GG}\Vert\G_T(g)\Vert^p \right]^{1/p}, \]
	where $\overline\GG=\{K((x-\cdot)/h_1,(y-\cdot)/h_2)-\mu(K((x-\cdot)/h_1,(y-\cdot)/h_2)) : (x,y)\in D\cap\mathbb{Q}^{2d}\}.$ 
	Now, since $\Z$ is exponentially $\beta$-mixing, Theorem 3.2 in \cite{dexheimer2020mixing} implies that for $m_T\in (0,T/4]$ there exist $\tau \in [m_T,2m_T]$ and a constant $c>0$ such that
	\begin{align}
		\E\left[\sup_{g\in\overline\GG}\Vert\G_T(g)\Vert^p \right]^{1/p}
		&\le c\Bigg(\int_0^\infty \log \mathcal{N}(u, \overline\GG,\tfrac{2m_t}{\sqrt{T}}d_\infty)\d u+\int_0^\infty \sqrt{\log \mathcal{N}(u,\overline{\GG},d_{\G,\tau})}\d u\notag
		\\&\hspace*{3em}+4 \sup_{g\in\overline\GG}\bigg(\tfrac{2m_T}{\sqrt{T}} \Vert g\Vert_\infty p+\Vert g\Vert_{\G,\tau}\sqrt{p}+\tfrac 1 2 \Vert g\Vert_\infty \sqrt{T}\e^{-\frac{\kappa m_T}{p}}\bigg)  \Bigg).\label{eq: inv dens rate proof}
	\end{align}
	Obviously, the results of Lemma \ref{lemma: covering numbers sup} continue to hold with different constants for $\overline{\GG}$.
	Thus, for large enough $T$,
	\begin{align*}
		\int_0^\infty \log \mathcal{N}(u, \overline\GG,\tfrac{2m_T}{\sqrt{T}}d_\infty)\d u&=c\frac{m_T}{\sqrt{T}}\int_0^{c} \log \mathcal{N}(u, \overline\GG,d_\infty)\d u
		\\
		&\le c\frac{m_T}{\sqrt{T}} \int_0^{c} \log\left(\frac{c (h_1^{-1}+h_2^{-1})}{u}\right) \d u \leq c\frac{m_T}{\sqrt{T}}\log T.
	\end{align*}
	From Proposition \ref{prop: variance not bounded} and the inequality
	\begin{equation}\label{eq:upperest}
		\int_0^C\sqrt{\log(M/u)}\d u\leq 4C\sqrt{\log(M/C)} \quad\text{if}\quad \log(M/C)\geq 2
	\end{equation}
	(see, e.g., p.~592 of \cite{Gine2009}), we get for large enough $T$
	\begin{align*}
		\int_0^\infty \sqrt{\log \mathcal{N}(u,\overline{\GG},d_{\G,\tau})}\d u&\leq c \int_0^{c(h_1h_2)^d\psi_d(h_1,h_2,T) }\sqrt{\log \frac{c(h_1h_2)^d\psi_d(h_1,h_2,T)(h_1^{-1}+h_2^{-1})}{u}}\d u
		\\
		&\leq c (h_1h_2)^d\psi_d(h_1,h_2,T)\sqrt{\log T},
	\end{align*}
	where we used \eqref{eq:upperest} and $h_1,h_2\in\mathcal{H}$. 
	Letting $m_T=(p/\kappa)\log T$ yields together with Proposition \ref{prop: variance not bounded} and \eqref{eq: inv dens rate proof} for large enough $T$ \textcolor{dex}{and $1\leq p\leq \gamma \log T$, with $\gamma>0$,}
	\[
	\E\left[\sup_{g\in\overline\GG}\Vert\G_T(g)\Vert^p \right]^{1/p}
	\le c\Bigg(\frac{p(\log T)^2}{\sqrt{T}}+(h_1h_2)^d\psi_d(h_1,h_2,t)\sqrt{\log T}
	+\frac{\log T}{\sqrt{T}}p^2+(h_1h_2)^d\psi_d(h_1,h_2,T)\sqrt{p}\Bigg),
	\]
	which completes the proof.
\end{proof}

\begin{proof}[Proof of Proposition \ref{th: inv density estimation bounded}] 
	Denoting $G_{h_1,h_2,T}(z)\coloneqq \widehat{\rho}_{h_1,h_2,T}(z)-\E[ \widehat{\rho}_{h_1,h_2,T}(z)]$, we obtain 
	\begin{equation}\label{eq: bias variance density}
		\widehat{\rho}_{h_1,h_2,T}(z)-\rho(z)=G_{h_1,h_2,T}(z)+(\rho \ast K_{h_1,h_2}-\rho)(z),\quad \forall z\in D.
	\end{equation}
	% It is well-known that there exists a constant $c>0,$ such that, for $h_1,h_2$ small enough and for all $z\in D$,
	% 	\begin{equation}\label{eq: aniso bias}
		% \vert (\rho \ast K_{h_1,h_2}-\rho)(z)\vert \leq c(h_1^{\beta_1}+h_2^{\beta_2})
		% 	\end{equation}
	% (see e.g.\ Proposition 1 in \cite{Comte2013}). 
	% We proceed by.
	For bounding $\E[\Vert G_{h_1,h_2,T}(z)\Vert_{L^\infty(D)}]$, we discretize $D$ by means of a finite set $D_T\subset D$ such that any point $z\in D$ fulfills $\inf_{\tilde z\in D_T}\vert z-\tilde{z}\vert\leq \delta_T,$ which can be done with $\operatorname{card}(D_T)\leq c\delta_T^{-2d}$. 
	Exploiting Lipschitz continuity of $K_1,K_2$ yields 
	\[
	\sup_{z\in D}\vert G_{h_1,h_2,T}(z)\vert - \sup_{z\in D_T}\vert G_{h_1,h_2,T}(z)\vert\leq c (h_1^{-1}+h_2^{-1})(h_1h_2)^{-d}\delta_T.
	\]
	Now, Proposition \ref{prop: variance bounded} and Lemma \ref{Bernstein1} imply that, for $h_1,h_2$ small enough and $m_t\in(0,\tfrac{t}{4}]$, there exists $\tau\in[m_t,2m_t]$ such that
	\begin{align*}
		&\Pro\left( \sup_{z\in D_T}\vert G_{h_1,h_2,T}(z)\vert>  \left(r \sqrt{\frac{\log T}{T}}\psi^\circ_d(h_1,h_2)\right) \right)\\
		&\hspace*{3em}\le 2\sum_{z\in D_T}\Bigg(2\exp\left(-\frac{ r^2 \log T\psi_d^\circ(h_1,h_2)^2}{32(\tau\Var(\hat{\rho}_{h_1,h_2,\tau}(z))+4 \Vert K_h\Vert_\infty r \sqrt{\frac{\log T}{T}}\psi_d^\circ(h_1,h_2) m_T)}\right)\\
		&\hspace*{7em}+\frac{T}{m_T}c_\kappa\e^{-\kappa m_T}\mathds{1}_{(0,8\Vert K_{h_1,h_2}\Vert_\infty)}\left( r \sqrt{\frac{\log T}{T}}\psi^\circ_d(h_1,h_2) \right) \Bigg)\\
		&\hspace*{3em}\le c\delta_T^{-2d}\Bigg(\exp\left(-\frac{ r^2 \log T\psi^\circ_d(h_1,h_2)^2}{c( \psi^\circ_d(h_1,h_2)^2+ (h_1h_2)^{-d}r \sqrt{\frac{\log T}{T}}\psi^\circ_d(h_1,h_2) m_T)}\right)\\
		&\hspace*{7em}+\frac{T}{m_T}\e^{-\kappa m_T}\mathds{1}_{(0,c(h_1h_2)^{-d})}\left( r \sqrt{\frac{\log T}{T}}\psi^\circ_d(h_1,h_2) \right) \Bigg)\\
		&\hspace*{3em}=c\delta_T^{-2d}\Bigg(\exp\left(-\frac{ r^2 \log T}{c(1+ (h_1h_2)^{-d}r \sqrt{\frac{\log T}{T}}\psi^\circ_d(h_1,h_2)^{-1} m_T)}\right)\\
		&\hspace*{7em}+\frac{T}{m_T}\e^{-\kappa m_T}\mathds{1}_{(0,c(h_1h_2)^{-d}T^{1/2}\log T^{-1/2}\psi^\circ_d(h_1,h_2)^{-1})}\left( r \right) \Bigg),
	\end{align*}
	where we assume that $m_T=c_m\log T$ for some $c_m>0$. 
	Then, the well-known inequality $$\E[Y]\leq a+\int_a^\infty \Pro(Y> r)\d r,\quad a\geq0,$$ implies for $Y\coloneqq \sup_{z\in D_T}\vert G_{h_1,h_2,T}(z)\vert\sqrt{\frac{T}{\log T}}\psi^\circ_d(h_1,h_2)^{-1}$, with a suitable constant $c_a>0$ depending on $a$,
	\begin{align*}
		&\E\left[\sup_{z\in D_T}\vert G_{h_1,h_2,T}(z)\vert\right]\\
		&\hspace*{3em}\le\left(\sqrt{\frac{\log T}{T}}\psi^\circ_d(h_1,h_2)\right)\Bigg(a+c\delta_T^{-2d}\int_a^\infty\Bigg(\exp\left(-\frac{ r^2 \log T}{c(1+ (h_1h_2)^{-d}r \sqrt{\frac{\log T}{T}}\psi^\circ_d(h_1,h_2)^{-1} m_T)}\right)\\
		&\hspace*{7em}+\frac{T}{m_T}\e^{-\kappa m_T}\mathds{1}_{(0,c(h_1h_2)^{-d}T^{1/2}\log T^{-1/2}\psi^\circ_d(h_1,h_2)^{-1})}\left( r \right) \Bigg)\d  r \Bigg)\\
		&\hspace*{3em}\le\left(\sqrt{\frac{\log T}{T}}\psi^\circ_d(h_1,h_2)\right)\Bigg(a+c\delta_T^{-2d}\Bigg(c_aT^{-\frac{a^2 }{c}}+(h_1h_2)^{-d}T^{3/2-\kappa c_m}\log T^{-3/2}\psi^\circ_d(h_1,h_2)^{-1} \Bigg) \Bigg),
	\end{align*}
	where we used Assumption \eqref{ass: bandwidth speed}.
	% $(h_1h_2)^{-d}\sqrt{\frac{\log T^{3}}{T}}\psi^\circ_d(h_1,h_2)^{-1} \to0$.
	Lipschitz continuity of $K_{h_1,h_2}$ then yields
	\begin{align*}
		\E\left[\Vert G_{h_1,h_2,T}(z)\Vert_{L^\infty(D)}\right]
		&\le\E\left[\vert \sup_{z\in D}\vert G_{h_1,h_2,T}(z)\vert-\sup_{z \in D_T}\vert G_{h_1,h_2,T}(z)\vert\vert\right]+\E\left[\sup_{z \in D_T}\vert G_{h_1,h_2,T}(z)\vert\right]\\
		&\le c(h_1h_2)^{-d}(h_1^{-1}+h_2^{-1})\delta_T\\
		&\hspace*{3em}+ \sqrt{\frac{\log T}{T}}\psi_d(h_1,h_2)\\&\hspace*{3em}\times\Bigg(a+c\delta_T^{-2d}\Bigg(c_aT^{-\frac{a^2 }{c}}+(h_1h_2)^{-d}T^{3/2-\kappa c_m}\log T^{-3/2}\psi^\circ_d(h_1,h_2)^{-1} \Bigg)\Bigg).
	\end{align*}
	Then, choosing $\delta_T=(h_1h_2)^{d+1}\sqrt{T^{-1}\log T}\psi^\circ_d(h_1,h_2)$ immediately yields 
	\[
	(h_1h_2)^{-d}(h_1^{-1}+h_2^{-1})\delta_T \in\cO\left(\sqrt{\frac{\log T}{T}}\psi^\circ_d(h_1,h_2) \right).
	\] 
	Now note that $h_1,h_2\in\mathcal{H}$ implies the existence of $c,Q>0$ such that $\delta_T^{-2d}\leq cT^{Q},$ for large enough $T$.
	Hence, choosing $a^2= cQ,c_m= \kappa^{-1}(\tfrac 3 2+ Q)$ yields $$\E[\Vert G_{h_1,h_2,T}(z)\Vert_{L^\infty(D)} ]\in\cO\left(\sqrt{\frac{\log T}{T}}\psi^\circ_d(h_1,h_2) \right).$$ 
	The assertion now follows by combining this with decomposition \eqref{eq: bias variance density}.
\end{proof}

\begin{proof}[Proof of Theorem \ref{cor: inv dens 1}]
	It is well-known that there exists a constant $c>0,$ such that, for $h_1,h_2$ small enough and for all $z\in D$,
	\begin{equation}\label{eq: aniso bias}
		\mathcal B_\rho(h_1,h_2)=\vert (\rho \ast K_{h_1,h_2}-\rho)(z)\vert \leq c(h_1^{\beta_1}+h_2^{\beta_2})
	\end{equation}
	(see, e.g., Proposition 1 in \cite{Comte2013}). 
	Plugging this bound and $h_1,h_2$ as specified in \eqref{def:Psi} and \eqref{def:h Psi} into \eqref{eq: inv density estimation general} and \eqref{eq: inv density estimation bounded}, the assertion follows, since $\beta_1>1,\beta_2>2$ implies that \eqref{ass: bandwidth speed} is satisfied.
\end{proof}
\section{Proofs for Section \ref{sec: drift}}
The proof of Theorem \ref{th: Bernstein drift} will require the following Lemma.
\begin{lemma}\label{lemma: moment bound drift}
	Suppose that $\Z$ is exponentially $\beta$-mixing, and let $\mathcal G$ be a countable class of bounded real-valued functions. 
	Then, for $m_t\in(0,t/4)$, there exists $\tau\in[m_t,2m_t]$ such that, for any $p\geq1$,
	\begin{align*}
		&\sup_{g\in\GG}\left(\E\bigg[\Big|\int_0^tg(Z_s)\d s\Big|^{ p}\bigg]\right)^{1/p}\leq 
		\sup_{g\in\GG}\left(c_1m_t\Vert g\Vert_\infty p+c_2\sqrt{tp}\Vert g\Vert_{\G,\tau}+2c_\kappa t\Vert g\Vert_\infty \e^{-\tfrac{\kappa m_t}{p}}+t\vert\mu(g)\vert\right),
	\end{align*}
	where $c_1=\frac{8}{3}\e^{1/2\e}\sqrt{2}\e^{1/(12)-1}, c_2=2 (2\e)^{-1/2}\e^{1/(2\e)}\sqrt{\pi}\e^{1/6}.$
\end{lemma}
\begin{proof}
	We start by splitting the process $(Z_s)_{0\leq s\leq t}$ into $2n_t$ parts of length $m_t$, where $t=2n_tm_t,n_t\in\N$, i.e., for $j\in\{1,\ldots,n_t\}$, we define the processes 
	\[
	Z^{j,1}\coloneqq (Z_s)_{s\in [2(j-1)m_t,(2j-1)m_t]},\quad Z^{j,2}\coloneqq (Z_s)_{s\in [(2j-1)m_t,2jm_t]}.
	\]
	Analogously to the proof of Lemma 3.1 and Theorem 3.2 of \cite{dexheimer2020mixing}, we use arguments of the proof of Proposition 5.2 of \cite{viennet1997}, yielding the existence of a process $(\hat{Z}_s)_{0\leq s\leq t}$ such that, for $k=1,2$,
	\begin{enumerate}
		\item[$(1)$] $Z^{j,k}\overset{(\mathrm{d})}{=}\widehat{Z}^{j,k}$ for all $j\in\{1,\ldots,n_t\}$,
		\item[$(2)$] $\exists c_\kappa,\kappa>0:\Pro(Z^{j,k}\neq\widehat{Z}^{j,k})\leq c_\kappa \e^{-\kappa m_t}$ for all $j\in\{1,\ldots,n_t\},$
		\item[$(3)$] $\widehat{Z}^{1,k},\ldots,\widehat{Z}^{n_t,k}$ are independent,
	\end{enumerate}
	where $\widehat{Z}^{j,k}$ is defined analogously to $Z^{j,k}$ for $j\in\{1,\ldots,n_t\}$, $k=1,2$. 
	Furthermore, define 
	\[
	I_g(Z^{j,1})\coloneqq \int_{2(j-1)m_t}^{(2j-1)m_t}g(Z_s)\d s,\quad I_g(Z^{j,2})\coloneqq \int_{(2j-1)m_t}^{2jm_t}g(Z_s)\d s,\quad j=1,\ldots,n_t,
	\] 
	and, analogously, define $I_g(\widehat{Z}^{j,k})$ for $k=1,2$, $j\in\{1,\ldots,n_t\}$. 
	Then, for fixed $p\geq 1$, $g\in\cG$, it holds
	\begin{align*}
		\left(\E\bigg[\Big|\int_0^tg(Z_s)\d s\Big|^{ p}\bigg]\right)^{1/p}
		&\le\left(\E\bigg[\Big|\sum_{k=1}^2\sum_{j=1}^{n_t} (I_g(Z^{j,k})-I_g(\widehat{Z}^{j,k}))\Big|^{p}\bigg]\right)^{1/p}+
		\left(\E\bigg[\Big|\sum_{k=1}^2\sum_{j=1}^{n_t}I_g(\widehat{Z}^{j,k})\Big|^{p}\bigg]\right)^{1/p}\\
		&\le 2m_t\Vert g\Vert_\infty\sum_{k=1}^2\sum_{j=1}^{n_t} \Pro(Z^{j,k}\neq \hat{Z}^{j,k})^{1/p}+\left(\E\bigg[\Big|\sum_{k=1}^2\sum_{j=1}^{n_t}I_g(\widehat{Z}^{j,k})\Big|^p\bigg]\right)^{1/p}\\
		&\le 2c_\kappa t\Vert g\Vert_\infty \e^{-\tfrac{\kappa m_t}{p}}+\sum_{k=1}^2\left(\E\bigg[\Big|\sum_{j=1}^{n_t}(I_g(\widehat{Z}^{j,k})-m_t\mu(g))\Big|^p\bigg]\right)^{1/p}+t\vert\mu(g)\vert.
	\end{align*}
	Since $\widehat{Z}^{1,k},\ldots,\widehat{Z}^{n_t,k}$ are independent, the classical Bernstein inequality gives for $u>0$
	\[
	\Pro\left(\Big|\sum_{j=1}^{n_t}(I_g(\widehat{Z}^{j,k})-m_t\mu(g))\Big|>\sqrt{2n_t\Var\left(\int_0^{m_t}g(Z_s)\d s\right)u} +\frac{4}{3}m_t\Vert g\Vert_\infty u\right)\le\e^{-u},
	\]
	and thus Lemma A.2 in \cite{Dirksen2015} implies 
	\[
	\sum_{k=1}^2\left(\E\bigg[\Big|\sum_{j=1}^{n_t}(I_g(\widehat{Z}^{j,k})-m_t\mu(g))\Big|^p\bigg]\right)^{1/p}
	\leq c_1'm_t\Vert g\Vert_\infty p+c_2'\sqrt{t\Var\left(\frac{1}{\sqrt{m_t}}\int_0^{m_t}g(Z_s)\d s\right)}\sqrt{p},
	\]
	where
	$c_1'= \frac{16}{3}\e^{1/2\e}(\sqrt{2}\e^{1/(12p)})^{1/p}\e^{-1}$, $c_2'=2 (2\e)^{-1/2}\e^{1/(2\e)}(\sqrt{\pi}\e^{1/(6p)})^{1/p}$. 
	The generalization to $m_t\in(0,t/4)$ is now analogous to the proof of Lemma 3.1 and Theorem 3.2 of \cite{dexheimer2020mixing}.
\end{proof}
\begin{proof}[Proof of Theorem \ref{th: Bernstein drift}]
	We start by noting that, letting 
	\[
	\I^{j}_{t,\sigma}(g)\coloneqq \frac{1}{\sqrt{t}}\int_0^t g(Z_s)\sum_{k=1}^d\sigma_{jk}(Z_s)\d W^k_s,\quad g\in\cG, j\in\{1,\ldots,d\},
	\]
	we obtain for any $p\ge1$
	\begin{equation}\label{eq: moment bound 0}
		\left(\E\left[\sup_{g\in\mathcal G}|\mathbb{H}^j_t(g)-\sqrt t\mu(gb^j)|^p\right]\right)^{1/p}
		\le\left(\E\left[\sup_{g\in\mathcal G}|\G_t(gb^j-\mu(gb^j))|^p\right]\right)^{1/p}+\left(\E\left[\sup_{g\in\mathcal G}|\I^{j}_{t,\sigma}(g)|^p\right]\right)^{1/p}.
	\end{equation}
	Theorem 3.2 of \cite{dexheimer2020mixing} then implies that there is a constant $c>0$ such that, for any $m_t\in (0,t/4]$, there exists $\tau\in[m_t,2m_t]$ such that, for any $p\geq 1$,
	\begin{equation}
		\begin{split}\label{eq: moment bound 1}
			\left(\E\left[\sup_{g\in\mathcal G}|\G_t(gb^j-\mu(gb^j))|^p\right]\right)^{1/p}
			&\le c \Bigg(\int_0^\infty\log\mathcal N\big(u,\mathcal Gb^j,\tfrac{m_t}{\sqrt{t}}d_\infty\big)\d u+\int_0^\infty\sqrt{\log \mathcal N(u,\mathcal Gb^j,d_{\G,\tau})}\diff u\\
			&\hspace*{3em}+\sup_{g\in\mathcal G}\Big(\frac{m_t}{\sqrt{t}}\|gb^j\|_\infty p+ \lVert gb^j\rVert_{\mathbb{G},\tau}\sqrt p +\lVert gb^j \rVert_{\infty} c_{\kappa} \sqrt{t} \mathrm{e}^{-\frac{\kappa m_t}{p}}\Big)\Bigg).
		\end{split}
	\end{equation}
	It thus remains to bound $\left(\E\left[\sup_{g\in\mathcal G}|\I^{j}_{t,\sigma}(g)|^p\right]\right)^{1/p}$. 
	Since, for any $g\in\cG$, $\int_0^t g(Z_s)\sum_{k=1}^d\sigma_{jk}(Z_s)\d W^k_s$ is a continuous martingale, \eqref{inequ: Bernstein martingal} yields
	\[
	\Pro\left(\vert \I^{j}_{t,\sigma}(g) \vert>u\right)\le2\e^{-\frac{tu^2}{2y}}+\Pro\left(\int_0^tg^2(Z_s)a_{jj}(Z_s)\d s>y\right),\quad u,y>0,
	\]
	where $a=\sigma\sigma^\top$.
	\textcolor{dex}{
		Additionally, Lemma \ref{Bernstein1} yields for $y>0$ that, for any $m_{t,2}\in(0,t/4)$, there exists $\tau_2\in [m_{t,2},2m_{t,2}]$ such that
		\begin{align*}
			&\Pro\left(\int_0^t\left(g^2a_{jj}(Z_s)\right)\d s>y+t\mu(g^2a_{jj})\right)\\
			&\hspace*{3em}\le2\exp\left(-\frac{y^2}{32t\big(\Var\big(\tfrac{1}{\sqrt{\tau_2}}\int_0^{\tau_2}\big(g^2a_{jj}\big)(Z_s)\d s\big)+2y\Vert g^2a_{jj}\Vert_\infty \tfrac{m_{t,2}}{t}\big)}\right)\\
			&\hspace*{20em}+\frac{t}{m_{t,2}}c_\kappa\e^{-\kappa m_{t,2}}\mathds{1}_{(0,4t\Vert g^2a_{jj}\Vert_\infty)}(y),
	\end{align*}}
	%	and thus we have for $y>0$
	%	\begin{align*}
		%		&\Pro\left(\int_0^tg^2(Z_s)a_{jj}(Z_s)\d s>\sqrt{32ut}y+t\mu(g^2a_{jj})\right)
		%		\\
		%		\leq &2\exp\left(-u\frac{y^2}{\Var\big(\tfrac{1}{\sqrt{\tau_2}}\int_0^{\tau_2}g^2a_{jj}(Z_s)\d s\big)+\sqrt{128u}y\Vert g^2a_{jj}\Vert_\infty \tfrac{m_{t,2}}{\sqrt{t}}}\right)
		%		\\&\hspace*{3em}+\frac{t}{m_{t,2}}c_\kappa\e^{-\kappa m_{t,2}}\mathds{1}_{(0,4t\Vert g^2a_{jj}\Vert_\infty)}(\sqrt{32ut}y),
		%	\end{align*}
	and, letting 
	\[
	y_{u,t}\coloneqq \sqrt{2\Var\left(\tfrac{1}{\sqrt{\tau_2}}\int_0^{\tau_2}\big(g^2a_{jj}\big)(Z_s)\d s\right)}+\sqrt{256u}	\Vert g^2a_{jj}\Vert_\infty \tfrac{m_{t,2}}{\sqrt{t}},
	\] 
	we get
	\[
	\Pro\left(\int_0^tg^2(Z_s)a_{jj}(Z_s)\d s>\sqrt{32ut}y_{u,t}+t\mu(g^2a_{jj})\right)\le2\e^{-u}+\frac{t}{m_{t,2}}c_\kappa\e^{-\kappa m_{t,2}}\mathds{1}_{(u,\infty)}(\tfrac{t}{16m_{t,2}}).
	\]
	The choice $m_{t,2}=\frac{\sqrt{t}}{2\sqrt{\kappa}}$ then yields, for large enough $t$,
	\begin{align*}
		\Pro\left(\int_0^tg^2(Z_s)a_{jj}(Z_s)\d s>\sqrt{32ut}y_{u,t}+t\mu(g^2a_{jj})\right)
		&\le2\e^{-u}+2c_\kappa\sqrt{\kappa t}\e^{-\frac{\sqrt{\kappa t}}{2}}\mathds{1}_{(u,\infty)}(\tfrac{\sqrt{\kappa t}}{8})\\
		&\le2\e^{-u}+2\e^{-\frac{\sqrt{\kappa t}}{8}}\mathds{1}_{(u,\infty)}(\tfrac{\sqrt{\kappa t}}{8}) \ \le\ 4\e^{-u}.	
	\end{align*}
	Hence, we have for $r>0$ 
	\[
	\Pro(\vert \I^{j}_{t,\sigma}(g) \vert>r)\le2\exp\left(-\frac{r^2}{\frac{16\sqrt{u}}{\sqrt{t}}\Vert g^2a_{jj}\Vert_{\G,\tau_2}+\sqrt{8192\kappa^{-1}}u	\Vert g^2a_{jj}\Vert_\infty \tfrac{1} {\sqrt{t}}+2\mu(g^2a_{jj})}\right)+4\e^{-u},
	\]
	and, thus, it holds for large enough $t$
	\begin{align}\nonumber
		6\e^{-u}&\ge 
		\Pro\left(\vert \I^{j}_{t,\sigma}(g) \vert>\frac{4(\sqrt{u}+u)}{t^{1/4}}\sqrt{\Vert g^2a_{jj}\Vert_{\G,\tau_2} }+u\left(\frac{8192\Vert a_{jj}\Vert_\infty^2}{\kappa t}\right)^{1/4} \Vert g\Vert_\infty+ \sqrt{u2\Vert a_{jj}\Vert_\infty \mu(g^2)} \right)\\\nonumber
		&\ge\Pro\left(\vert \I^{j}_{t,\sigma}(g) \vert>\frac{4\sqrt{\Vert a_{jj}\Vert_\infty}(\sqrt{u}+u)}{(\kappa t)^{1/8}}\mu(g^4)^{1/4} +u\left(\frac{8192\Vert a_{jj}\Vert_\infty^2}{\kappa t}\right)^{1/4} \Vert g\Vert_\infty+ \sqrt{u2\Vert a_{jj}\Vert_\infty \mu(g^2)} \right)\\\label{eq: bernstein stoch int}
		&\textcolor{dex}{\ge\Pro\Bigg(\vert \I^{j}_{t,\sigma}(g) \vert>u\left(\frac{256\Vert a_{jj}\Vert_\infty^2}{\sqrt{\kappa t}}\right)^{1/4}\left(\mu(g^4)^{1/4}+\left(\frac{32}{\sqrt{\kappa t}} \right)^{1/4}\Vert g\Vert_\infty \right)}\nonumber\\
		&\hspace{10em}\textcolor{dex}{+ \sqrt{u\Vert a_{jj}\Vert_\infty}\left(\sqrt{2\mu(g^2)}+4\left(\frac{\mu(g^4)}{\sqrt{\kappa t}} \right)^{1/4} \right) \Bigg),}
	\end{align}
	\textcolor{dex}{where we used Jensen's inequality and Fubini's theorem for showing
		\begin{align*}
			\sqrt{\Vert g^2a_{jj}\Vert_{\G,\tau_2}}&\leq\left(\frac{1}{\tau_2}\E\left[ \left(\int_0^{\tau_2} \left(g^2a_{jj}\right)(Z_s)\d s \right)^2 \right]\right)^{1/4}
			=\left(\tau_2\E\left[ \left(\frac{1}{\tau_2}\int_0^{\tau_2}\left(g^2a_{jj}\right)(Z_s)\d s \right)^2 \right]\right)^{ 1/ 4}
			\\
			&\leq \left(\E\left[ \int_0^{\tau_2} \left(g^4a^2_{jj}\right)(Z_s)\d s  \right]\right)^{1 /4}
			\leq \tau_2^{1/4} \Vert a_{jj}\Vert^{1/2}_\infty \mu(g^4)^{1/4}\leq \left(\frac{t}{\kappa} \right)^{1/8} \Vert a_{jj}\Vert^{1/2}_\infty \mu(g^4)^{1/4}.
		\end{align*}
	}
	We now want to use Theorem 3.5 in \cite{Dirksen2015} which requires a bound of the form $2\exp(-u)$. 
	However, inspection of the proof of this theorem and, in particular, the proof of Lemma A.4 of \cite{Dirksen2015} used therein shows that the bound in \eqref{eq: bernstein stoch int} suffices. 
	Thus, we have that there exist constants $c_1,c_2>0$ and $t_0>0$ such that, for any $p\geq1$ and $t\geq t_0$,
	\begin{align}\nonumber
		\left(\E\left[\sup_{g\in\mathcal G}|\I^{j}_{t,\sigma}(g)|^p\right]\right)^{1/p}
		&\le \textcolor{dex}{c_1\int_0^\infty\log\mathcal N\big(u,\mathcal G,t^{-1/4}d_\infty+t^{-1/8}d_{L^4(\mu)}\big)\d u} \\&\hspace*{3em}\nonumber\textcolor{dex}{+ c_2\int_0^\infty\sqrt{\log \mathcal N(u,\mathcal G,d_{L^2(\mu)}+t^{-1/8}d_{L^4(\mu)})}\diff u}\\&\hspace*{3em}\notag+2\sup_{g\in\mathcal G}\left(\E\left[|\I^{j}_{t,\sigma}(g)|^p\right]\right)^{1/p}\\\nonumber
		&\le \textcolor{dex}{c_1\int_0^\infty\log\mathcal N\big(u,\mathcal G,t^{-1/4}d_\infty+t^{-1/8}d_{L^4(\mu)}\big)\d u} \\&\hspace*{3em}\nonumber\textcolor{dex}{+ c_2\int_0^\infty\sqrt{\log \mathcal N(u,\mathcal G,d_{L^2(\mu)}+t^{-1/8}d_{L^4(\mu)})}\diff u}\\\nonumber
		&\hspace*{3em}+\frac{2}{\sqrt{t}}C_p\sup_{g\in\mathcal G}\left(\E\left[\left(\int_0^t g^2(Z_s)a_{jj}(Z_s)\d s \right)^{p/2}\right]\right)^{1/p}\\
		\begin{split}\label{eq: moment bound 2}
			&\le \textcolor{dex}{c_1\int_0^\infty\log\mathcal N\big(u,\mathcal G,t^{-1/4}d_\infty+t^{-1/8}d_{L^4(\mu)}\big)\d u}\\ 
			&\hspace*{3em}\textcolor{dex}{+ c_2\int_0^\infty\sqrt{\log \mathcal N(u,\mathcal G,d_{L^2(\mu)}+t^{-1/8}d_{L^4(\mu)})}\diff u}\\
			&\hspace*{3em}+\frac{2}{\sqrt{t}}C_p\sup_{g\in\mathcal G}\left(\E\left[\left(\int_0^t g^2(Z_s)a_{jj}(Z_s)\d s \right)^{p}\right]\right)^{1/(2p)},
		\end{split}
	\end{align}
	where we used the Burkholder--Davis--Gundy inequality (with $C_p>0$ denoting the corresponding constant) and H\"older's inequality. 
	Additionally, we bounded the $\gamma_\alpha$ functionals appearing in Theorem 3.5 of \cite{Dirksen2015} by the corresponding entropy integrals (see Section 1.2 in \cite{talagrand2005}). 
	Thus, we can see that $c_1,c_2$ can be set to
	\[
	\textcolor{dex}{c_1=4\tilde{C}_0\left(\frac{\Vert a_{jj}\Vert_\infty^2}{\sqrt{\kappa}} \right)^{1/4}\left(1+\left(\frac{32}{\sqrt{\kappa}}\right)^{1/4} \right),  \quad c_2= \tilde{C}_1 \sqrt{\Vert a_{jj}\Vert_\infty}\left(\sqrt{2}+4\kappa^{-1/8} \right),} \]
	where $\tilde{C}_0,\tilde{C}_1$ represent the universal constants from Theorem 3.5 in \cite{Dirksen2015}, adjusted to the bound in \eqref{eq: bernstein stoch int} and multiplied by the respective constants involved in bounding the $\gamma_\alpha$ functionals. Furthermore, combining Proposition 4.2 in \cite{barlow1982} with the H\"older inequality shows that there exists a universal constant $\tilde{C}_2>0$ such that $C_p\leq  \tilde{C}_2\sqrt{p}$ and, thus, Lemma \ref{lemma: moment bound drift} implies that for any $\tilde m_t\in (0,t/4)$ there exists $\tilde\tau \in [\tilde m_t,2\tilde m_t]$ such that
	\begin{align}\nonumber
		&\frac{2}{\sqrt{t}}C_p\sup_{g\in\mathcal G}\left(\E\left[\left(\int_0^t g^2(Z_s)a_{jj}(Z_s)\d s \right)^{p}\right]\right)^{1/(2p)}\le\frac{2\tilde{C}_2\sqrt{p}}{\sqrt{t}}\sup_{g\in\mathcal G}\left(\E\left[\left(\int_0^t g^2(Z_s)a_{jj}(Z_s)\d s \right)^{p}\right]\right)^{1/(2p)}\\\nonumber
		&\hspace*{3em}\le\frac{2\tilde{C}_2\sqrt{p}}{\sqrt{t}}\sup_{g\in\mathcal G}\left(c\Big(\tilde m_t\Vert g^2 a_{jj}\Vert_\infty p+\sqrt{tp}\Vert g^2 a_{jj}\Vert_{\G,\tilde\tau}+t\Vert g^2 a_{jj}\Vert_\infty \e^{-\tfrac{\kappa \tilde m_t}{p}}\Big)+t\vert\mu(g^2 a_{jj})\vert\right)^{1/2}\\\label{eq: moment bound 3}
		&\hspace*{3em}\textcolor{dex}{\le \sup_{g\in\mathcal G}\Big(c\Big(p\sqrt{\frac{\tilde m_t\Vert a_{jj}\Vert_\infty}{t}}\Vert g\Vert_\infty +p^{3/4}(\tilde{\tau}/t)^{1/4}\Vert g\Vert_{L^4(\mu)} +\sqrt{p\Vert a_{jj}\Vert_\infty} \Vert g\Vert_\infty \e^{-\tfrac{\kappa \tilde m_t}{2p}}\Big)}\\
		&\hspace{28em}\textcolor{dex}{+2\tilde{C}_2\sqrt{p\Vert a_{jj}\Vert_\infty}\Vert g\Vert_{L^2(\mu)}\Big)}.\notag
	\end{align}
	Combining \eqref{eq: moment bound 0}, \eqref{eq: moment bound 1}, \eqref{eq: moment bound 2} and \eqref{eq: moment bound 3} now yields the required assertion.
\end{proof}

\begin{proof}[Proof of Proposition \ref{theorem: drift}]
	We start with the usual decomposition
	\begin{align*}
		\mathcal{R}^{(p)}_{\infty}\big(\overline{b}_{j,h_3,h_4,T}, b^j\rho ; D\big)&
		%		\\=& \left(\E\left[\left\|\overline b_{j,h_3,h_4,T}-b^j\rho\right\|^p_{L^\infty(D)}\right]\right)^{1\slash p}
		\le\left(\E\left[\left\|\overline b_{j,h_1,h_2,T}-\mu(K_{h_1,h_2}(z-\cdot)b^j)\right\|^p_{L^\infty(D)}\right]\right)^{\frac1p} +\underbrace{\left\|\mu(K_{h_1,h_2}(z-\cdot)b^j)-b^j\rho\right\|_{L^\infty(D)}}_{=\mathcal B_{b^j\rho}(h_1,h_2)}.
	\end{align*}
	%As is well-known (see, e.g., Proposition 1 in \cite{Comte2013}), the latter bias term is bounded by $c(h_3^{\beta_1}+h_4^{\beta_2})$ for some constant $c>0$. 
	% Furthermore, we have
	% \begin{align*}
		% &\left(\E\left[\left\|\overline b_{j,h_3,h_4,T}-\mu(K_{h_3,h_4}(z-\cdot)b^j)\right\|^p_{L^\infty(D)}\right]\right)^{1\slash p}\\
		% &\hspace*{3em}=(h_3h_4)^{-d}T^{-1/2}\left(\E\left[\left\|\H^j_t(K((z-\cdot)/(h_3h_4)))-\sqrt{T}\mu(K((z-\cdot)/(h_3h_4))b^j)\right\|^p_{L^\infty(D)}\right]\right)^{1\slash p}.
		% \end{align*}
	Now denseness of $\mathbb Q$, the dominated convergence theorem for stochastic integrals and Theorem \ref{th: Bernstein drift} yield
	\begin{align*}
		&\left(\E\left[\left\|\overline b_{j,h_1,h_2,T}-\mu(K_{h_1,h_2}(z-\cdot)b^j)\right\|^p_{L^\infty(D)}\right]\right)^{1\slash p}\\
		&\hspace*{3em}=(h_1h_2)^{-d}T^{-1/2}\left(\E\left[\sup_{g\in\GG}\left|\H^j_t(K((z-\cdot)/(h_1h_2)))-\sqrt{T}\mu(K((z-\cdot)/(h_1h_2))b^j)\right|^p\right]\right)^{1\slash p}\\
		&\hspace*{3em}\le c(h_1h_2)^{-d}T^{-1/2}\Big(\int_0^\infty\log\mathcal N\big(u,\mathcal Gb^j,\tfrac{m_T}{\sqrt{T}}d_\infty\big)\d u+\int_0^\infty\sqrt{\log \mathcal N(u,\mathcal Gb^j,d_{\G,\tau})}\diff u\\
		&\hspace*{7em}\textcolor{dex}{+\int_0^\infty\log\mathcal N\big(u,\mathcal G,T^{-1/4}d_\infty+T^{-1/8}d_{L^4(\mu)}\big)\d u+ \int_0^\infty\sqrt{\log \mathcal N(u,\mathcal G,d_{L^2(\mu)}+T^{-1/8}d_{L^4(\mu)})}\diff u}\\
		&\hspace*{7em}+\sup_{g\in\mathcal G}\Big(\frac{m_T}{\sqrt{T}}\|g\|_\infty p+ \lVert g\rVert_{\mathbb{G},\tau}\sqrt p +\frac{1}{2}\lVert g \rVert_{\infty} \sqrt{T} \mathrm{e}^{-\frac{\kappa m_T}{p}}+p\sqrt{\frac{\tilde m_T\Vert a_{jj}\Vert_\infty}{T}}\Vert g\Vert_\infty \\
		&\hspace*{7em}
		\textcolor{dex}{+p^{3/4} (\tilde{\tau}/T)^{1/4}\Vert g\Vert_{L^4(\mu)}} +\sqrt{p\Vert a_{jj}\Vert_\infty} \Vert g\Vert_\infty \e^{-\tfrac{\kappa \tilde m_T}{2p}}+\sqrt{p\Vert a_{jj}\Vert_\infty}\Vert g\Vert_{L^2(\mu)}\Big)\Big).
	\end{align*}
	We continue by bounding the entropy integrals.
	Elementary calculations and Lemma \ref{lemma: covering numbers sup} yield	
	\begin{align*}
		\int_0^\infty\log\mathcal N\big(u,\mathcal Gb^j,\tfrac{m_T}{\sqrt{T}}d_\infty\big)\d u&=\frac{m_T}{\sqrt{T}}\int_0^{2\sup_{x \in \mathcal K}\vert b^j(x)\vert\Vert K\Vert_\infty}\log\mathcal N\big(u,\mathcal Gb^j,d_\infty\big)\d u
		\\
		&\leq c\frac{m_T}{\sqrt{T}}\int_0^{2\sup_{x \in \mathcal K}\vert b^j(x)\vert\Vert K\Vert_\infty}\log\Big(\frac{c(h_1^{-1}+h_2^{-1})}{u }\Big)\d u 
		\\
		&\leq c\frac{m_T}{\sqrt{T}}\Big(1+\log\Big(h_1^{-1}+h_2^{-1}\Big)\Big).
	\end{align*}
	\textcolor{dex}{
		Analogously, we get
		\begin{align*}
			&\int_0^\infty\log\mathcal N\big(u,\mathcal G,T^{-1/4}d_\infty+T^{-1/8}d_{L^4(\mu)}\big)\d u\\
			&\hspace*{3em}\le\int_0^\infty\log\mathcal N\big(u,\mathcal G,T^{-1/8}d_{L^4(\mu)}\big)\d u+\int_0^\infty\log\mathcal N\big(u,\mathcal G,T^{-1/4}d_\infty\big)\d u
			\\
			&\hspace*{3em}\le\int_0^\infty\log\mathcal N\big( T^{1/8}(ch_1h_2)^{-d/4}u,\mathcal G,d_\infty\big)\d u+c T^{-1/4}\Big(1+\log\Big(h_1^{-1}+h_2^{-1}\Big)\Big)
			\\
			&\hspace*{3em}\le c\Big(1+\log\Big(h_1^{-1}+h_2^{-1}\Big)\Big)\left(T^{-1/4}+T^{-1/8}(h_1h_2)^{d/4}\right).
		\end{align*}
	}
	Furthermore, Proposition \ref{prop: variance not bounded} yields for $f,g\in\mathcal G$ and large enough $T$
	\[
	d_{\G,t}(fb^j,gb^j) \leq c (h_1h_2)^d \psi_{d}(h_1,h_2)\eqqcolon \mathbb V,
	\] 
	and, hence, using \eqref{eq:upperest}
	we get by Lemma \ref{lemma: covering numbers sup} for large enough $T$
	\begin{align*}
		\int_0^\infty\sqrt{\log \mathcal N(u,\mathcal Gb^j,d_{\G,\tau})}\diff u&\leq \int_0^{\mathbb V}\sqrt{\log \mathcal N(u\psi_d(h_1,h_2)^{-1},\mathcal Gb^j,d_{\infty})}\diff u
		\\
		&\leq \int_0^{\mathbb V}\sqrt{\log\left(\frac{c(h_1^{-1}+h_2^{-1})(h_3h_4)^d\psi_d(h_1,h_2)}{u} \right)}\diff u
		\\
		&\leq c(h_1h_2)^d\psi_{d}(h_1,h_2) \sqrt{\log \left(h_1^{-1}+h_2^{-1}\right)}.
	\end{align*}
	For the remaining integral, we argue similarly and get for large enough $T$
	\textcolor{dex}{
		\begin{align*}
			\int_0^\infty\sqrt{\log \mathcal N(u,\mathcal G,d_{L^2(\mu)}+T^{-1/8}d_{L^4(\mu)})}\diff u
			&\leq\int_0^{c(h_1h_2)^{d/2}\Vert K\Vert_\infty}\sqrt{\log \mathcal N(c(h_1h_2)^{-d/2}u,\mathcal G,d_{\infty})}\diff u\\
			&\qquad+\int_0^{c(h_1h_2)^{d/4}T^{-1/8}\Vert K\Vert_\infty}\sqrt{\log \mathcal N(c(h_1h_2)^{-d/4}T^{1/8}u,\mathcal G,d_{\infty})}\diff u
			\\
			&\leq\int_0^{c(h_1h_2)^{d/2}\Vert K\Vert_\infty}\sqrt{\log\left(\frac{c(h_1^{-1}+h_2^{-1})(h_1h_2)^{d/2}}{u} \right)}\diff u\\
			&\qquad+\int_0^{c(h_1h_2)^{d/4}T^{-1/8}\Vert K\Vert_\infty}\sqrt{\log\left(\frac{c(h_1^{-1}+h_2^{-1})(h_1h_2)^{d/4}}{T^{1/8}u} \right)}\diff u
			\\
			&\leq c \sqrt{\log \left(h_1^{-1}+h_2^{-1}\right)} \left((h_1h_2)^{d/2}+(h_1h_2)^{d/4}T^{-1/8} \right).
		\end{align*}
		Combining everything above and choosing $m_T=\tilde m_T=\frac{p}{\kappa}\log T$, we obtain for $T$ large enough and $p\leq \gamma\log T$, with $\gamma>0$,
		\begin{align}\nonumber
			&\left(\E\left[\left\|\overline b_{j,h_1,h_2,T}-\mu(K_{h_1,h_2}(z-\cdot)b^j)\right\|^p_{L^\infty(D)}\right]\right)^{1\slash p}\\\nonumber
			&\hspace*{3em}\le c(h_1h_2)^{-d}T^{-1/2}\Big(\frac{p\log \left(h_1^{-1}+h_2^{-1}\right)\log T}{\sqrt{T}}+\notag (h_1h_2)^d\psi_{d}(h_1,h_2)\sqrt{\log \left(h_1^{-1}+h_2^{-1}\right)}\\\nonumber
			&\hspace*{7em} +T^{-1/8}\left(T^{-1/8}+(h_1h_2)^{d/4}\right)\log \left(h_1^{-1}+h_2^{-1}\right)+ (h_1h_2)^{d/2} \sqrt{\log \left(h_1^{-1}+h_2^{-1}\right)} +\frac{p^2\log T}{\sqrt{T}}\\\nonumber
			&\hspace*{7em}+\sqrt{p}(h_1h_2)^d\psi_{d}(h_1,h_2) +T^{-1/2}+\sqrt{\frac{p^3\log T}{T}}+p \left(\frac{\log T}{T}\right)^{1/4} (h_1h_2)^{d/4} \\\nonumber&\hspace*{7em} +\sqrt{p}T^{-1/2}+\sqrt{p}(h_1h_2)^{d/2}\Big)\\
			&\hspace*{3em}\le c(h_1h_2)^{-d}T^{-1/2}\Big(\frac{\log T^3}{\sqrt{T}}+\notag (h_1h_2)^{d/2}\sqrt{\log \left(h_1^{-1}+h_2^{-1}\right)}\\\nonumber
			&\hspace*{7em} +T^{-1/8}\left(T^{-1/8}+(h_1h_2)^{d/4}\right)\log \left(h_1^{-1}+h_2^{-1}\right)+ (h_1h_2)^{d/2} \sqrt{\log \left(h_1^{-1}+h_2^{-1}\right)} +\frac{\log T^3}{\sqrt{T}}\\\nonumber
			&\hspace*{7em}+\sqrt{\log T}(h_1h_2)^d\psi_{d}(h_1,h_2) +T^{-1/2}+\sqrt{\frac{\log T^4}{T}}+\log T \left(\frac{\log T}{T}\right)^{1/4} (h_1h_2)^{d/4} \\\nonumber&\hspace*{7em} +\sqrt{\frac{\log T}{T}}+\sqrt{\log T}(h_1h_2)^{d/2}\Big)\\\label{eq: estimator moment bound}
			&\hspace*{3em}\le c_\gamma(h_1h_2)^{-d/2}T^{-1/2}\sqrt{\log \left(h_1^{-1}+h_2^{-1}\right)}, 
		\end{align}
		where we used that $h_1,h_2\in\mathcal{H}$ and $(h_1h_2)^d\geq T^{-1/2}\log(h_1^{-1}+h_2^{-1})$, and where the constant $c_\gamma$ depends on $\gamma$.
	}
\end{proof}

\begin{proof}[Proof of Theorem \ref{thm: rate drift}]
	Introduce the set 
	$B_T\coloneqq \{\Vert\hat{\rho}_{h_1,h_2,T}(z)- \rho(z)\Vert_{L^\infty(D)} \leq r_T \}$.
	Markov's inequality and \textcolor{dex}{Theorem \ref{cor: inv dens 1}} then imply, for large enough $T$ and some constant $c$ which is independent of $1\leq p\le c_p \sqrt{\log T}$,
	\[
	\Pro(B_T\c)\leq c (\Psi\chi_\cB)^p(\beta_1,\beta_2,d,0) r_T^{-p}=c\exp\left(-p\sqrt{\log T}\right).
	\] 
	\textcolor{dex}{
		Note furthermore that $\bar{\beta}>d$ implies $(h_1h_2)^d\geq T^{-1/2}\log(h_1^{-1}+h_2^{-1})$, for large enough $T$. }
	Thus, for large enough $T$, it holds on the event $B_T\c$
	\begin{align*}
		&\E\left[\sup_{z \in D}\vert (\hat{b}_{j,\h,T,r_T}(z)-b(z))\rho(z)\vert\1_{B_T\c}\right]\\
		&\hspace*{3em}\le\E\left[\sup_{z \in D}\vert \hat{b}_{j,\h,T,r_T}(z)\rho(z)\vert\1_{B_T\c}\right]+\E\left[\sup_{z \in D}\vert b(z)\rho(z)\vert\1_{B_T\c}\right]\\
		&\hspace*{3em}\le\E\left[\sup_{z \in D}\vert\hat{b}_{j,\h,T,r_T}(z)\rho(z)\vert^2\right]^{\frac{1}{2}}\Pro(B_T\c)^{\frac 1 2}+c\Pro(B_T\c)\\
		&\hspace*{3em}\le c\left( r_T^{-1} \left(\left(\frac{\log T}{T}\right)^{\frac{\overline{\beta}}{2(\overline{\beta}+d)}}+1 \right)\exp\left(-(p/2)\sqrt{\log T}\right)+\exp\left(-p\sqrt{\log T}\right)\right)\\
		&\hspace*{3em}\le c\left( \sqrt{T}\exp\left(-(p/2+1)\sqrt{\log T}\right)+\exp\left(-p\sqrt{\log T}\right)\right),
	\end{align*}
	where we used equation \eqref{eq: estimator moment bound} and the Minkowski inequality in the second to last line. 
	Choosing $p=5\sqrt{\log T}$ now gives 
	\[
	\E\left[\sup_{z \in D}\vert (\hat{b}_{j,\h,T,r_T}(z)-b(z))\rho(z)\vert\1_{B_T\c}\right]\in\cO(T^{-2})\subset\mathcal{O}\left(\left(\frac{\log T}{T}\right)^{\frac{\overline{\beta}}{2(\overline{\beta}+d)}} \right).
	\]
	On the other hand, on the event $B_T$ it holds $\rho/(\hat{\rho}_{h^{(\rho)}_1,h^{(\rho)}_2,T}+r_T)\leq 1$.
	Thus, by Theorem \ref{cor: inv dens 1} and Proposition \ref{theorem: drift},
	\begin{align*}
		&\E\left[\sup_{z \in D}\vert (\hat{b}_{j,\h,T,r_T}(z)-b(z))\rho(z)\vert\1_{B_T}\right]\\
		&\hspace*{3em}\le\E\left[\sup_{z \in D}\vert \hat{b}_{j,\h,T,r_T}(z)-\tfrac{b(z)\rho(z)}{\hat\rho_{h_1^{(\rho)},h_2^{(\rho)},T}(z)+r_T}))\rho(z)\vert\1_{B_T}\right]+\E\left[\sup_{z \in D}\vert (\tfrac{b(z)\rho(z)}{\hat\rho_{h_1^{(\rho)},h_2^{(\rho)},T}+r_T}-b(z))\rho(z)\vert\1_{B_T}\right]\\
		&\hspace*{3em}\le c\left(\left(\frac{\log T}{T} \right)^{\frac{\overline{\beta}}{2(\overline{\beta}+d)}} +\E[\sup_{z \in D}\vert \rho(z)-\hat\rho_{h^{(\rho)}_1,h^{(\rho)}_2,T}-r_T\vert\1_{B_T}]\right)\\
		&\hspace*{3em}\le c\left(\left(\frac{\log T}{T} \right)^{\frac{\overline{\beta}}{2(\overline{\beta}+d)}} +r_T+(\Psi\chi_\cB)(T,\beta_1,\beta_2,d,0)\right),
	\end{align*}
	where we used the bias bound \eqref{eq: aniso bias}. 
	%Furthermore, we have
	% \begin{align*}
		% &\left(\E\left[\left\|\overline b_{j,h_3,h_4,T}-\mu(K_{h_3,h_4}(z-\cdot)b^j)\right\|^p_{L^\infty(D)}\right]\right)^{1\slash p}\\
		% &\hspace*{3em}=(h_3h_4)^{-d}T^{-1/2}\left(\E\left[\left\|\H^j_t(K((z-\cdot)/(h_3h_4)))-\sqrt{T}\mu(K((z-\cdot)/(h_3h_4))b^j)\right\|^p_{L^\infty(D)}\right]\right)^{1\slash p}.
		% \end{align*}
	The assertion now follows since $\Upsilon>0$ (recall \eqref{def:Ups}, \eqref{def:Psi}) implies \[(\Psi\chi_\cB)(T,\beta_1,\beta_2,d,0) + r_t\in \mathcal{O}\left(\left(\frac{\log T}{T}\right)^{\frac{\overline{\beta}}{2(\overline{\beta}+d)}} \right).\]
\end{proof}

\begin{proof}[Proof of Proposition \ref{thm:adapdrift}]
	Fix $j\in\{1,\ldots,d\}$.
	In what follows, the dependencies on $j$ and $q$ will be regularly suppressed in the notation.
	We start with stating an important auxiliary result.
	
	\begin{lemma}\label{lemma: concentration drift}
		Let $\mathcal G_{\h}\coloneqq \left\{K_1((x-\cdot)/h_1)K_2((y-\cdot)/h_2): (x,y)\in D\cap \mathbb{Q}^{2d}\right\}$, $\h=(h_1,h_2)\in\mathcal H_t$, and recall the definition of $\mathbb H^j_t$ (see \eqref{eq:mathbbH}).
		Then, for any $\gamma>0$ and large enough $t$, it holds
		\[
		\forall u_t\in [1,\gamma\log(t)],\quad \Pro\left(\sup_{g\in\mathcal G_{\h}}|\H^j_t(g)-\sqrt t\mu(gb^j)|>\Delta_{\h,t}(u_t)\right)\leq \e^{-u_t},
		\]
		where 
		\begin{equation}\label{def:Delta}
			\Delta_{\h,t}(u)\coloneqq4\e\sqrt{\Vert \rho\Vert_\infty\Vert a_{jj}\Vert_\infty(h_1h_2)^{d}}\left(\tilde C_1 \sqrt{384d\log( h_1^{-1}+h_2^{-1})}\Vert K\Vert_\infty
			+\tilde C_2\Vert K\Vert_{L^2(\lebesgue)}u^{1/2} \right).
		\end{equation}
	\end{lemma}
	\begin{proof}
		To prove the assertion, we want to combine Markov's inequality with the uniform moment bounds derived in Theorem \ref{th: Bernstein drift}. 
		Choosing $p=p(t)=u_t\leq \gamma \log t$ for fixed $\gamma>0$, we get as in the derivation of equation \eqref{eq: estimator moment bound} that there exist $c_1,c_2>0$ such that, for large enough $t$, 
		\[
		\E\left[\sup_{g\in\mathcal G_{\h}}|\H^j_t(g)-\sqrt t\mu(gb^j)|^p\right]^{1/p}
		\le 2(h_1h_2)^{d/2}\left(c_1\sqrt{\log(h_1^{-1}+h_2^{-1})}+c_2\sqrt{u_t}\right),
		\]
		where the constants $c_1,c_2$ need to satisfy
		\begin{align*}
			\tilde C_1\int_0^\infty\sqrt{6\Vert a_{jj}\Vert_\infty\log \mathcal N(u,\mathcal G_{\h},d_{L^2(\mu)})}\diff u &\leq c_1(h_1h_2)^{d/2}\sqrt{\log( h_1^{-1}+h_2^{-1})},
			\\
			2 \tilde C_2 \sqrt{\Vert a_{jj}\Vert_\infty}\Vert g\Vert_{L^2(\mu)}&\leq c_2(h_1h_2)^{d/2},
		\end{align*}
		for large enough $t$. 
		Here, $\tilde C_1, \tilde C_2>0$ correspond to the constants obtained in the proof of Theorem \ref{th: Bernstein drift}. Now
		\[ \Vert g\Vert_{L^2(\mu)}\leq (h_1h_2)^{d/2} \Vert \rho\Vert_\infty^{1/2} \Vert K\Vert_{L^2(\lebesgue)}\]
		shows that $c_2=2\tilde C_2 \sqrt{\Vert a_{jj}\Vert_\infty \Vert \rho\Vert_\infty}\Vert K\Vert_{L^2(\lebesgue)}$ is an adequate choice. 
		Additionally, straightforward computations using \eqref{eq:upperest} and Lemma \ref{lemma: covering numbers sup} show that 
		$c_1=\tilde C_1\sqrt{1536d\Vert a_{jj}\Vert_\infty\Vert \rho\Vert_\infty }\Vert K\Vert_\infty$ also satisfies the given requirement. 
		Hence, defining $\Delta_{\h,t}$ as in \eqref{def:Delta} implies the assertion through Markov's inequality.
	\end{proof}
	
	For any $\h=(h_1,h_2)^\top,\Beta=(\eta_1,\eta_2)^\top\in(0,1]^2$, set
	\begin{align*}
		s_{\h}(\cdot,\cdot)=s_{h_1,h_2}(\cdot,\cdot)&\coloneqq \int_{\R^{2d}}K_{h_1,h_2}(u-\cdot,v-\cdot)(b^j\rho)(u,v)\d u\d v,\\
		s_{\h,\Beta}^\star(\cdot,\cdot)=s_{h_1,h_2,\eta_1,\eta_2}^\star(\cdot,\cdot)&\coloneqq\int_{\R^{2d}}\left(K_{h_1,h_2}\star K_{\eta_1,\eta_2}\right)(u-\cdot,v-\cdot)(b^j\rho)(u,v)\d u\d v.
	\end{align*}
	For any kernel estimator 
	\[\overline b_{\h}(x,y)=\overline b_{j,\h}(x,y)\equiv \overline b_{j,h_1,h_2,t}(x,y)=\frac1t\int_0^tK_{h_1,h_2}(x-X_u,y-Y_u)\d Y_u^j\] of $b^j\rho$, denote its stochastic error by 
	$\xi_{\h}(\cdot,\cdot)\coloneqq \overline b_{\h}(\cdot,\cdot)-s_{\h}(\cdot,\cdot)$, and set
	\[
	\zeta_t\coloneqq \sup_{(\eta_1,\eta_2)\in\mathcal H_t}\left\{\left[\|\xi_{\eta_1,\eta_2}\|_\infty-A_t(\eta_1,\eta_2)\right]_+\right\},
	\] where $A_t(\cdot,\cdot)$ is defined as in \eqref{def:At}.
	The triangle inequality implies that, for any $\h\in \mathcal H_t$,
	\[
	\|\overline b_{\hat{\h}}-b^j\rho\|_\infty
	\le\|\overline b_{\hat{\h}}-\overline b_{\h,\hat{\h}}\|_\infty+\|\overline b_{\h,\hat{\h}}-\overline b_{\h}\|_\infty+\|\overline b_{\h}-b^j\rho\|_\infty.
	\]
	Since $\hat{\h}\in \mathcal H_t$, we have
	\[
	\|\overline b_{\hat{\h}}-\overline b_{\h,\hat{\h}}\|_\infty\le\sup_{\Beta\in\mathcal H_t}\left\{\left[\|\overline b_{\Beta}-\overline b_{\h,\Beta}\|_\infty-A_t(\Beta)\right]_+\right\}+A_t(\hat{\h})=\hat\Delta_t(\h)+A_t(\hat{\h}),
	\]
	and, since $\overline b_{\h,\hat{\h}}=\overline b_{\hat{\h},\h}$,
	\begin{align*}
		\|\overline b_{\hat{\h}}-b^j\rho\|_\infty
		&\le \hat\Delta_t(\h)+A_t(\hat{\h})+\hat\Delta_t(\hat{\h})+A_t(\h)+\|\overline b_{\h}-b^j\rho\|_\infty\\
		&\le 2\left(\hat\Delta_t(\h)+A_t(\h)\right)+\|\overline b_{\h}-b^j\rho\|_\infty.
	\end{align*}
	In view of
	\[
	\|\overline b_{\h}-b^j\rho\|_\infty\le\|\xi_{\h}\|_\infty+\mathcal B_{b^j\rho}(\h)\le\zeta_t+\mathcal B_{b^j\rho}(\h)+A_t(\h),
	\]
	it remains to bound $\hat\Delta_t(\h)+A_t(\h)$.
	For doing so, note first that, for any $\h,\Beta\in(0,1]^2$,
	\begin{align*}
		\|\overline b_{\h,\Beta}-s^\ast_{\h,\Beta}\|_\infty&=\sup_{(x,y)\in\R^{2d}}\left|\int_{\R^{2d}}K_{\eta_1,\eta_2}(x-u,y-v)\xi_{h_1,h_2}(u,v)\d u\d v\right|\le \ka_1\|\xi_{\h}\|_\infty,\\
		\|s^\star_{\h,\Beta}-s_{\Beta}\|_\infty&\le \ka_1\mathcal B_{b^j\rho}(\h).
	\end{align*}
	Thus,
	\begin{align*}
		\|\overline b_{\h,\Beta}-\overline b_{\Beta}\|_\infty
		&\le\|\overline b_{\h,\Beta}-s^\star_{\h,\Beta}\|_\infty+\|s^\star_{\h,\Beta}-s_{\Beta}\|_\infty+\|s_{\Beta}-\overline b_{\Beta}\|_\infty\\
		&\le\ka_1\left(\|\xi_{\h}\|_\infty+\mathcal B_{b^j\rho}(\h)\right)+\|\xi_{\Beta}\|_\infty\ \le\ \ka_1\left(\zeta_t+A_t(\h)+\mathcal B_{b^j\rho}(\h)\right)+\zeta_t+A_t(\Beta),
	\end{align*}
	and
	\begin{align*}
		\hat\Delta_t(\h)=\sup_{\Beta\in\mathcal H_t}\left\{\left[\|\overline b_{\h,\Beta}-\overline b_{\Beta}\|_\infty-A_t(\Beta)\right]_+\right\}
		&\le\sup_{\Beta\in\mathcal H_t}\left\{\ka_1\left(\zeta_t+A_t(\h)+\mathcal B_{b^j\rho}(\h)\right)+\zeta_t\right\}\\
		&\le\ (1\vee \ka_1)\left(2\zeta_t+A_t(\h)+\mathcal B_{b^j\rho}(\h)\right),
	\end{align*}
	giving
	\[
	\hat\Delta_t(\h)+A_t(\h)\le (1\vee\ka_1)\left(2\zeta_t+2A_t(\h)+\mathcal B_{b^j\rho}(\h)\right).
	\]
	Consequently,
	\begin{align*}
		\|\overline b_{\hat{\h}}-b^j\rho\|_\infty&\le 2\left(\hat\Delta_t(\h)+A_t(\h)\right)+\|\overline b_{\h}-b^j\rho\|_\infty\\
		&\le 2(1\vee\ka_1)\left(2\zeta_t+2A_t(\h)+\mathcal B_{b^j\rho}(\h)\right)+\zeta_t+\mathcal B_{b^j\rho}(\h)+A_t(\h)\\
		&\le (1\vee\ka_1)\left(5\zeta_t+5A_t(\h)+3\mathcal B_{b^j\rho}(\h)\right),
	\end{align*}
	and, for any $\h\in \mathcal H_t$,
	\[
	\E\left[\|\overline b_{\hat{\h}}-b^j\rho\|_\infty^p\right]^{1/p}\le(1\vee\ka_1)\left(5A_t(\h)+3\mathcal B_{b^j\rho}(\h)\right)+(1\vee\ka_1)5\left(\E[\zeta_t^p]\right)^{1/p}.
	\]
	It remains to bound $\E[\zeta_t^p]$. 
	We start by writing
	\[
	\E[\zeta_t^q]=\E\left[\sup_{(\eta_1,\eta_2)\in\mathcal H_t}\left\{\left[\|\xi_{\eta_1,\eta_2}\|_\infty-A_t(\eta_1,\eta_2)\right]_+^q\right\}\right]\le
	\sum_{\Beta\in\mathcal H_t}\E\left[\left[\|\xi_{\Beta}\|_\infty-A_t(\Beta)\right]_+^q\right].
	\]
	Now Lemma \ref{lemma: concentration drift} implies for large enough $t$
	% \footnote{Hiernach m\"usste $c_A= 8\e\sqrt{d\Vert \rho\Vert_\infty}\left(\tilde C_1 \sqrt{32\lebesgue(\S)}\Vert K\Vert_\infty+\tilde C_2\sqrt{\Vert a_{jj}\Vert_\infty}\Vert K\Vert_{L^2(\lebesgue)} \right)$ gelten.}
	\[
	\Pro(\Vert \xi_{\Beta}\Vert_\infty> A_t(\Beta))=\Pro\left(\sup_{g\in\mathcal G_{\Beta}}|\H^j_t(g)-\sqrt t\mu(gb^j)|\geq \Delta_{\Beta,t}(2dq\log(\eta_1^{-1}+\eta_2^{-1}))\right)\le\left(\frac{1}{\eta_1^{-1}+\eta_2^{-1}}\right)^{2dq}.
	\]
	Then, H\"older's inequality and equation \eqref{eq: estimator moment bound} entail for large enough $t$
	\begin{align*}
		\E\left[\left[\|\xi_{\Beta}\|_\infty-A_t(\Beta)\right]_+^q\right]&\leq \E\left[(\|\xi_{\Beta}\|_\infty-A_t(\Beta))^{2q}\right]^{1/2} \Pro(\Vert \xi_{\Beta}\Vert_\infty\geq A_t(\Beta))^{1/2}
		\\
		&\leq c (\eta_1\eta_2)^{-dq/2}t^{-q/2}\log((\eta_1^{-1}+\eta_2^{-1}))^{q/2}\left(\frac{1}{\eta_1^{-1}+\eta_2^{-1}}\right)^{dq}
		\\
		&=c\left(\sqrt{\frac{\eta_1}{\eta_2}}+\sqrt{\frac{\eta_2}{\eta_1}}\right)^{-dq}t^{-q/2}\log((\eta_1^{-1}+\eta_2^{-1}))^{q/2}\in\cO(((\log t)t^{-1})^{q/2}).
	\end{align*}
	Finally, $\E[\zeta_t^p]^{1/p}\leq \E[\zeta_t^q]^{1/q}\lesssim \card(\mathcal H_t)^{1/q}\left(\frac{\log t}{t}\right)^{1/2}\lesssim (\log t)^{2/q+1/2}t^{-1/2}$.
\end{proof}

\begin{proof}[Proof of Theorem \ref{thm:adapdrift2}]
	Fix $j\in\{1,\ldots,d\}$.
	For the proof of \eqref{upperbounddrivec}, note first that, for any $\h\in\overline{\mathcal H}_t$,
	\begin{align*}
		\big|\hat b_{j,\h,t}-b^j\big|
		&= \bigg|\frac{\left(\overline b_{j,\h,t}-b^j\rho\right)+b^j\left(\rho-\hat\rho_{\h,t}\vee\rho_\star\right)}{\hat\rho_{\h,t}\vee\rho_\star}\bigg|\\
		&\le \frac{1\lor \big|b^j\big|}{\rho_\star}\left(\big|\overline b_{j,\h,t}-b^j\rho\big|+\big|\rho-\frac 1 2\left( \hat\rho_{\h,t}+\rho_\star +\big|\hat{\rho}_{\h,t}-\rho_\star\big|\right)\big|\right)
		\\
		&\le \frac{1\lor \big|b^j\big|}{\rho_\star}\left(\big|\overline b_{j,\h,t}-b^j\rho\big|+\frac1 2 \left(\big|\rho-\hat\rho_{\h,t}\big|+\big| \rho -\rho_\star
		-\big|\hat{\rho}_{\h,t}-\rho_\star\big|\big|\right)\right)
		\\
		&= \frac{1\lor \big|b^j\big|}{\rho_\star}\left(\big|\overline b_{j,\h,t}-b^j\rho\big|+\frac1 2 \left(\big|\rho-\hat\rho_{\h,t}\big|+\big| \big|\rho -\rho_\star\big|
		-\big|\hat{\rho}_{\h,t}-\rho_\star\big|\big|\right)\right)
		\\
		&\le \frac{1\lor \big|b^j\big|}{\rho_\star}\left(\big|\overline b_{j,\h,t}-b^j\rho\big|+ \big|\rho-\hat\rho_{\h,t}\big|\right).
	\end{align*}
	Thus,
	\[
	\E\left[\|\hat b_{j,\h,t}-b^j\|_{L^\infty(D)}^p\right]^{1/p}
	\le\frac{1\vee\sup_{(x,y)\in D}|b^j(x,y)|}{\rho_\star}\left(\E\left[\|\overline b_{j,\h,t}-b^j\rho\|_{L^\infty(D)}^p\right]^{1/p}
	+\E\left[\|\hat\rho_{\h,t}-\rho\|_{L^\infty(D)}^p\right]^{1/p}\right).
	\]
	Letting $\Phi_{d,\beta}(t)\coloneqq(\log t/t)^{\frac{\overline\beta}{2(\overline\beta+d)}}$, it remains to verify
	$\mathcal R_\infty^{(p)}(\overline b_{j,\hat{\h},t},b^j\rho;D)\vee\mathcal R_\infty^{(p)}(\hat \rho_{\hat{\h},t},\rho;D)\in\mathcal O(\Phi_{d,\beta}(t))$.
	The first term is bounded by means of Proposition \ref{thm:adapdrift}.
	The smoothness assumption on $b^j\rho$ implies that there exists some positive constant $\mathfrak c$, depending only on $\mathfrak b$, $K$ and $d$ such that
	$\mathcal B_{b^j\rho}(\h)\le\mathfrak c(\mathcal L_1h_1^{\beta_1}+\mathcal L_2h_2^{\beta_2})$. 
	The bandwidth $\hat \h=(h_1,h_2)^\top$ is then chosen by solving 
	\[\mathcal L_j\hat h_j^{\beta_j}=\left(\hat h_1\hat h_2\right)^{-\frac{d}{2}}\sqrt{\frac{\log\left(\hat h_1^{-1}+\hat h_2^{-1}\right)}{t}}\quad\text{ such that }\quad
	\hat h_j\sim \left(\frac{\log t}{t}\right)^{\frac{\overline\beta}{2\beta_j(\overline \beta+d)}},\quad j=1,2.\]
	The obtained solution belongs to $\overline{\mathcal H}_t$, and plugging the specified bandwidths into the rhs of \eqref{upperbound:numerator}, we obtain $\mathcal R_\infty^{(p)}(\overline b_{j,\hat{\h},t},b^j\rho;D)\in\mathcal O(\Phi_{d,\beta}(t))$.
	Furthermore, since $\overline{\mathcal H}_t\subset\mathcal H(Q_1,Q_2)$, it follows from Proposition \ref{th: inv density estimation general} that
	%\[
	%\mathcal R_\infty^{(1)}(\hat \rho_{\hat{\h},t},\rho;D) \lesssim \mathcal B_\rho(\hat{\h})+\frac{\log(t)^2}{T(h_1h_2)^d}+\psi_d(h_1,h_2,t)\sqrt{\frac{\log t}{t}}.
	%%\qquad \text{ with } \psi_d(h_1,h_2)\lesssim \begin{cases} h_1\sqrt{\log t}, & d=1,\\ h_1^{d+1}h_2^{d-1}, & d\ge 2.\end{cases}.
	%\]
	%Thus,
	%% for $d\ge 2$,
	\begin{align*}
		\mathcal R_\infty^{(p)}(\hat \rho_{\hat{\h},t},\rho;D)&\lesssim \hat h_1^{\beta_1}+\hat h_2^{\beta_2}+\frac{\log(t)^2}{T(\hat h_1\hat h_2)^d}+\psi_d(\hat h_1,\hat h_2,t)\sqrt{\frac{\log t}{t}}\\
		&\lesssim \left(\frac{\log t}{t}\right)^{\frac{\overline\beta}{2(\overline \beta+d)}}+\frac{\log(t)^2}{\sqrt{T}},
	\end{align*}
	where we used $\hat h_1 \hat h_2\geq t^{-1/(2d)}$ and $\psi_d(\hat h_1, \hat h_2)\leq \psi_{2,d}(\hat h_1, \hat h_2)\leq (\hat h_1 \hat h_2)^{-d/2}$.
	%The computation for $d=1$ is similar, thus completing the proof.
\end{proof}
\printbibliography
\end{document}